\documentclass[11pt, leqno]{amsart} 
\usepackage{amssymb,amscd,amsfonts,amsbsy}
\usepackage{latexsym}
\usepackage{exscale}
\usepackage{amsmath,amsthm,amsfonts}
\usepackage{mathrsfs}
\usepackage{xcolor} 
\usepackage[colorlinks=true,linkcolor=blue,citecolor=red,urlcolor=red, backref=page]{hyperref} 
\usepackage{esint} 
\usepackage{stmaryrd}
\usepackage{pifont}

\usepackage[utf8]{inputenc}

\calclayout

\newcommand{\black}{\color{black}}

\allowdisplaybreaks

\newtheorem{theorem}{Theorem}[section]

\newtheorem{lemma}[theorem]{Lemma}

\theoremstyle{definition}
\newtheorem{definition}[theorem]{Definition}
\theoremstyle{remark}
\newtheorem{remark}[theorem]{Remark}

\numberwithin{equation}{section} 

\def\dist{\operatorname{dist}}

\def\supp{\operatorname{supp}}
\def\BMO{\operatorname{BMO}}

\def\loc{\operatorname{loc}}
\def\A{\mathcal{A}}
\def\B{\mathcal{B}}

\def\M{\mathcal{M}}
\def\N{\mathbb{N}}
\def\D{\mathcal{D}}

\def\F{\mathcal{F}}

\def\Z{\mathbb{Z}}

\def\R{\mathbb{R}}
\def\Rn{\mathbb{R}^n}
\def\Sn{\mathbb{S}^{n-1}} 
\def\S{\mathcal{S}}

\def\b{\mathbf{b}}
\def\m{\mathfrak{m}}
\def\p{\mathfrak{p}}

\def\s{\mathfrak{s}}

\DeclareMathOperator*{\esssup}{ess\,sup}
\DeclareMathOperator*{\essinf}{ess\,inf}
\renewcommand{\emptyset}{\text{\textup{\O}}}

\begin{document}

\author{Mingming Cao}
\address{Mingming Cao\\
Instituto de Ciencias Matem\'aticas CSIC-UAM-UC3M-UCM\\
Con\-se\-jo Superior de Investigaciones Cient{\'\i}ficas\\
C/ Nicol\'as Cabrera, 13-15\\
E-28049 Ma\-drid, Spain} \email{mingming.cao@icmat.es}

\author{Honghai Liu}
\address{Honghai Liu\\
School of Mathematics and Information Science\\
Henan Polytechnic University\\
Jiaozuo 454000\\
People's Republic of China} \email{hhliu@hpu.edu.cn}

\author{Zengyan Si}
\address{Zengyan Si\\
School of Mathematics and Information Science\\
Henan Polytechnic University\\
Jiaozuo 454000\\
People's Republic of China} \email{zengyan@hpu.edu.cn}

\author{K\^{o}z\^{o} Yabuta}
\address{K\^{o}z\^{o} Yabuta\\
Research Center for Mathematics and Data Science\\
Kwansei Gakuin University\\
Gakuen 2-1, Sanda 669-1337\\
Japan}\email{kyabuta3@kwansei.ac.jp}

\thanks{The first author acknowledges financial support from Spanish Ministry of Science and Innovation through the Juan de la Cierva-Formaci\'{o}n 2018 (FJC2018-038526-I), through the ``Severo Ochoa Programme for Centres of Excellence in R\&D'' (CEX2019-000904-S), and through PID2019-107914GB-I00, and from the Spanish National Research Council through the ``Ayuda extraordinaria a Centros de Excelencia Severo Ochoa'' (20205CEX001). The third author is supported  by  Natural Science Foundation of Henan (No. 202300410184). Part of this work was carried out while the first author was visiting the Hausdorff Institute for Mathematics, Bonn (Germany). The first author expresses his gratitude to this institution}

\date{\today}



\begin{abstract}
In recent years, sharp or quantitative weighted inequalities have attracted considerable attention on account of $A_2$ conjecture solved by Hyt\"{o}nen. Advances have greatly improved conceptual understanding of classical objects such as Calder\'{o}n-Zygmund operators. However, plenty of operators do not fit into the class of Calder\'{o}n-Zygmund operators and fail to be bounded on all $L^p(w)$ spaces for $p \in (1, \infty)$ and $w \in A_p$. In this paper we develop Rubio de Francia extrapolation with quantitative bounds to investigate quantitative weighted inequalities for operators beyond the (multilinear) Calder\'{o}n-Zygmund theory. We mainly establish a quantitative multilinear limited range extrapolation in terms of exponents $p_i \in (\p_i^-, \p_i^+)$ and weights $w_i^{p_i} \in A_{p_i/\p_i^-} \cap RH_{(\p_i^+/p_i)'}$, $i=1, \ldots, m$, which refines a result of Cruz-Uribe and Martell. We also present an extrapolation from multilinear operators to the corresponding commutators. Additionally, our result is quantitative and allows us to extend special quantitative estimates in the Banach space setting to the quasi-Banach space setting. Our proof is based on an off-diagonal extrapolation result with quantitative bounds. Finally, we present various applications to illustrate the utility of extrapolation by concentrating on quantitative weighted estimates for some typical multilinear operators such as bilinear Bochner-Riesz means, bilinear rough singular integrals, and multilinear Fourier multipliers. In the linear case, based on the Littlewood-Paley theory, we include weighted jump and variational inequalities for rough singular integrals. 
\end{abstract}

\title[Limited range extrapolation with quantitative bounds and applications]
{Limited range extrapolation with quantitative bounds and applications}

\maketitle
\tableofcontents

\section{Introduction}
In the last two decades, it has been of great interest to obtain sharp weighted norm inequalities for operators $T$, which concerns   estimates of the form 
\begin{align}\label{Tsharp}
\|T\|_{L^p(w) \to L^p(w)} \leq C_{n, p, T} \, [w]_{A_p}^{\alpha_p(T)}, \quad\forall p \in (1, \infty), \, w \in A_p, 
\end{align}
where the positive constant $C_{n, p, T}$ depends only on $n$, $p$, and $T$, and the exponent $\alpha_p(T)$ is optimal such that \eqref{Tsharp} holds. This kind of estimates gives the exact rate of growth of the weights norm. The first result  was given by Buckley \cite{Buc} for the Hardy-Littlewood maximal operator $M$ that 
\begin{align}\label{eq:Buckley}
\|M\|_{L^p(w) \to L^p(w)} \le C_{n, p} \, [w]_{A_p}^{\frac{1}{p-1}},\quad\forall p \in (1, \infty), \, w \in A_p, 
\end{align}
and the exponent $\frac{1}{p-1}$ is the best possible. The problem \eqref{Tsharp} for singular integrals gained new momentum from certain important applications to PDE. In the borderline case, a long-standing regularity problem for the solution of Beltrami equation on the plane was conjectured by Astala, Iwaniec, and Saksman \cite{AIS}, and first settled by Petermichl and Volberg \cite{PV} based on the sharp weighted estimate for the Ahlfors-Beurling operator $B$ with $\alpha_2(B)=1$. Then a question arose whether \eqref{Tsharp} with $\alpha_2(T)=1$ holds for the general Calder\'{o}n-Zygmund operators $T$, which is known as the $A_2$ conjecture. Focusing on the critical case $p=2$ results from a quantitative version of Rubio de Francia extrapolation due to Dragi$\check{\rm c}$evi\'{c} et al. \cite{DGPP}. 

Since then, many remarkable publications came to enrich the literature in this area. Petermichl \cite{Pet-1} applied the method of Bellman function to obtain \eqref{Tsharp} for Hilbert transform $H$ by showing $\alpha_2(H)=1$. The same estimate holds for Riesz transforms $R_j$ on $\Rn$, see \cite{Pet-2}. Later on, Lacey, Petermichl, and Reguera \cite{LPR} investigated Haar shift operators $S_{\tau}$ with parameter $\tau$ in order to present a unified approach to obtain the sharp weighted estimates for $B$, $H$, and $R_j$, by proving $\alpha_2(S_{\tau})=1$ and noting that such three kinds of operators can be obtained by appropriate averaging of Haar shifts, see \cite{DV, Pet, PTV}. By means of local mean oscillation and extrapolation with sharp constants \cite{DGPP}, Lerner \cite{Ler11} established the sharp estimates \eqref{Tsharp} for Littlewood-Paley operators $S$ with $\alpha_p(S)=\max\{\frac12, \frac{1}{p-1}\}$, and Cruz-Uribe et al. \cite{CMP12} gave an alternative and simpler proof of \eqref{Tsharp} for $B$, $H$, and $R_j$. In 2012, Hyt\"{o}nen \cite{Hyt} fully solved the $A_2$ conjecture by showing a resulting representation of an arbitrary Calder\'{o}n-Zygmund operator as an average of dyadic shifts over random dyadic systems. Significantly, it opened the study of dyadic analysis in the fields including the multilinear theory, the multiparameter theory, and the non-homogeneous theory. In particular, in terms of sharp weighted estimates, it promoted the development of sparse domination for varieties of operators. To sum up, there are three kinds of sparse domination: identities with suitable averaging, pointwise dominations, and bilinear forms. The specific type depends on the singularity of operators. For example, the Calder\'{o}n-Zygmund operator \cite{Hyt} and Riesz potential \cite{CG} can be recovered from dyadic operators by averaging over dyadic grids. The pointwise sparse dominations hold for the Calder\'{o}n-Zygmund operators \cite{Ler16} and the corresponding commutators \cite{LOR}, the multilinear Calder\'{o}n-Zygmund operators \cite{DHL}, the multilinear pseudo-differential operators \cite{CXY}, and the multilinear Littlewood-Paley operators with minimal regularity \cite{CY}. Additionally, the sparse domination with a bilinear form goes to singular non-integral operators \cite{BFP}, Bochner-Riesz multipliers \cite{BBL, LMR}, rough operators \cite{CCDO}, and oscillatory integrals \cite{LS}.

As aforementioned, one of the most useful and powerful tools in the weighted theory is the celebrated Rubio de Francia extrapolation theorem \cite{RdF}, which states that if a given operator $T$ is bounded on $L^{p_0}(w_0)$ for some $p_0 \in [1, \infty)$ and for all $w_0 \in A_{p_0}$, then $T$ is bounded on $L^p(w)$ for all $p \in (1, \infty)$ and for all $w \in A_p$.  Indeed, extrapolation theorems allow us to reduce the general weighted $L^p$ estimates for certain operators to a suitable case $p=p_0$, for example, see \cite{CXY} for the Coifman-Fefferman's inequality for $p_0 = 1$, \cite{CMP12, Hyt} for the Calder\'{o}n-Zygmund operators for $p_0 = 2$, \cite{CMP12, Ler11} for square functions for $p_0=3$, and \cite{LMPT} for fractional integral operators for $p_0 \in (1, n/\alpha)$ with $0<\alpha<n$. Even more, the technique of extrapolation can refine some weighted estimates, see \cite{CMP05} for the Sawyer conjecture, \cite{LOP, LOP1} for the weak Muckenhoupt-Wheeden conjecture, and \cite{CXY, OPR} for the local exponential decay estimates. Another interesting point is that by  means of extrapolation, the vector-valued inequalities immediately follows from the corresponding scalar-valued estimates.

Over the years, Rubio de Francia's result has been extended and complemented in different manners, see \cite{CMP11} and the references therein. Using the boundedness of the Hardy-Littlewood maximal operator instead of the Muckenhoupt weights, Cruz-Uribe and Wang \cite{CW}  presented  extrapolation in variable Lebesgue spaces, which was improved to generalized Orlicz spaces \cite{CH} and general Banach function spaces \cite{CMM}.  It is worth mentioning that the latter was stated in measure spaces and for general Muckenhoupt bases. This leads lots of applications, such as the well-posedness of the Dirichlet problem in the upper half-space whenever the boundary data belongs to different function spaces, the weighted boundedness of layer potential operators on domains, and the local $Tb$ theorem for square functions in non-homogeneous spaces. Recently, 
a longstanding problem about extrapolation for multilinear Muckenhoupt classes of weights was solved by Li, Martell, and Ombrosi \cite{LMO} by introducing some new multilinear Muckenhoupt classes $A_{\vec{p}, \vec{r}}$ (cf. Definition \ref{def:Apr}), which contains the multivariable nature and is a generalization of the classes $A_{\vec{p}}$ introduced in \cite{LOPTT} (cf. \eqref{eq:w-Ap} below). Shortly afterwards, it was improved to the case with infinite exponents in \cite{LMMOV} and with a quantitative bound in  \cite{Nie}. On the other hand, Hyt\"onen and Lappas \cite{HL-1, HL-2} established a ``compact version'' of Rubio de Francia's extrapolation theorem, which allows one to extrapolate the compactness of an operator from just one space to the full range of weighted spaces, provided that the operator is bounded. This result has been extended to the multilinear setting \cite{COY} by means of weighted interpolation for multilinear compact operators and weighted Fr\'{e}chet-Kolmogorov characterization of compactness in the non-Banach case.

Motivated by the work above, the mail goal of this paper is to establish multivariable Rubio de Francia extrapolation with quantitative bounds in order to investigate quantitative weighted inequalities for multilinear operators beyond the multilinear  Calder\'{o}n-Zygmund theory. We focus on the limited range extrapolation with exponents $p_i \in (\p_i^-, \p_i^+)$ and weights $w_i^{p_i} \in A_{p_i/\p_i^-} \cap RH_{(\p_i^+/p_i)'}$, $i=1, \ldots, m$, which is quite different from \cite{Nie} for $\vec{w}=(w_1, \ldots, w_m) \in A_{\vec{p}}$ (or general $A_{\vec{p}, \vec{r}}$).  The main reason why we study it is that plenty of operators are beyond the Calder\'{o}n-Zygmund theory so that they may not be bounded on all $L^p(w)$ spaces for $p \in (1, \infty)$ and $w \in A_p$. This is the case for operators with the strong singularity, such as Bochner-Riesz means \cite{BBL}, rough singular integrals \cite{Wat}, Riesz transforms and square functions associated with second-order elliptic operators \cite{AM-1}, operators associated with the Kato conjecture \cite{AM-3}, and singular ``non-integral" operators \cite{BFP}. As well as the classes $A_p$ are natural for the Calder\'{o}n-Zygmund operators and characterize the weighted boundedness of Hardy-Littlewood maximal operators, the classes $A_{\vec{p}}$ are also the natural ones for multilinear Calder\'on-Zygmund operators and the multilinear Hardy-Littlewood maximal operators (cf. Theorem \ref{thm:CZO}). In the multilinear setting, there are also many operators so that weighted inequalities holds for limited ranges. For multilinear Fourier multipliers, it is interesting that different forms of Sobolev regularity appear to determine whether product of scalar weights or multiple weights $A_{\vec{p}}$ could be used. Fujita and Tomita \cite{FT, FT14} proved that whenever the symbol satisfies a product type Sobolev regularity, the weighted boundedness of multilinear Fourier multipliers holds for $\vec{w} \in A_{p_1/r_1} \times \cdots \times A_{p_m/r_m}$ but does not hold for $\vec{w} \in A_{(p_1/r_1, \ldots, p_m/r_m)}$, while the latter is valid under the classical Sobolev regularity. Other examples include strongly singular bilinear Calder\'{o}n-Zygmund operators \cite[Corollary 3.2]{BCN}, bilinear differential operators associated with fractional Leibniz rules \cite[Theorem 1.1]{CN}, bilinear pseudo-differential operators with symbols in the H\"{o}rmander classes \cite[Remark 3.4]{MRS}, and so on.

In order to state our main result we need some notation. More definitions and notation are given in Section \ref{sec:pre}. Given $1 \le \p_- < \p_+ \le \infty$ and $p \in [\p_-, \p_+]$ with $p \neq \infty$, considering Lemma \ref{lem:weights}, for any $w^p \in A_{p/\p_-} \cap RH_{(\p_+/p)'}$, we define 
\begin{align}\label{def:wpApRH}
[w^p]_{A_{p/\p_-} \cap RH_{(\p_+/p)'}} 
:= 
\begin{cases}
[w^{p(\p_+/p)'}]_{A_{\tau_p}}, & p < \p_+, 
\\
\max\{[w^p]_{A_{p/\p_-}}, [w^p]_{RH_{(\p_+/p)'}}\}, & p =\p_+, 
\end{cases}
\end{align}
where $\tau_p := \big(\frac{\p_+}{p}\big)' \big(\frac{p}{\p_-} -1\big) +1$. Throughout this paper, given $p_i, q_i \in [\p_i^-, \p_i^+]$, we always denote 
\begin{align*}
\gamma_i(p_i, q_i) := 
\begin{cases}
\max\big\{1, \frac{\tau_{q_i} - 1}{\tau_{p_i} -1} \big\}, & q_i<\p_i^+, 
\\[4pt]
\frac{q_i}{\tau_{p_i} -1} \big(\frac{1}{\p_i^-} - \frac{1}{\p_i^+} \big), & q_i = \p_i^+.
\end{cases}
\end{align*}
Let $\F$ denote a family of $(m+1)$-tuples $(f, f_1, \ldots, f_m)$ of non-negative measurable functions. We would like to present an abstract methodology for extrapolation. We will see that extrapolation enables us to obtain vector-valued inequalities and weak-type estimates from extrapolation results immediately. In the current paper, we mainly apply this methodology to obtain quantitative weighted norm inequalities for plenty of operators. 

Our main result is formulated as follows. 
\begin{theorem}\label{thm:lim}
Given $m \geq 1$, let $\F$ be a family of extrapolation (m+1)-tuples. Let $1 \le \p_i^- < \p_i^+ \le \infty$ for each $i=1, \ldots, m$. Assume that for each $i=1, \ldots, m$, there exists an exponent $q_i \in (0, \infty)$ with $q_i \in [\p_i^-, \p_i^+]$ such that  for all weights $v_i^{q_i} \in A_{q_i/\p_i^-} \cap RH_{(\p_i^+/q_i)'}$, $i=1, \ldots, m$, 
\begin{equation}\label{eq:lim-1}
\|f\|_{L^q(v^q)} 
\leq \prod_{i=1}^m \Phi_i \big([v_i^{q_i}]_{A_{q_i/\p_i^-} \cap RH_{(\p_i^+/q_i)'}}\big)  
\|f_i\|_{L^{q_i}(v_i^{q_i})}, \quad (f, f_1, \ldots, f_m) \in \F,
\end{equation}
where $\frac1q = \sum_{i=1}^m \frac{1}{q_i}$, $v= \prod_{i=1}^m v_i$, and $\Phi_i : [1, \infty) \to [1, \infty)$ is an increasing function. Then for all exponents $p_i \in (\p_i^{-}, \p_i^{+})$ and all weights $w_i^{p_i} \in A_{p_i/\p_i^{-}} \cap RH_{(\p_i^{+}/p_i)'}$, $i=1, \ldots, m$,
\begin{equation}\label{eq:lim-2}
\|f\|_{L^p(w^p)} 
\leq \prod_{i=1}^m \mathfrak{C}_i \, \Phi_i 
\big(C_i \, [w_i^{p_i}]_{A_{p_i/\p_i^-} \cap RH_{(\p_i^+/p_i)'}}^{\gamma_i(p_i, q_i)}\big) 
\|f_i\|_{L^{p_i}(w_i^{p_i})}, \quad (f, f_1, \ldots, f_m) \in \F,
\end{equation}
where $\frac1p = \sum_{i=1}^m \frac{1}{p_i}$, $w=\prod_{i=1}^m w_i$, $\mathfrak{C}_i := 2^{\max\{\frac{\tau_{p_i}}{p_i}, \frac{\tau'_{p_i}}{q_i}\}}$, and $C_i$ depends only on $n$, $p_i$, $q_i$, $\p_i^-$, and $\p_i^+$. 

Moreover, for the same family of exponents and weights, and for all exponents $r_i \in (\p_i^-, \p_i^+)$,
\begin{equation}\label{eq:lim-3}
\bigg\| \Big(\sum_k |f^k|^r \Big)^{\frac1r}\bigg\|_{L^p(w^p)}
\leq \prod_{i=1}^m \mathfrak{C}'_i \, \Phi_i 
\big(C'_i \, [w_i^{p_i}]_{A_{p_i/\p_i^-} \cap RH_{(\p_i^+/p_i)'}}^{\gamma_i(p_i, r_i) \gamma_i(r_i, q_i)}\big) 
\bigg\| \Big(\sum_k |f^k_i|^{r_i} \Big)^{\frac{1}{r_i}}\bigg\|_{L^{p_i}(w_i^{p_i})},
\end{equation}
for all $\{(f^k, f^k_1, \cdots, f^k_m)\}_k \subset \F$, where $\frac{1}{r} = \sum_{i=1}^m \frac{1}{r_i}$, $\mathfrak{C}'_i := 2^{\max\{\frac{\tau_{p_i}}{p_i}, \frac{\tau'_{p_i}}{r_i}\} + \max\{\frac{\tau_{r_i}}{r_i}, \frac{\tau'_{r_i}}{q_i}\}}$, and the constant $C'_i$ depends only on $n$, $p_i$, $q_i$, $r_i$, $\p_i^-$, and $\p_i^+$. 
\end{theorem}

Some comments are in order. The limited range extrapolation, arising naturally in the study of the Riesz transforms and other operators associated to elliptic differential operators, was first established by Auscher and Martell \cite{AM-1}, which was extended to the multivariable case \cite[Theorem 1.3]{CM}. Compared with the latter, Theorem \ref{thm:lim} gives the quantitative weighted bounds, which in turn covers the multivariable extrapolation in \cite[Theorem 6.1]{Duo} and \cite[Theorem 1.1]{GM} by taking $\p_i^-=1$ and $\p_i^+=\infty$, $i=1, \ldots, m$. As a result of Theorem \ref{thm:lim} we can extend weighted estimates only valid in the Banach range to the quasi-Banach range. For example, weighted norm inequalities for the commutators of multilinear operators $T$ with $\BMO$ functions, more singular than operators $T$, were just proved in the case $p \ge 1$ \cite{BMMST} since one used the trick of so-called Cauchy integral and Minkowski's inequality. We will use Theorem \ref{thm:lim} to deal with this problem and obtain a quantitative extrapolation from operators to the corresponding commutators with full ranges (cf. Theorem \ref{thm:lim-Tb}). The sharp weighted estimate \eqref{eq:Buckley} and sharp reverse H\"{o}lder's inequality in Lemma \ref{lem:RH} enable us to get the quantitative weights norm in \eqref{eq:lim-2} and \eqref{eq:lim-Tb-2}. Concerning the proof of Theorem \ref{thm:lim}, we borrow the ideas from \cite{CM, Duo}, which essentially reduce the multilinear problem to a linear one by acting on one function at a time. Indeed, we prove a limited range, off-diagonal extrapolation theorem with sharp weighted bounds (cf. Theorem \ref{thm:lim-off}), whose proof is distinct from and much simpler than that in \cite{CM} because it only needs to define a Rubio de Francia iteration algorithm each time we consider the case $q<q_0$ or $q>q_0$. When the exponents are greater than one, we can obtain quantitative $A_p$ and off-diagonal extrapolation (cf. Theorems \ref{thm:Ap} and \ref{thm:off}) by showing a ``product-type embedding" theorem (cf. Theorems \ref{thm:Apvw} and \ref{thm:off-vw}), respectively, which is quite different from the embedding technique used in \cite[Proposition 3.18]{CMM} to get extrapolation on general weighted Banach function spaces. Moreover, based on $A_p$ extrapolation and interpolation, we present an extrapolation from weak type inequalities to strong type estimates (cf. Theorem \ref{thm:weakAp}). This allows us to obtain quantitative weighted strong estimates from weak (1, 1) type.

In order to present an extrapolation theorem for commutators, let us introduce relevant notation and some definitions. Given a function $b \in L^1_{\loc}(\Rn)$, we say that $b \in \BMO$ if
\begin{equation*}
\|b\|_{\BMO} :=\sup_{Q} \fint_{Q} |b(x)-b_Q| \, dx < \infty.
\end{equation*}
where the supremum is taken over the collection of all cubes $Q \subset \Rn$ and $b_Q :=\fint_Q b\, dx$.

Let $T$ be an operator from $X_1  \times \cdots \times X_m$ into $Y$, where $X_1, \ldots, X_m$ are some normed spaces and and $Y$ is a quasi-normed space. Given $\vec{f} := (f_1, \ldots ,f_m) \in X_1 \times \cdots \times X_m$, $\b=(b_1, \ldots, b_m)$ of measurable functions, and $k \in \N$, we define, whenever it makes sense, the $k$-th order commutator of $T$ in the $i$-th entry of $T$ as 
\begin{align*}
[T, \b]_{k e_i} (\vec{f})(x) 
:= T(f_1,\ldots, (b_i(x)-b_i)^k f_i, \ldots, f_m)(x), \quad 1 \leq i \leq m, 
\end{align*}
where $e_i$ is the basis of $\Rn$ with the $i$-th component being $1$ and other components being $0$. Then, for a multi-index $\alpha = (\alpha_1, \ldots, \alpha_m) \in \N^m$, we define
\begin{align*}
[T, \b]_{\alpha}:= [\cdots[[T, \b]_{\alpha_1 e_1}, \b]_{\alpha_2 e_2} \cdots, \b]_{\alpha_m e_m}.
\end{align*}
In particular, if $T$ is an $m$-linear operator with a kernel representation of the form 
\begin{equation*}
T(\vec{f})(x) 
:= \int_{\R^{nm}} K(x, \vec{y}) f_1(y_1) \cdots f_m(y_m) \, d\vec{y}, 
\end{equation*}
then one can write $[T, \b]_{\alpha}$ as 
\begin{equation*}
[T, \b]_{\alpha}(\vec{f})(x) 
:= \int_{\R^{nm}} \prod_{i=1}^m (b_i(x)-b_i(y_i))^{\alpha_i} K(x, \vec{y}) f_1(y_1) \cdots f_m(y_m) \, d\vec{y}. 
\end{equation*}

\begin{theorem}\label{thm:lim-Tb}
Let $T$ be an $m$-linear operator and let $1 \le \p_i^{-}<\p_i^{+} \le \infty$, $i=1, \ldots, m$, be such that $\frac{1}{\p_+} := \sum_{i=1}^m \frac{1}{\p_i^+}<1$. Assume that for each $i=1, \ldots, m$, there exists an exponent $q_i \in (0, \infty)$ with $q_i \in [\p_i^{-}, \p_i^{+}]$ such that for all weights $v_i^{q_i} \in A_{q_i/\p_i^{-}} \cap RH_{(\p_i^{+}/q_i)'}$, $i=1, \ldots, m$, we have
\begin{equation}\label{eq:lim-Tb-1}
\| T(\vec{f})\|_{L^q(v^q)} 
\leq \prod_{i=1}^m \Phi_i \big([v_i^{q_i}]_{A_{q_i/\p_i^{-}} \cap RH_{(\p_i^{+}/q_i)'}}\big) \|f_i\|_{L^{q_i}(v_i^{q_i})},
\end{equation}
where $\vec{f}=(f_1, \ldots, f_m)$, $\frac{1}{q} = \sum_{i=1}^m \frac{1}{q_i}$, $v= \prod_{i=1}^m v_i$, and $\Phi_i : [1, \infty) \to [1, \infty)$ is an increasing function. Then for all exponents $p_i \in (\p_i^{-}, \p_i^{+})$, all weights $w_i^{p_i} \in A_{p_i/\p_i^{-}} \cap RH_{(\p_i^{+}/p_i)'}$, for all functions $\b=(b_1,\ldots, b_m) \in \BMO^m$, and for each multi-index $\alpha \in \N^m$, 
\begin{equation}\label{eq:lim-Tb-2}
\|[T, \b]_{\alpha}(\vec{f})\|_{L^p(w^p)} 
\leq C_0 \prod_{i=1}^m \widetilde{\Phi}_i \big(C'_i \, [w_i^{p_i}]_{A_{p_i/\p_i^{-}} \cap RH_{(\p_i^{+}/p_i)'}}^{\gamma_i(p_i, s_i)}\big) 
\|b_i\|_{\BMO}^{\alpha_i} \|f_i\|_{L^{p_i}(w_i^{p_i})}, 
\end{equation} 
whenever $s_i \in (\p_i^-, \p_i^+)$, $i=1, \ldots, m$, satisfy $\frac1s := \sum_{i=1}^m \frac{1}{s_i} \le 1$, where $\frac{1}{p} = \sum_{i=1}^m \frac{1}{p_i}$, $w=\prod_{i=1}^m w_i$, 
$\widetilde{\Phi}_i(t) := t^{\alpha_i \max\{1, \frac{1}{\tau_{s_i}-1}\}} \Phi_i(C_i \, t^{\gamma_i(s_i, q_i)})$, $C_i$ depends only on $n$, $s_i$, $q_i$, $\p_i^-$, and $\p_i^+$, $C'_i$ depends only on $n$, $p_i$, $s_i$, $\p_i^-$, and $\p_i^+$, and $C_0$ depends only on $\alpha$, $n$, $p_i$, $q_i$, $s_i$, $\p_i^-$, and $\p_i^+$. 

Moreover, for the same family of exponents $\vec{p}$, weights $\vec{w}$, functions $\b$, multi-index $\alpha$, and for all exponents $r_i \in (\p_i^-, \p_i^+)$,
\begin{align}\label{eq:lim-Tb-3}
\bigg\| \Big(\sum_k | [T, \b]_{\alpha}(\vec{f}^k)|^r \Big)^{\frac1r} \bigg\|_{L^p(w^p)}
&\leq C  \prod_{i=1}^m \widetilde{\Phi}_i \big(C''_i \, [w_i^{p_i}]_{A_{p_i/\p_i^{-}} \cap RH_{(\p_i^{+}/p_i)'}}^{\gamma_i(p_i, r_i) \gamma_i(r_i, s_i)}\big) 
\\ \nonumber 
&\qquad\qquad \times \|b_i\|_{\BMO}^{\alpha_i}
\bigg\| \Big(\sum_k |f^k_i|^{r_i} \Big)^{\frac{1}{r_i}}\bigg\|_{L^{p_i}(w_i^{p_i})},
\end{align}
where $\vec{f}^k=(f_1^k, \ldots, f_m^k)$, $\frac{1}{r} = \sum_{i=1}^m \frac{1}{r_i}$, $C$ depends only on $\alpha$, $n$, $p_i$, $q_i$, $r_i$, $s_i$, $\p_i^-$, and $\p_i^+$, and $C''_i$ depends only on $n$, $p_i$, $r_i$, $s_i$, $\p_i^-$, and $\p_i^+$. 
\end{theorem}

\begin{remark}\label{rem:ss}
Let us see the existence of $s_i \in (\p_i^-, \p_i^+)$, $i=1, \ldots, m$, satisfying $\frac1s := \sum_{i=1}^m \frac{1}{s_i} \le 1$. Indeed, by means of Theorem \ref{thm:lim}, the estimate \eqref{eq:lim-Tb-1} can be improved to all exponents $s_i \in (\p_i^-, \p_i^+)$, $i=1, \ldots, m$. Given $s_i \in (\p_i^-, \p_i^+)$, $i=1, \ldots, m$, there holds 
\begin{align*}
\frac1s = \sum_{i=1}^m \frac{1}{s_i}
= \sum_{i=1}^m \bigg(\frac{1}{s_i} - \frac{1}{\p_i^+} \bigg) 
+ \sum_{i=1}^m \frac{1}{\p_i^+}
\to \frac{1}{\p_+} < 1, 
\quad\text{ if } s_i \to \p_i^+, \, i=1, \ldots, m.  
\end{align*}
This means that whenever $\p_+>1$, one can always choose $s_i$ (for example, sufficiently close to $\p_i^+$) such that $\frac1s \le 1$. 

To illustrate the existence, we present a special case: 
\[
\frac{1}{\p_-} - \frac{1}{\p_+} < \p_i^+ \bigg(\frac{1}{\p_i^-} - \frac{1}{\p_i^+}\bigg), \quad i=1, \ldots, m, 
\]  
where $\frac{1}{\p_{\pm}} := \sum_{i=1}^m \frac{1}{\p_i^{\pm}}$. In this scenario, picking 
\[
s_i := \p_i^- \bigg[1 + \bigg(\frac{1}{\p_-} - \frac{1}{\p_+}\bigg)\bigg], \quad i=1, \ldots, m, 
\]
we easily verify that $s_i \in (\p_i^-, \p_i^+)$ and 
\begin{equation*}
\frac1s = \sum_{i=1}^m \frac{1}{s_i}
=\frac{1}{\p_-} \bigg[\frac{1}{\p_-} + \bigg( 1- \frac{1}{\p_+}\bigg)\bigg]^{-1} 
<1, 
\end{equation*} 
provided $\p_+>1$. 
\end{remark}

\begin{remark}\label{rem:pp}
Let $T$ be an $m$-linear operator. 
If the hypotheses \eqref{eq:lim-1} and \eqref{eq:lim-Tb-1} are assumed for $T$ and all exponents $q_i \in (\p_i^-, \p_i^+)$, then 
we will get better estimates. This means the following extrapolation: Assume that for all exponents $p_i \in (\p_i^{-}, \p_i^{+})$ and all weights $w_i^{p_i} \in A_{p_i/\p_i^{-}} \cap RH_{(\p_i^{+}/p_i)'}$, $i=1, \ldots, m$, 
\begin{equation*}
\|T(\vec{f})\|_{L^p(w^p)} 
\leq \prod_{i=1}^m \Phi_i 
\big[w_i^{p_i}]_{A_{p_i/\p_i^-} \cap RH_{(\p_i^+/p_i)'}}\big) \|f_i\|_{L^{p_i}(w_i^{p_i})},
\end{equation*}
where $\frac{1}{p} = \sum_{i=1}^m \frac{1}{p_i}$ and $w=\prod_{i=1}^m w_i$. Then for all exponents $p_i, r_i \in (\p_i^{-}, \p_i^{+})$ and all weights $w_i^{p_i} \in A_{p_i/\p_i^{-}} \cap RH_{(\p_i^{+}/p_i)'}$, $i=1, \ldots, m$, 
\begin{equation*}
\bigg\| \Big(\sum_k |T(\vec{f}^k)|^r \Big)^{\frac1r}\bigg\|_{L^p(w^p)}
\leq C_0 \prod_{i=1}^m \Phi_i \big(C_i \, [w_i^{p_i}]_{A_{p_i/\p_i^-} \cap RH_{(\p_i^+/p_i)'}}^{\gamma_i(p_i, r_i)}\big) 
\bigg\| \Big(\sum_k |f^k_i|^{r_i} \Big)^{\frac{1}{r_i}}\bigg\|_{L^{p_i}(w_i^{p_i})},
\end{equation*}
where $\vec{f}^k=(f_1^k, \ldots, f_m^k)$, $\frac{1}{r} = \sum_{i=1}^m \frac{1}{r_i}$, $C_0$ and $C_i$ depend only on $n$, $p_i$, $r_i$, $\p_i^-$, and $\p_i^+$. 

Moreover, for the same family of exponents $\vec{p}$ and weights $\vec{w}$, for all functions $\b=(b_1,\ldots, b_m) \in \BMO^m$, and for each multi-index $\alpha \in \N^m$, 
\begin{align*}
&\|[T, \b]_{\alpha}(\vec{f})\|_{L^p(w^p)} 
\leq C'_0 \prod_{i=1}^m \widetilde{\Phi}_i \big(C'_i \, [w_i^{p_i}]_{A_{p_i/\p_i^{-}} \cap RH_{(\p_i^{+}/p_i)'}}^{\gamma_i(p_i, s_i)}\big) 
\|b_i\|_{\BMO}^{\alpha_i} \|f_i\|_{L^{p_i}(w_i^{p_i})}, 
\\ 
&\bigg\| \Big(\sum_k | [T, \b]_{\alpha}(\vec{f}^k)|^r \Big)^{\frac1r} \bigg\|_{L^p(w^p)}
\leq C''_0  \prod_{i=1}^m \widetilde{\Phi}_i \big(C''_i \, [w_i^{p_i}]_{A_{p_i/\p_i^{-}} \cap RH_{(\p_i^{+}/p_i)'}}^{\gamma_i(p_i, r_i) \gamma_i(r_i, s_i)}\big) 
\\ \nonumber 
&\qquad\qquad\qquad\qquad\qquad\qquad\qquad\qquad \times \|b_i\|_{\BMO}^{\alpha_i}
\bigg\| \Big(\sum_k |f^k_i|^{r_i} \Big)^{\frac{1}{r_i}}\bigg\|_{L^{p_i}(w_i^{p_i})},
\end{align*}
whenever $\frac1s := \sum_{i=1}^m \frac{1}{s_i} \le 1$ with $s_i \in (\p_i^-, \p_i^+)$, where $\widetilde{\Phi}_i(t) := t^{\alpha_i \max\{1, \frac{1}{\tau_{s_i}-1}\}} \Phi_i(C_i \, t)$, $C'_0$ depends only on $\alpha$, $n$, $p_i$, $s_i$, $\p_i^-$, and $\p_i^+$, $C'_i$ depends only on $n$, $p_i$, $s_i$, $\p_i^-$, and $\p_i^+$, and $C''_0$ depends only on $\alpha$, $n$, $p_i$, $r_i$, $s_i$, $\p_i^-$, and $\p_i^+$, and $C''_i$ depends only on $n$, $p_i$, $r_i$, $s_i$, $\p_i^-$, and $\p_i^+$. The proof is the same as that of Theorems \ref{thm:lim} and \ref{thm:lim-Tb}. Details are left to the reader. 
\end{remark}

The rest of the paper is organized as follows. In Section \ref{sec:pre}, we present some preliminaries and auxiliary results including the embedding and factorization of Muckenhoupt weights. Section \ref{sec:quant} includes quantitative weighted estimates for various operators. Section \ref{sec:proof} is devoted to showing Theorems \ref{thm:lim} and \ref{thm:lim-Tb} by means of a limited range off-diagonal extrapolation and extrapolation for commutators with Banach ranges. We also establish ``product-type embedding" theorems to deduce quantitative $A_p$ and off-diagonal extrapolation. In Section \ref{sec:app}, we include many applications of Theorems \ref{thm:lim} and \ref{thm:lim-Tb}. First, we give quantitative weighted norm inequalities for the bilinear Bochner-Riesz means of order $\delta$ and commutators, where we utilize the $A_{p_1} \times A_{p_2}$ weights when $\delta \ge n-1/2$, and the $A_{p_1/\p_1^-} \cap RH_{(\p_1^+/p_1)'} \times A_{p_2/\p_2^-} \cap RH_{(\p_2^+/p_2)'}$ weights when $0<\delta<n-1/2$. The same weights conditions are used for the bilinear rough singular integrals for $\Omega \in L^{\infty}(\Sn)$ and $L^q(\Sn)$ with $q \in (1, \infty)$, respectively. Additionally, under the minimal Sobolev regularity, we obtain the quantitative weighted bounds for the $m$-linear Fourier multipliers, the corresponding higher order commutators, and vector-valued inequalities, which only hold for product of scalar weights as mentioned before. Beyond that, after presenting quantitative weighted Littlewood-Paley theory, we establish weighted jump and variational inequalities for rough operators with $\Omega \in L^q(\Sn)$ with $q \in (1, \infty)$. The proof also needs quantitative weighted estimates for rough singular integrals $T_{\Omega}$ and rough maximal operators $M_{\Omega}$, see Section \ref{sec:quant}. They contain many applications to Harmonic Analysis since variation inequalities not only immediately yield the pointwise convergence of the family of operators without using the Banach principle, but also can be used to measure the speed of convergence. Finally, we end up Section \ref{sec:app} with Riesz transforms associated to Schr\"{o}dinger operators.

\section{Preliminaries and auxiliary results}\label{sec:pre}

A measurable function $w$ on $\Rn$ is called a weight if $0<w(x)<\infty$ for a.e.~$x \in \Rn$. For $p \in (1, \infty)$, we define the Muckenhoupt class $A_p$ as the collection of all weights $w$ on $\Rn$ satisfying
\begin{equation*}
[w]_{A_p}:=\sup_{Q} \bigg(\fint_{Q} w\, dx \bigg) \bigg(\fint_{Q}w^{1-p'}\, dx \bigg)^{p-1}<\infty,
\end{equation*}
where the supremum is taken over all cubes $Q \subset \Rn$. As for the case $p=1$, we say that $w\in A_1$ if
\begin{equation*}
[w]_{A_1} :=\sup_{Q} \bigg(\fint_Q w\, dx\bigg) \esssup_Q w^{-1}<\infty.
\end{equation*}
Then, we define $A_{\infty} :=\bigcup_{p\geq 1}A_p$ and $[w]_{A_{\infty}}=\inf_{p>1} [w]_{A_p}$.

Given $1 \le p \le \infty$ and $0<q \le \infty$, we say that $w \in A_{p,q}$ if it satisfies
\begin{align*}
[w]_{A_{p,q}} 
:= \sup_{Q} \bigg(\fint_{Q} w^q \, dx \bigg)^{\frac1q} \bigg(\fint_{Q} w^{-p'} dx\bigg)^{\frac{1}{p'}} < \infty, 
\end{align*}
where one has to replace the first term by $\esssup_Q w$ when $q=\infty$ and the second term by $\esssup_Q w^{-1}$ when $p=1$. One can easily check that $w \in A_{p, q}$ if and only if $w^q \in A_{1+q/p'}$ if and only if $w^{-p'} \in A_{1+p'/q}$ with 
\begin{align*}
[w]_ {A_{p, q}}
=[w^q]_{A_{1+q/p'}}^{\frac1q} 
=[w^{-p'}]_{A_{1+p'/q}}^{\frac{1}{p'}}, 
\quad\text{when }\, 
1<p \le \infty, \, 0<q<\infty. 
\end{align*}
If $p=1$ and $0<q<\infty$, then $w \in A_{p, q}$ if and only if $w^q \in A_1$ with $[w]_{A_{p, q}} = [w^q]_{A_1}^{\frac1q}$. If $1<p \le \infty$ and $q=\infty$, $w \in A_{p, q}$ if and only if $w^{-p'} \in A_1$ with $[w]_{A_{p, q}}=[w^{-p'}]_{A_1}^{\frac{1}{p'}}$. 

For $s\in(1,\infty]$, the reverse H\"{o}lder class $RH_s$ is the collection of all weights $w$ such that
\begin{equation*}
[w]_{RH_s} := \sup_{Q} \bigg(\fint_Q w^s\,dx\bigg)^{\frac1s} \bigg(\fint_Q w\,dx\bigg)^{-1}<\infty.
\end{equation*}
When $s=\infty$, $(\fint_Q w^s\,dx)^{1/s}$ is understood as $(\esssup_{Q}w)$. Define $RH_1 := \bigcup\limits_{1<s \le \infty} RH_s$. Then we see that $RH_1=A_{\infty}$ (cf. \cite[Theorem 7.3.3]{Gra}).

\subsection{Muchenhoupt weights} 
The Hardy-Littlewood maximal operator $M$ is defined by 
\[
Mf(x) := \sup_{Q \ni x} \fint_Q |f(y)| \, dy, 
\]
where the supremum is taken over all cubes $Q \subset \Rn$ containing $x$. 
We begin with the following estimate concerning the growth of $C_{n, p}$ in \eqref{eq:Buckley} with respect to $n$ and $p$. 

\begin{lemma}\label{lem:sharp}
For any $p \in (1, \infty)$ and $w \in A_p$, 
\begin{align}\label{eq:sharp}
\|M\|_{L^p(w) \to L^p(w)} \le 2^n \cdot 3^{n(\frac{p}{p-1}+\frac{6}{p})} \, [w]_{A_p}^{\frac{1}{p-1}}. 
\end{align}
\end{lemma}

\begin{proof}
We follow the proof of \cite[Theorem 7.1.9]{Gra} to track the precise constants. Given a weight $w$, the centered weighted Hardy-Littlewood maximal operator $M_w^c$ is defined by 
\[
M_w^c f(x) := \sup_{Q \ni x} \frac{1}{w(Q)} \int_Q |f(y)| \, dw(y), 
\]
where the supremum is taken over all cubes $Q \subset \Rn$ centered at $x$. Let $M^c$ denote $M_w^c$ when $w \equiv 1$. It was proved in \cite[p.509]{Gra} that 
\begin{align}
\|M_w^c\|_{L^1(w) \to L^{1, \infty}(w)} \le 24^n 
\quad\text{ and }\quad
\|M_w^c\|_{L^{\infty}(w) \to L^{\infty}(w)} \le 1, 
\end{align}
which together with interpolation theorem gives that for any weight $w$, 
\begin{align}\label{eq:Mwc}
\|M_w^c\|_{L^p(w) \to L^p(w)} \le 24^{\frac{n}{p}}, \quad\forall p \in (1, \infty). 
\end{align}

To proceed, we fix $w \in A_p$ with $p \in (1, \infty)$, and set $\sigma:=w^{-\frac{1}{p-1}}$. As shown in \cite[p. 508]{Gra} that 
\begin{align*}
Mf(x) 
\le 2^n M^c f(x) 
\le 2^n \cdot 3^{\frac{np}{p-1}} [w]_{A_p}^{\frac{1}{p-1}} 
M_w^c\big(M_{\sigma}^c(f \sigma^{-1})^{p-1} w^{-1} \big)(x)^{\frac{1}{p-1}}. 
\end{align*}
which along with \eqref{eq:Mwc} in turn implies 
\begin{align*}
\|M\|_{L^p(w) \to L^p(w)} 
& \le 2^n \cdot 3^{\frac{np}{p-1}} [w]_{A_p}^{\frac{1}{p-1}} 
\|M_w^c\|_{L^{p'}(w) \to L^{p'}(w)}^{\frac{1}{p-1}} 
\|M_{\sigma}^c\|_{L^p(\sigma) \to L^p(\sigma)}  
\\
&\le 2^n \cdot 3^{\frac{np}{p-1}} \cdot 24^{\frac{n}{p'(p-1)} + \frac{n}{p}} [w]_{A_p}^{\frac{1}{p-1}} 
< 2^n \cdot 3^{n(\frac{p}{p-1}+\frac{6}{p})} [w]_{A_p}^{\frac{1}{p-1}}. 
\end{align*}
The proof is complete.  
\end{proof}

Based on the weighted boundedness of Hardy-Littlewood maximal operator above, one can establish Rubio de Francia extrapolation theorem below, whose proof was contained in \cite{CMP11}. 

\begin{theorem}\label{thm:RdF}
For any $p \in (1, \infty)$ and $w \in A_p$, there exists an operator $\mathcal{R}: L^p(w) \to L^p(w)$ such that for every non-negative function $h \in L^p(w)$,
\begin{list}{\rm (\theenumi)}{\usecounter{enumi}\leftmargin=1.2cm \labelwidth=1cm \itemsep=0.2cm \topsep=.2cm \renewcommand{\theenumi}{\alph{enumi}}}

\item\label{eq:RdF-1} $h \le \mathcal{R} h$;

\item\label{eq:RdF-2} $\|\mathcal{R} h\|_{L^p(w)} \le 2 \|h\|_{L^p(w)}$;

\item\label{eq:RdF-3} $\mathcal{R}h \in A_1$ with $[\mathcal{R} h]_{A_1} \le 2 \|M\|_{L^p(w) \to L^p(w)}$.

\end{list}
\end{theorem}

Let us recall the sharp reverse H\"{o}lder's inequality.
\begin{lemma}\label{lem:RH}
Let $p \in(1, \infty)$ and $w \in A_p$. Then there holds 
\begin{align}\label{eq:sharpRH}
\bigg(\fint_{Q} w^{1+ \gamma_w} dx \bigg)^{\frac{1}{1+\gamma_w}} \le 2 \fint_Q w \, dx,
\end{align}
for every cube $Q$, where
\begin{equation}\label{eq:gaw}
\gamma_w=
\begin{cases}
\frac{1}{2^{n+1}[w]_{A_1}}, & p=1, \\
\frac{1}{2^{n+1+2p}[w]_{A_p}}, &p \in (1, \infty), \\
\frac{1}{2^{n+11}[w]_{A_{\infty}}}, &p=\infty.
\end{cases}
\end{equation}
In particular, for any measurable subset $E \subset Q$, 
\begin{align}\label{eq:EQEQ}
w(E)/w(Q) \le 2 (|E|/|Q|)^{\frac{\gamma_w}{1+\gamma_w}}. 
\end{align}
\end{lemma}

\begin{proof}
The estimate \eqref{eq:sharpRH} was proved in \cite{CGPSZ, HP12, LOP}. Let us prove \eqref{eq:EQEQ}. If we set $r:=1+\gamma_w$, then \eqref{eq:sharpRH} implies that for any measurable subset $E \subset Q$, 
\begin{align*}
\frac{w(E)}{|Q|} = \fint_{Q} \mathbf{1}_E \, w \, dx 
\le \bigg(\fint_Q \mathbf{1}_E^{r'} \, dx \bigg)^{\frac{1}{r'}} 
\bigg(\fint_Q w^r dx\bigg)^{\frac1r} 
\le 2\bigg(\frac{|E|}{|Q|}\bigg)^{\frac{\gamma_w}{1+\gamma_w}} \frac{w(Q)}{|Q|}. 
\end{align*}
This shows \eqref{eq:EQEQ}. 
\end{proof}

\begin{lemma}\label{lem:open}
For any $q \in (1, \infty)$ and $v \in A_q$, there exist $\gamma \in (0, 2^{-n-3})$ and $q_0 \in (1, q)$ such that 
\begin{equation}\label{eq:vava}
\begin{aligned}
& q_0 = \frac{q}{1+\varepsilon}, \quad
\frac{(q-1)\gamma}{q(1+\gamma)'} < \varepsilon < \frac{q-1}{(1+\gamma)'}, \quad 
(1+\gamma)' \simeq [v]_{A_q}^{\max\{1, \frac{1}{q-1}\}}, 
\\ 
& [v^{1+\gamma}]_{A_q} \le 2^{q(1+\gamma)} [v]_{A_q}^{1+\gamma},  
\quad\text{ and }\quad 
[v]_{A_{q_0}} \le 2^q [v]_{A_q}. 
\end{aligned}
\end{equation} 
\end{lemma}

\begin{proof}
Let $q \in (1, \infty)$ and $v \in A_q$. Then, $v^{1-q'} \in A_{q'}$, and by Lemma \ref{lem:RH}, 
\begin{align}
\label{eq:ga-1} \bigg(\fint_{Q} v^{1+ \gamma_1} dx \bigg)^{\frac{1}{1+\gamma_1}} 
&\le 2 \fint_Q v \, dx,
\\
\label{eq:ga-2} \bigg(\fint_{Q} v^{-\frac{1+ \gamma_2}{q-1}} dx \bigg)^{\frac{1}{1+\gamma_2}} 
&\le 2 \fint_Q v^{-\frac{1}{q-1}} \, dx,
\end{align}
for any cube $Q \subset \Rn$, where 
\begin{align}\label{eq:ga-3}
\gamma_1 := \frac{1}{2^{n+1+2q}[v]_{A_q}} \quad\text{ and }\quad 
\gamma_2 := \frac{1}{2^{n+1+2q}[v^{1-q'}]_{A_{q'}}}. 
\end{align}
Setting 
\begin{align}\label{eq:gaqq}
\gamma := \min\{\gamma_1, \gamma_2\} < 2^{-n-3} \quad\text{and}\quad 
q_0 := \frac{q}{1+\varepsilon} = \frac{q+\gamma}{1+\gamma} \in (1, q), 
\end{align}
we see that 
\[
\frac{(q-1)\gamma}{q(1+\gamma)'}
<\varepsilon = \frac{(q-1) \gamma}{q+\gamma}
< \frac{(q-1)\gamma}{1+\gamma} 
= \frac{q-1}{(1+\gamma)'}, 
\]
and use Jensen's inequality and \eqref{eq:ga-1}--\eqref{eq:ga-3} to obtain 
\begin{equation*}
(1+\gamma)' \simeq \max\{[v]_{A_q}, [v^{1-q'}]_{A_{q'}}\}
= [v]_{A_q}^{\max\{1, \frac{1}{q-1}\}}, 
\end{equation*}
\begin{align*}
\bigg(\fint_Q v^{1+ \gamma} dx \bigg)^{\frac{1}{1+\gamma}} 
\bigg(\fint_Q v^{-\frac{1+ \gamma}{q-1}} dx \bigg)^{\frac{q-1}{1+\gamma}} 
&\le \bigg(\fint_Q v^{1+ \gamma_1} dx \bigg)^{\frac{1}{1+\gamma_1}} 
\bigg(\fint_Q v^{-\frac{1+ \gamma_2}{q-1}} dx \bigg)^{\frac{q-1}{1+\gamma_2}} 
\\ 
&\le 2^q \bigg(\fint_Q v \, dx \bigg) \bigg(\fint_Q v^{-\frac{1}{q-1}} \, dx\bigg)^{q-1} , 
\end{align*}
and 
\begin{align*}
\bigg(\fint_Q v \, dx \bigg) \bigg(\fint_Q v^{-\frac{1}{q_0-1}} dx \bigg)^{q_0-1} 
&= \bigg(\fint_Q v \, dx \bigg) \bigg(\fint_Q v^{-\frac{1+ \gamma}{q-1}} dx \bigg)^{\frac{q-1}{1+\gamma}} 
\\ 
&\le \bigg(\fint_Q v^{1+ \gamma} dx \bigg)^{\frac{1}{1+\gamma}} 
\bigg(\fint_Q v^{-\frac{1+ \gamma}{q-1}} dx \bigg)^{\frac{q-1}{1+\gamma}},  
\end{align*}
which immediately implies \eqref{eq:vava}. 
\end{proof}

\begin{lemma}\label{lem:fac}
The following properties hold: 
\begin{list}{\rm (\theenumi)}{\usecounter{enumi}\leftmargin=1.2cm \labelwidth=1cm \itemsep=0.2cm \topsep=.2cm \renewcommand{\theenumi}{\alph{enumi}}}

\item\label{list:fac-1} Let $1 \le p \le p_0 < \infty$. Then for any $u \in A_p$ and $v \in A_1$,
\begin{align*}
uv^{p-p_0} \in A_{p_0} \quad\text{with}\quad [u v^{p-p_0}]_{A_{p_0}} \le [u]_{A_p} [v]_{A_1}^{p_0-p}.
\end{align*}

\item\label{list:fac-2} Let $1 \le q_0, q_1<\infty$. Then for any $w_0 \in A_{q_0}$, $w_1 \in A_{q_1}$, and $\theta \in[0, 1]$, 
\begin{align*}
[w]_{A_q} \le [w_0]_{A_{q_0}}^{(1-\theta) \frac{q}{q_0}} [w_1]_{A_{q_1}}^{\theta \frac{q}{q_1}}, 
\end{align*}
where $\frac1q = \frac{1-\theta}{q_0} + \frac{\theta}{q_1}$ and $w^{\frac1q} = w_0^{\frac{1-\theta}{q_0}} w_1^{\frac{\theta}{q_1}}$. In particular, for any $1 \le p_0 < p<\infty$, $u \in A_p$, and $v \in A_1$, 
\begin{equation*}
u^{\frac{p_0-1}{p-1}} v^{\frac{p-p_0}{p-1}} \in A_{p_0} 
\quad\text{ with }\quad 
\big[u^{\frac{p_0-1}{p-1}} v^{\frac{p-p_0}{p-1}} \big]_{A_{p_0}} 
\le [u]_{A_p}^{\frac{p_0-1}{p-1}} [v]_{A_1}^{\frac{p-p_0}{p-1}}. 
\end{equation*}
\end{list} 
\end{lemma}

\begin{proof}
We begin with showing part \eqref{list:fac-1}. Let $u \in A_p$ and $v \in A_1$. For each cube $Q$,
\begin{align}\label{eq:fac-1}
\fint_Q u v^{p-p_0} \, dx
\le \bigg(\fint_Q u \, dx\bigg) \,  (\esssup_Q v^{-1})^{p_0-p}.
\end{align}
Set $r=\frac{p'-1}{p'_0-1} = \frac{p_0-1}{p-1} \ge 1$. Then $r'=\frac{p_0-1}{p_0-p}$, and by H\"{o}lder's inequality,
\begin{align}\label{eq:fac-2}
\bigg(\fint_Q (u v^{p-p_0})^{1-p'_0} dx \bigg)^{p_0-1}
&\le \bigg(\fint_Q u^{(1-p'_0)r} dx \bigg)^{\frac{p_0-1}{r}}
\bigg(\fint_Q v^{(p-p_0)(1-p'_0)r'} dx \bigg)^{\frac{p_0-1}{r'}}
\\ \nonumber
&=\bigg(\fint_Q u^{1-p'} \, dx \bigg)^{p-1}  \bigg(\fint_Q v \, dx \bigg)^{p_0-p}.
\end{align}
Then it follows from \eqref{eq:fac-1} and \eqref{eq:fac-2} that $[u v^{p-p_0}]_{A_{p_0}} \le [u]_{A_p} [v]_{A_1}^{p_0-p}$.

Next, let us prove part \eqref{list:fac-2}. Note that $\frac1q = \frac{1-\theta}{q_0} + \frac{\theta}{q_1}$, and then 
\begin{align*}
\frac{1-\theta}{\frac{q_0}{q_0-1}} + \frac{\theta}{\frac{q_1}{q_1-1}}
=(1-\theta) \bigg(1-\frac{1}{q_0}\bigg) + \theta \bigg(1-\frac{1}{q_1}\bigg)
=1-\frac{1-\theta}{q_0} - \frac{\theta}{q_1}
=\frac{q-1}{q}. 
\end{align*}
Thus, H\"{o}lder's inequality gives  
\begin{align}\label{qqw-1}
\fint_Q w\, dx 
=\fint_Q \Big(w_0^{\frac{1-\theta}{q_0}} w_1^{\frac{\theta}{q_1}}\Big)^q dx 
\le \bigg(\fint_Q w_0 \, dx\bigg)^{(1-\theta)\frac{q}{q_0}} \bigg(\fint_Q w_1 \, dx \bigg)^{\theta\frac{q}{q_1}} 
\end{align}
and
\begin{align}\label{qqw-2}
\bigg(\fint_Q w^{-\frac{1}{q-1}}\, dx\bigg)^{q-1}
&= \bigg(\fint_Q \Big(w_0^{-\frac{1-\theta}{q_0}} w_1^{-\frac{\theta}{q_1}} \Big)^{\frac{q}{q-1}} dx \bigg)^{q-1}
\\ \nonumber 
&\le \bigg(\fint_Q w_0^{-\frac{1}{q_0-1}} \, dx\bigg)^{(q_0-1)(1-\theta)\frac{q}{q_0}} 
\bigg(\fint_Q w_1^{-\frac{1}{q_1-1}} \, dx \bigg)^{(q_1-1)\theta\frac{q}{q_1}}. 
\end{align}
By definition, \eqref{qqw-1}, and \eqref{qqw-2}, we immediately obtain $[w]_{A_q} \le [w_0]_{A_{q_0}}^{(1-\theta) \frac{q}{q_0}} [w_1]_{A_{q_1}}^{\theta \frac{q}{q_1}}$. To conclude the proof, it suffices to pick 
\[
w_0 := u, \quad 
w_1 := v, \quad 
q_0 := p, \quad
q_1 := 1,\quad 
q := p_0, \quad 
w := u^{\frac{p_0-1}{p-1}} v^{\frac{p-p_0}{p-1}}, \quad
\theta := \frac{p-p_0}{p_0(p-1)}, 
\]
and note that $w_0^{\frac{1-\theta}{q_0}} w_1^{\frac{\theta}{q_1}} 
=u^{\frac{p_0-1}{p_0(p-1)}} v^{\frac{p-p_0}{p_0(p-1)}}
=w^{\frac{1}{p_0}}
=w^{\frac1q}$, 
and 
\begin{align*}
\frac{1-\theta}{q_0} + \frac{\theta}{q_1} 
&=\frac{p_0-1}{p_0(p-1)} + \frac{p-p_0}{p_0(p-1)}
=\frac{1}{p_0} 
= \frac1q. 
\end{align*}
The proof is complete. 
\end{proof}

We sum up some of the properties of these classes in the following result.

\begin{lemma}\label{lem:weights}
The following statements hold:
\begin{list}{\rm (\theenumi)}{\usecounter{enumi}\leftmargin=1.2cm \labelwidth=1cm \itemsep=0.2cm \topsep=.2cm \renewcommand{\theenumi}{\alph{enumi}}}

\item\label{list:ApRH-1} For any $w_1, w_2 \in A_1$,  $w:=w_1^{1/s} w_{2}^{1-p} \in A_p \cap RH_s$ for all $1 \leq p < \infty$ and $1<s \leq \infty$. Moreover,  
\begin{align}
\max\{[w]_{A_p}, [w]_{RH_s} \}
\le [w_1]_{A_1}^{\frac1s} [w_2]_{A_1}^{p-1}. 
\end{align}

\item\label{list:ApRH-2} Given $1 \leq p < \infty$ and $1 \leq s < \infty$, $w \in A_p \cap RH_s$ if and only if $w^s \in A_{\tau}$. Moreover,
\begin{align}
[w^s]_{A_{\tau}} \le [w]_{A_p}^s [w]_{RH_s}^s \quad\text{and}\quad
\max\big\{[w]_{A_p}^s, [w]_{RH_s}^s \big\} \le [w^s]_{A_{\tau}},
\end{align}
where $\tau=s(p-1)+1$.

\item\label{list:ApRH-3} Let $1 \le \p_- < \p_+ \le \infty$ and $p \in (\p_-, \p_+)$. Then $w^p \in A_{p/\p_-} \cap RH_{(\p_+/p)'}$ if and only if $w^{-p'} \in A_{p'/\p'_+} \cap RH_{(\p'_-/p')'}$ with
\begin{align}
[w^{p(\p_+/p)'}]_{A_{\tau_p}} = [w^{-p'(\p'_-/p')'}]_{A_{\tau'_p}}^{\tau_p -1},
\end{align}
where $\tau_p= \big(\frac{\p_+}{p}\big)' \big(\frac{p}{\p_-} -1\big) +1$.

\item\label{list:ApRH-4} Given $1 \le \p_- < \p_+ \le \infty$, $p \in (\p_-, \p_+)$, and $w^p \in A_{p/\p_-} \cap RH_{(\p_+/p)'}$, there exists $\widetilde{\p}_- \in (\p_-, p)$ such that $w^p \in A_{p/\widetilde{\p}_-} \cap RH_{(\p_+/p)'}$ with 
\begin{align}\label{eq:wawa}
[w^{p(\p_+/p)'}]_{A_{\widetilde{\tau}_p}} 
\le 2^{\tau_p} [w^{p(\p_+/p)'}]_{A_{\tau_p}} 
\quad\text{ and }\quad 
\frac{\frac{1}{\widetilde{\p}_-}}{\frac{1}{\widetilde{\p}_-} - \frac1p} 
< (1+2^{-n-3}) \frac{\frac{1}{\p_-}}{\frac{1}{\p_-} - \frac1p}, 
\end{align}
where $\tau_p= \big(\frac{\p_+}{p}\big)' \big(\frac{p}{\p_-} -1\big) +1$ and $\widetilde{\tau}_p = \big(\frac{\p_+}{p}\big)' \big(\frac{p}{\widetilde{\p}_-} -1\big) +1$. 
\end{list}
\end{lemma}

\begin{proof}
Parts \eqref{list:ApRH-1}--\eqref{list:ApRH-3} are essentially contained in \cite{AM-1, JN}. We present a detailed proof to track the weight norms. To show \eqref{list:ApRH-1}, we fix $1 \leq p < \infty$, $1<s \leq \infty$, and let $w_1, w_2 \in A_1$. By Jensen's inequality, 
\begin{align}\label{ess-1}
\fint_Q w \, dx 
\le \bigg(\fint_Q w_1^{\frac1s} \, dx\bigg) \big(\esssup_Q w_2^{-1}\big)^{p-1}
\le \bigg(\fint_Q w_1 \, dx\bigg)^{\frac1s} \big(\esssup_Q w_2^{-1}\big)^{p-1}
\end{align}
and 
\begin{align}\label{ess-2}
\bigg(\fint_Q w^{-\frac{1}{p-1}} \, dx \bigg)^{p-1}
= \bigg(\fint_Q w_1^{-\frac{1}{s(p-1)}} w_2 \, dx \bigg)^{p-1}
\le \big(\esssup_Q w_1^{-1}\big)^{\frac1s} \bigg(\fint_Q w_2 \, dx\bigg)^{p-1}, 
\end{align}
when $p=1$, the inequality \eqref{ess-2} is replaced by 
\begin{align}\label{ess-3}
\esssup_Q w^{-1}
= \big(\esssup_Q w_1^{-1} \big)^{\frac1s}. 
\end{align}
Then it follows from \eqref{ess-1}--\eqref{ess-3} that 
\[
[w]_{A_p} \le [w_1]_{A_1}^{\frac1s} [w_2]_{A_1}^{p-1}.
\]
Moreover, by definition and Jensen's inequality, we have 
\begin{align*}
\bigg(\fint_Q w^s \, dx \bigg)^{\frac1s} 
&= \bigg(\fint_Q w_1 w_2^{s(1-p)} \, dx \bigg)^{\frac1s} 
\le \bigg(\fint_Q w_1 \, dx \bigg)^{\frac1s} \big(\esssup_Q w_2^{-1} \big)^{p-1}
\\ 
&\le [w_1]_{A_1}^{\frac1s} [w_2]_{A_1}^{p-1} \big(\essinf_Q w_1^{\frac1s} \big) \bigg(\fint_Q w_2 \, dx\bigg)^{1-p}
\\ 
&\le [w_1]_{A_1}^{\frac1s} [w_2]_{A_1}^{p-1} \big(\essinf_Q w_1^{\frac1s} \big) \bigg(\fint_Q w_2^{1-p} \, dx\bigg)
\\ 
&\le [w_1]_{A_1}^{\frac1s} [w_2]_{A_1}^{p-1} \bigg(\fint_Q w_1^{\frac1s} w_2^{1-p} \, dx\bigg)
\\ 
&= [w_1]_{A_1}^{\frac1s} [w_2]_{A_1}^{p-1} \bigg(\fint_Q w \, dx\bigg), 
\end{align*}
when $s=\infty$, the above still holds since $(\fint_Q w^s \, dx)^{\frac1s}$ is replaced by $\esssup_Q w$. 
This means 
\[
[w]_{RH_s} \le [w_1]_{A_1}^{\frac1s} [w_2]_{A_1}^{p-1}.
\]

Let us next show \eqref{list:ApRH-2}. Assume first that $w \in A_p \cap RH_s$. Note that for any cube $Q$,
\begin{align}\label{eq:wstau}
\bigg(\fint_Q w^{s(1-\tau')} dx \bigg)^{\tau-1}
=\bigg(\fint_Q w^{1-p'} dx \bigg)^{s(p-1)}.
\end{align}
This implies
\begin{align*}
\bigg(\fint_Q w^s \, dx \bigg) \bigg(\fint_Q w^{s(1-\tau')} dx \bigg)^{\tau-1}
\le [w]_{RH_s}^s \bigg(\fint_Q w \, dx \bigg)^s \bigg(\fint_Q w^{1-p'} dx \bigg)^{s(p-1)},
\end{align*}
and hence,
\begin{align}\label{eq:wsat-1}
[w^s]_{A_{\tau}} \le [w]_{A_p}^s [w]_{RH_s}^s .
\end{align}
On the other hand, assuming $w^s \in A_{\tau}$, we deduce by Jensen's inequality and \eqref{eq:wstau},
\begin{align}\label{eq:QW-1}
\bigg(\fint_Q w \, dx \bigg) \bigg(\fint_Q w^{1-p'} dx \bigg)^{p-1}
\le \bigg(\fint_Q w^s \, dx \bigg)^{\frac1s} \bigg(\fint_Q w^{s(1-\tau')} dx \bigg)^{\frac{\tau-1}{s}},
\end{align}
and
\begin{align}\label{eq:QW-2}
\bigg(\fint_Q w^s \, dx \bigg)^{\frac1s}
&\le [w^s]_{A_{\tau}}^{\frac1s} \bigg(\fint_Q w^{s(1-\tau')} dx \bigg)^{-\frac{\tau-1}{s}}
\\ \nonumber
&= [w^s]_{A_{\tau}}^{\frac1s} \bigg(\fint_Q w^{1-p'} dx \bigg)^{-(p-1)}
\le [w^s]_{A_{\tau}}^{\frac1s} \bigg(\fint_Q w dx \bigg),
\end{align}
which follows from
\begin{align*}
1 = \bigg(\fint_Q w^{\frac1p} w^{-\frac1p} dx \bigg)^p
\le \bigg(\fint_Q w \, dx \bigg) \bigg(\fint_Q w^{1-p'} dx \bigg)^{p-1}.
\end{align*}
Then, \eqref{eq:QW-1} and \eqref{eq:QW-2} imply
\begin{align}\label{eq:wsat-2}
[w]_{A_p} \le [w^s]_{A_{\tau}}^{\frac1s} \quad\text{ and }\quad [w]_{RH_s} \le [w^s]_{A_{\tau}}^{\frac1s}.
\end{align}
Hence, \eqref{list:ApRH-2} follows from \eqref{eq:wsat-1} and \eqref{eq:wsat-2}.

We turn to the proof of \eqref{list:ApRH-3}. One can check that
\begin{align}\label{eq:ttaauu}
\bigg(\frac{\p_+}{p}\bigg)' (\tau'_p - 1) = \bigg(\frac{\p'_-}{p'} \bigg)'(p' -1) \quad\text{ and }\quad
\tau'_p = \bigg(\frac{\p'_-}{p'} \bigg)' \bigg(\frac{p'}{\p'_+} - 1 \bigg) +1 .
\end{align}
Then it follows that
\begin{align*}
-p'(\p'_-/p')' (1-(\tau'_p)') = p(p'-1) (\p'_-/p')' (\tau_p - 1) = p(\p_+/p)',
\end{align*}
and for any cube $Q$,
\begin{align*}
&\bigg(\fint_Q w^{p(\p_+/p)'} dx\bigg) \bigg(\fint_Q w^{p(\p_+/p)' (1 - \tau'_p)} dx\bigg)^{\tau_p -1}
\\
&\quad= \bigg[\bigg(\fint_Q w^{-p'(\p'_-/p')' (1-(\tau'_p)')} dx\bigg)^{\tau'_p -1}
\bigg(\fint_Q w^{-p'(\p'_-/p')'} dx\bigg) \bigg]^{\tau_p -1},
\end{align*}
which implies
\begin{align*}
[w^{p(\p_+/p)'}]_{A_{\tau_p}} = [w^{-p(\p'_-/p')'}]_{A_{\tau'_p}}^{\tau_p -1}.
\end{align*}

Finally, let us demonstrate \eqref{list:ApRH-4}. By part \eqref{list:ApRH-2}, there holds $v := w^{p(\p_+/p)'} \in A_{\tau_p}$, which along with \eqref{eq:vava} and \eqref{eq:gaqq} applied to exponents $q=\tau_p$  and $q_0=\widetilde{\tau}_p$, to arrive at the first estimate in \eqref{eq:wawa} and 
\begin{align*}
\widetilde{\p}_- = \frac{p}{1+\frac{\tau_p - 1}{(\p_+/p)' (1+\gamma)}}
= \frac{1}{\frac1p + \frac{1}{1+\gamma}(\frac{1}{\p_-} - \frac1p)} \in (\p_-, p). 
\end{align*}
Moreover, 
\begin{align*}
\frac{\frac{1}{\widetilde{\p}_-}}{\frac{1}{\widetilde{\p}_-} - \frac1p} 
= (1+\gamma) \frac{\p_-}{\widetilde{\p}_-} \frac{\frac{1}{\p_-}}{\frac{1}{\p_-} - \frac1p} 
< (1+2^{-n-3}) \frac{\frac{1}{\p_-}}{\frac{1}{\p_-} - \frac1p}, 
\end{align*}
This proves the second estimate in \eqref{eq:wawa} and completes the proof. 
\end{proof}

\subsection{Multilinear Muckenhoupt weights}
The multilinear maximal operator is defined by
\begin{equation}\label{eq:def-M}
\M(\vec{f})(x):= \sup_{Q \ni x} \prod_{i=1}^m \fint_Q |f_i(y_i)| dy_i,
\end{equation}
where the supremum is taken over all cubes $Q$ containing $x$.

We are going to present the definition of the multilinear Muckenhoupt classes $A_{\vec{p}, \vec{r}}$ introduced in \cite{LMO, LMMOV}.  Given $\vec{p}=(p_1, \ldots, p_m)$ with $1 \le p_1, \ldots, p_m \le \infty$ and $\vec{r}=(r_1, \ldots, r_{m+1})$ with $1 \le r_1, \ldots, r_{m+1} < \infty$, we say that $\vec{r} \preceq \vec{p}$ whenever
\begin{align*}
r_i  \le p_i,\, i=1, \ldots, m, \text{ and } r'_{m+1} \ge p, \text{ where } \frac1p:=\frac{1}{p_1}+\dots+\frac{1}{p_m}.
\end{align*}
Analogously, we say that $\vec{r} \prec \vec{p}$ if $r_i<p_i$ for each $i=1, \ldots, m$, and $r'_{m+1}>p$.

\begin{definition}\label{def:Apr}
Let $\vec{p}=(p_1,\ldots,p_m)$ with $1\leq p_1, \ldots, p_m \le \infty$ and let $\vec{r}=(r_1, \ldots, r_{m+1})$ with $1 \le r_1, \ldots, r_{m+1} < \infty$ such that $\vec{r} \preceq \vec{p}$. Suppose that $\vec{w}=(w_1,\ldots,w_m)$ and each $w_i$ is a weight on $\Rn$. We say that $\vec{w} \in A_{\vec{p}, \vec{r}}$ if
\begin{align}\label{eq:w-Apr}
[\vec{w}]_{A_{\vec{p}, \vec{r}}}
:=\sup_Q \bigg(\fint_Q w^{\frac{r'_{m+1}p}{r'_{m+1}-p}}\, dx\bigg)^{\frac{1}{p}-\frac{1}{r'_{m+1}}}
\prod_{i=1}^m \bigg(\fint_Q w_i^{\frac{r_i p_i}{r_i-p_i}}dx \bigg)^{\frac{1}{r_i}-\frac{1}{p_i}}< \infty,
\end{align}
where $\frac1p = \sum_{i=1}^m \frac{1}{p_i}$, $w=\prod_{i=1}^m w_i$, and the supremum is taken over all cubes $Q \subset \Rn$. When $p=r'_{m+1}$, the term corresponding to $w$ needs to be replaced by $\esssup_Q w$ and, analogously, when $p_i=r_i$, the term corresponding to $w_i$ should be $\esssup_Q w_i^{-1}$. Also, if $p_i = \infty$, the term corresponding to $w_i$ becomes $\big(\fint_Q w_i^{-r_i} dx\big)^{\frac{1}{r_i}}$. If $p=\infty$, one will necessarily have $r_{m+1} = 1$ and $p_1= \cdots = p_m=\infty$, hence the term corresponding to $w$ must be $\esssup_Q w$ while the terms corresponding to $w_i$ become $\big(\fint_Q w_i^{-r_i} dx\big)^{\frac{1}{r_i}}$.  When $r_{m+1} = 1$ and $p<\infty$ the term corresponding to $w$ needs to be replaced by $\big(\fint_Q w^p dx\big)^{\frac1p}$.
\end{definition}

Denote $A_{\vec{p}} := A_{\vec{p}, (1, \ldots, 1)}$ in Definition \ref{def:Apr}, that is, 
\begin{align}\label{eq:w-Ap}
[\vec{w}]_{A_{\vec{p}}}
:= \sup_Q \bigg(\fint_Q w^p\, dx\bigg)^{\frac{1}{p}}
\prod_{i=1}^m \bigg(\fint_Q w_i^{-p'_i}dx \bigg)^{\frac{1}{p'_i}}< \infty,
\end{align}
where $\frac1p = \sum_{i=1}^m \frac{1}{p_i}$ and $w=\prod_{i=1}^m w_i$. We would like to observe that our definition of the classes $A_{\vec{p}}$ and $A_{\vec{p}, \vec{r}}$ is slightly different to that in \cite{LOPTT} and \cite{LMO}. Essentially, they are the same since picking $w_i=v_i^{p_i}$ for every $i=1, \ldots, m$ in \eqref{eq:w-Apr} and \eqref{eq:w-Ap}, we see that $\vec{v}=(v_1, \ldots, v_m)$ belongs to $A_{\vec{p}, \vec{r}}$ in \cite{LMO} and to $A_{\vec{p}}$ in \cite{LOPTT}, respectively.

\begin{lemma}\label{lem:ApRHApr}
Let $1 \le \p_i^{-}<\p_i^{+} \le \infty$, $i=1, \ldots, m$. Assume that $p_i \in [\p_i^-, \p_i^+]$ and $w_i^{p_i} \in A_{p_i/\p_i^-} \cap RH_{(\p_i^+/p_i)'}$, $i=1, \ldots, m$. Then $\vec{w}=(w_1, \ldots, w_m) \in A_{\vec{q}, \vec{r}}$ with
\begin{align*}
[\vec{w}]_{A_{\vec{q}, \vec{r}}}
\le \prod_{i=1}^m [w_i^{p_i(\p_i^+/p_i)'}]_{A_{\tau_{p_i}}}^{\frac{1}{p_i} - \frac{1}{\p_i^+}}, 
\end{align*}
for any $\vec{q}=(q_1, \ldots, q_m)$ with $1 \le q_1, \ldots, q_m < \infty$ and $\vec{r} = (r_1, \ldots, r_{m+1})$ with $1 \le r_1, \ldots, r_{m+1} < \infty$ such that $\vec{r} \preceq \vec{q}$, and
\begin{align}\label{eq:rqpp}
\frac1q - \frac{1}{r'_{m+1}} = \sum_{i=1}^m \bigg(\frac{1}{p_i} - \frac{1}{\p_i^+}\bigg),
\qquad \frac{1}{r_i} - \frac{1}{q_i} = \frac{1}{\p_i^-} - \frac{1}{p_i}, \quad i=1, \ldots, m.
\end{align}
\end{lemma}

\begin{proof}
By Lemma \ref{lem:weights} part \eqref{list:ApRH-2}, one has
\begin{align*}
[w_i^{p_i(\p_i^+/p_i)'}]_{A_{\tau_{p_i}}}
\le \Big([w_i^{p_i}]_{A_{p_i/\p_i^-}} [w_i^{p_i}]_{RH_{(\p_i^+/p_i)'}} \Big)^{p_i(\p_i^+/p_i)'},
\end{align*}
where $\tau_{p_i} := \big(\frac{\p_i^+}{p_i} \big)' \big(\frac{p_i}{\p_i^-} -1 \big) +1$, $i=1, \ldots, m$.  Set
\begin{align}\label{eq:tspp}
\frac{1}{t_i} := \frac{1}{p_i} - \frac{1}{\p_i^+}  \quad\text{ and }\quad
\frac{1}{s_i} := 1- \bigg(\frac{1}{\p_i^-} - \frac{1}{p_i} \bigg), \quad i=1, \ldots, m.
\end{align}
Then it is easy to check that
\begin{align}\label{eq:tipi}
t_i = p_i(\p_i^+/p_i)' \quad\text{ and }\quad
s'_i = t_i (\tau'_{p_i} -1),  \quad i=1, \ldots, m,
\end{align}
which gives
\begin{align}\label{eq:witi}
[w_i^{t_i}]_{A_{\tau_{p_i}}}
&= \sup_{Q} \bigg(\fint_Q w_i^{t_i}\, dx\bigg) \bigg(\fint_Q w_i^{-t_i(\tau'_{p_i}-1)}\bigg)^{\tau_{p_i}-1}
\\ \nonumber
&=\sup_{Q} \bigg[\bigg(\fint_Q w_i^{t_i}\, dx\bigg)^{\frac{1}{t_i}} \bigg(\fint_Q w_i^{-s'_i}\bigg)^{\frac{1}{s'_i}}\bigg]^{t_i}
=[w_i]_{A_{s_i, t_i}}^{t_i}.
\end{align}
On the other hand, let $\vec{q}=(q_1, \ldots, q_m)$ with $1 \le q_1, \ldots, q_m < \infty$ and $\vec{r}=(r_1, \ldots, r_{m+1})$ with $1\le r_1, \ldots, r_{m+1}<\infty$ such that $\vec{r} \preceq \vec{q}$ and \eqref{eq:rqpp} holds. It follows from \eqref{eq:rqpp} and \eqref{eq:tspp} that
\begin{align}\label{eq:qrtt}
\frac1q-\frac{1}{r'_{m+1}} = \frac{1}{t_1}+\cdots+\frac{1}{t_m}
\quad\text{ and }\quad
\frac{1}{r_i} - \frac{1}{q_i} = \frac{1}{s'_i}, \quad i=1, \ldots, m.
\end{align}
Thus, writing $w=\prod_{i=1}^m w_i$, we use \eqref{eq:qrtt} and H\"{o}lder's inequality to obtain
\begin{align}\label{eq:wrq}
&\bigg(\fint_Q w^{\frac{r'_{m+1}q}{r'_{m+1}-q}}\, dx\bigg)^{\frac{1}{q}-\frac{1}{r'_{m+1}}}
\prod_{i=1}^m \bigg(\fint_Q w_i^{\frac{r_i q_i}{r_i - q_i}}dx \bigg)^{\frac{1}{r_i}-\frac{1}{q_i}}
\\ \nonumber
&\quad\le \prod_{i=1}^m \bigg(\fint_Q w_i^{t_i}\, dx\bigg)^{\frac{1}{t_i}}
\bigg(\fint_Q w_i^{-s'_i}dx \bigg)^{\frac{1}{s'_i}}
\le  \prod_{i=1}^m [w_i]_{A_{s_i, t_i}}.
\end{align}
As a consequence, collecting \eqref{eq:tspp}, \eqref{eq:tipi}, \eqref{eq:witi}, and \eqref{eq:wrq}, we conclude that
\begin{align*}
[\vec{w}]_{A_{\vec{q}, \vec{r}}}
\le \prod_{i=1}^m [w_i]_{A_{s_i, t_i}}
= \prod_{i=1}^m [w_i^{t_i}]_{A_{\tau_{p_i}}}^{\frac{1}{t_i}}
= \prod_{i=1}^m [w_i^{p_i(\p_i^+/p_i)'}]_{A_{\tau_{p_i}}}^{\frac{1}{p_i} - \frac{1}{\p_i^+}}.
\end{align*}
This completes the proof.
\end{proof}

\begin{lemma}\label{lem:ApRH-m}
Let $1 \le \p_i^{-}<\p_i^{+} \le \infty$, $i=1, \ldots, m$. Assume that $p_i \in [\p_i^-, \p_i^+]$ and $w_i^{p_i} \in A_{p_i/\p_i^-} \cap RH_{(\p_i^+/p_i)'}$, $i=1, \ldots, m$. Write $w=\prod_{i=1}^m w_i$. Then $w^p \in A_{p/\p_-} \cap RH_{(\p_+/p)'}$ with
\begin{align}\label{eq:wppp}
[w^{p(\p_+/p)'}]_{A_{\tau_p}}
\le \prod_{i=1}^m [w_i^{p_i(\p_i^+/p_i)'}]_{A_{\tau_{p_i}}}^{\frac{\frac{1}{p_i} - \frac{1}{\p_i^+}}{\frac1p - \frac{1}{\p_+}}},
\end{align}
where $\frac1p= \frac{1}{p_1} + \cdots+ \frac{1}{p_m}$ and $\frac{1}{\p_{\pm}} = \frac{1}{\p_1^{\pm}} + \cdots + \frac{1}{\p_m^{\pm}}$. In particular, if we take 
\begin{align}\label{eq:wwpp}
w_{m+1} := w^{-1}, \quad 
p_{m+1} := p', \quad 
\p_{m+1}^- := \p'_+, 
\quad\text{ and }\quad 
\p_{m+1}^+ := \p'_-, 
\end{align} 
then it follows 
\begin{align}\label{eq:wpppp}
[w_{m+1}^{p_{m+1}(\p_{m+1}^+/p_{m+1})'}]_{A_{\tau_{p_{m+1}}}} 
\le \prod_{i=1}^m [w_i^{p_i(\p_i^+/p_i)'}]_{A_{\tau_{p_i}}}^{\frac{\frac{1}{p_i} - \frac{1}{\p_i^+}}{\frac{1}{\p_-} - \frac1p}}. 
\end{align}
\end{lemma}

\begin{proof}
Set
\begin{align}
\label{eq:rr} \frac1r & := \frac1p - \frac{1}{\p_+} = \sum_{i=1}^m \Big(\frac{1}{p_i} -\frac{1}{\p_i^+} \Big)
=: \sum_{i=1}^m \frac{1}{r_i},
\\
\label{eq:ss} \frac1s & := \frac{1}{\p_-} - \frac1p = \sum_{i=1}^m \Big(\frac{1}{\p_i^-} - \frac{1}{p_i} \Big)
=: \sum_{i=1}^m \frac{1}{s_i}.
\end{align}
Observe that
\begin{align}\label{eq:pipi}
p_i(\p_i^+/p_i)' = \frac{1}{\frac{1}{p_i} - \frac{1}{\p_i^+}}
\quad\text{ and }\quad
\frac{p_i(\p_i^+/p_i)'}{\tau_{p_i}-1} = \frac{1}{\frac{1}{\p_i^-} - \frac{1}{p_i}}, \quad i=1, \ldots, m.
\end{align}
With \eqref{eq:rr}--\eqref{eq:pipi} in hand, we use H\"{o}lder's inequality to obtain
\begin{align}\label{eq:S1}
\mathscr{S}_1
&:= \bigg(\fint_Q w^{p(\p_+/p)'} dx \bigg)^{\frac{1}{p(\p_+/p)'}}
=\bigg(\fint_Q w^{\frac{1}{\frac1p - \frac{1}{\p_+}}} dx \bigg)^{\frac1p - \frac{1}{\p_+}}
\\ \nonumber
&= \bigg(\fint_Q \Big(\prod_{i=1}^m w_i \Big)^r dx \bigg)^{\frac1r}
\le \prod_{i=1}^m \bigg(\fint_Q w_i^{r_i} dx \bigg)^{\frac{1}{r_i}}
\\ \nonumber
&= \prod_{i=1}^m \bigg(\fint_Q w_i^{\frac{1}{\frac{1}{p_i} - \frac{1}{\p_i^+}}} dx \bigg)^{\frac{1}{p_i} - \frac{1}{\p_i^+}}
= \prod_{i=1}^m \bigg(\fint_Q w_i^{p_i(\p_i^+/p_i)'} dx \bigg)^{\frac{1}{p_i} - \frac{1}{\p_i^+}}.
\end{align}
Analogously, we have
\begin{align}\label{eq:S2}
\mathscr{S}_2
&:= \bigg(\fint_Q w^{-\frac{p(\p_+/p)'}{\tau_p - 1}} dx \bigg)^{\frac{\tau_p - 1}{p(\p_+/p)'}}
= \bigg(\fint_Q w^{-\frac{1}{\frac{1}{\p_-} - \frac1p}} dx \bigg)^{\frac{1}{\p_-}-\frac1p}
\\ \nonumber
&= \bigg(\fint_Q \Big(\prod_{i=1}^m w_i^{-1} \Big)^s dx \bigg)^{\frac1s}
\le \prod_{i=1}^m \bigg(\fint_Q w_i^{-s_i} dx \bigg)^{\frac{1}{s_i}}
\\ \nonumber
&= \prod_{i=1}^m \bigg(\fint_Q w_i^{-\frac{1}{\frac{1}{\p_i^-} - \frac{1}{p_i}}} dx \bigg)^{\frac{1}{\p_i^-} - \frac{1}{p_i}}
= \prod_{i=1}^m \bigg(\fint_Q w_i^{-\frac{p_i(\p_i^+/p_i)'}{\tau_{p_i} - 1}} dx \bigg)^{\tau_{p_i}(\frac{1}{p_i} - \frac{1}{\p_i^+})}.
\end{align}
Then gathering \eqref{eq:S1} and \eqref{eq:S2}, we arrive at
\begin{align*}
&\bigg(\fint_Q w^{p(\p_+/p)'} dx \bigg)
\bigg(\fint_Q w^{-\frac{p(\p_+/p)'}{\tau_p - 1}} dx \bigg)^{\tau_p - 1}
= (\mathscr{S}_1 \times \mathscr{S}_2 )^{p(\p_+/p)'}
\\
&\quad\le \prod_{i=1}^m \bigg[\bigg(\fint_Q w_i^{p_i(\p_i^+/p_i)'} dx \bigg)
\bigg(\fint_Q w_i^{-\frac{p_i(\p_i^+/p_i)'}{\tau_{p_i} - 1}} dx \bigg)^{\tau_{p_i}-1} \bigg]^{\frac{\frac{1}{p_i} - \frac{1}{\p_i^+}}{\frac1p - \frac{1}{\p_+}}},
\end{align*}
which immediately gives \eqref{eq:wppp}. 

To proceed, we note that by \eqref{eq:ttaauu} and \eqref{eq:wwpp}, 
\begin{equation*}
\tau_{p_{m+1}} 
= \bigg(\frac{\p_{m+1}^+}{p_{m+1}} \bigg)' \bigg(\frac{p_{m+1}}{\p_{m+1}^-} -1 \bigg) +1
= \bigg(\frac{\p'_-}{p'} \bigg)' \bigg(\frac{p'}{\p'_+} -1 \bigg) +1
= \tau'_p, 
\end{equation*} 
which along with Lemma \ref{lem:weights} part \eqref{list:ApRH-3} and \eqref{eq:wppp} yields  
\begin{align*}
[w_{m+1}^{p_{m+1}(\p_{m+1}^+/p_{m+1})'}]_{A_{\tau_{p_{m+1}}}} 
= [w^{-p' (\p'_-/p')'}]_{A_{\tau'_p}} 
= [w^{p (\p_+/p)'}]_{A_{\tau_p}}^{\frac{1}{\tau_p-1}}  
\le \prod_{i=1}^m [w_i^{p_i(\p_i^+/p_i)'}]_{A_{\tau_{p_i}}}^{\frac{\frac{1}{p_i} - \frac{1}{\p_i^+}}{\frac{1}{\p_-} - \frac1p}}. 
\end{align*}
This shows \eqref{eq:wpppp}. 
\end{proof}

\subsection{Multilinear Calder\'{o}n-Zygmund operators}
Given $\delta>0$, we say that a function $K: \R^{n(m+1)} \setminus \{x=y_1=\cdots=y_m\} \to \mathbb{C}$ is a $\delta$-Calder\'{o}n-Zygmund kernel, if there exists a constant $A>0$ such that 
\begin{align*}
|K(x,\vec{y})| & \le  \frac{A}{\big(\sum_{j=1}^{m}|x-y_j|\big)^{mn}},
\\ 
|K(x,\vec{y}) - K(x',\vec{y})| & \le \frac{A \, |x-x'|^{\delta}}{\big(\sum_{j=1}^{m}|x-y_j|\big)^{mn+\delta}}, 
\end{align*}
whenever $|x-x'| \leq \frac12 \max\limits_{1\leq j \leq m}|x-y_j|$, and for each $i=1,\ldots,m$, 
\begin{align*}
|K(x,\vec{y}) - K(x,y_1,\ldots,y'_i,\ldots,y_m)| & \le \frac{A \, |y_i-y'_i|^{\delta}}{\big(\sum_{j=1}^{m}|x-y_j|\big)^{mn+\delta}}, 
\end{align*}
whenever
$|y_i-y_i'| \leq \frac12 \max\limits_{1\leq j \leq m}|x-y_j|$.

An $m$-linear operator $T: \S(\Rn) \times \cdots \times \S(\Rn) \to \S'(\Rn)$ is called a $\delta$-Calder\'{o}n-Zygmund operator if there exists a $\delta$-Calder\'{o}n-Zygmund kernel $K$ such that 
\begin{align*}
T(\vec{f})(x) =\int_{(\Rn)^m} K(x, \vec{y}) f_1(y_1)\cdots f_m(y_m) d\vec{y},  
\end{align*}
whenever $x \not\in \bigcap_{i=1}^m \supp(f_i)$ and $\vec{f}=(f_1,\ldots,f_m) \in \mathscr{C}_c^{\infty}(\Rn) \times \cdots \times \mathscr{C}_c^{\infty}(\Rn)$, and $T$ can be boundedly extended from $L^{q_1}(\Rn) \times \cdots \times L^{q_m}(\Rn)$ to $L^q(\Rn)$ for some $\frac1q=\frac{1}{q_1}+\cdots+\frac{1}{q_m}$ with $1<q_1,\ldots,q_m<\infty$.

Given a symbol $\sigma$, the $m$-linear Fourier multiplier $T_{\sigma}$ is defined by
\begin{align*}
T_{\sigma}(\vec{f})(x) := \int_{(\Rn)^m} \sigma(\vec{\xi})
e^{2\pi i x \cdot (\xi_1+\cdots+\xi_m)} \widehat{f}_1(\xi_1)
\cdots \widehat{f}_m(\xi_m) d\vec{\xi},
\end{align*}
for all $f_i \in \S(\Rn)$, $i=1,\ldots,m$. The operator $T_{\sigma}$ is called an $m$-linear Coifman-Meyer multiplier, if the symbol $\sigma \in \mathscr{C}^s(\R^{nm} \setminus \{0\})$ satisfies 
\begin{align*}
\big|\partial_{\xi_1}^{\alpha_1} \cdots \partial_{\xi_m}^{\alpha_m}  \sigma(\vec{\xi}) \big|
&\leq C_{\alpha} (|\xi_1|+\cdots+|\xi_m|)^{-\sum_{i=1}^m |\alpha_i|}, \quad \forall \vec{\xi} \in \R^{nm} \setminus\{0\}, 
\end{align*}
for each multi-indix $\alpha=(\alpha_1,\ldots,\alpha_m)$ with $|\alpha|=\sum_{i=1}|\alpha_i| \le mn+1$. 

It was shown in \cite[Proposition 6]{GT} that Coifman-Meyer multipliers are examples of multilinear Calder\'{o}n-Zygmund operators. 

Below, the sharp weighted inequality for multilinear Calder\'{o}n-Zygmund operators was given in \cite[Theorem 1.4]{LMS} with $p \ge 1$ and extended to the case $p<1$ in \cite[Corollary 4.4]{Nie}. 
\begin{theorem}\label{thm:CZO}
Let $T$ be an $m$-linear Calder\'{o}n-Zygmund operator. Then for all $1 < p_1, \ldots p_m < \infty$ and $\vec{w} \in A_{\vec{p}}$, 
\begin{align*}
\|T\|_{L^{p_1}(w_1^{p_1}) \times \cdots \times L^{p_m}(w_m^{p_m}) \to L^p(w^p)} 
\lesssim [\vec{w}]_{A_{\vec{p}}}^{\max\{p, p'_1, \ldots, p'_m\}}, 
\end{align*}
where $w=\prod_{i=1}^m w_i$ and $\frac1p=\sum_{i=1}^m \frac{1}{p_i}$.
\end{theorem}

\begin{theorem}[{\cite[Theorem 4.3]{BMMST}}]\label{thm:TTb}
Let $T$ be an $m$-linear operator. Fix $\theta_i >0$ and $r_i \in (1, \infty)$, $i=1, \ldots, m$. Let $\frac1p=\sum_{i=1}^m \frac{1}{p_i} \le 1$ with $1<p_1, \ldots, p_m<\infty$. Assume that there exist increasing functions $\Psi_i : [1, \infty) \to [0, \infty)$ such that for all $v_i^{\theta_i} \in A_{r_i}$, $i=1, \ldots, m$,
\begin{align}
\|T(\vec{f})\|_{L^p(v^p)}
\le \prod_{i=1}^m \Psi_i([v_i^{\theta_i}]_{A_{r_i}}) \|f_i\|_{L^{p_i}(v_i^{p_i})}, 
\end{align}
where $v=\prod_{i=1}^m v_i$. Then, for all weights $w_i^{\eta_i \theta_i} \in A_{s_i}$ with some $\eta_i \in (1, \infty)$, for all $\b = (b_1, \ldots, b_m) \in \BMO^m$, and for each multi-index $\alpha \in \N^m$,
\begin{align}
\|[T, \b]_{\alpha}(\vec{f})\|_{L^p(w^p)}
\le \alpha! \prod_{i=1}^m \delta_i^{-\alpha_i} \Psi_i \big(4^{\delta_i \theta_i} [w_i^{\eta_i \theta_i}]_{A_{r_i}}^{\frac{1}{\eta_i}} \big) \|b_i\|_{\BMO}^{\alpha_i} \|f_i\|_{L^{p_i}(w_i^{p_i})}, 
\end{align}
where $w=\prod_{i=1}^m w_i$, and $\delta_i = \min\{1, r_i-1\}/(\eta'_i \theta_i)$, $i=1, \ldots, m$.
\end{theorem}

Let us record Marcinkiewicz-Zygmund inequalities contained in \cite[Proposition 5.3]{CMO}. 
\begin{lemma}\label{lem:MZ}
Let $0 < p, q_1, \ldots, q_m < r < 2$ or $r=2$ and $0<p, q_1, \ldots, q_m <\infty$.  Let $\mu_1, \ldots, \mu_m$ and $\nu$ be arbitrary $\sigma$-finite measures on $\Rn$. Let $T$ be an $m$-linear operator. Then, there exists a constant $C>0$ such that the following estimates hold: 
\begin{enumerate}
\item[(i)] If $T$ is bounded from $L^{q_1}(\mu_1) \times \cdots \times L^{q_m}(\mu_m)$ to $L^p(\nu)$, then 
\begin{align*}
\bigg\| \bigg(\sum_{k_1, \ldots, k_m} |T(f^1_{k_1}, \ldots, f^m_{k_m})|^r \bigg)^{\frac1r}\bigg\|_{L^p(\nu)}
\le C \|T\| \prod_{i=1}^m \bigg\| \bigg(\sum_{k_i} |f^i_{k_i}|^r \bigg)^{\frac1r}\bigg\|_{L^{q_i}(\mu_i)}, 
\end{align*}
where $\|T\| := \|T\|_{L^{q_1}(\mu_1) \times \cdots \times L^{q_m}(\mu_m) \to L^p(\nu)}$. 

\item[(ii)] If $T$ is bounded from $L^{q_1}(\mu_1) \times \cdots \times L^{q_m}(\mu_m)$ to $L^{p, \infty}(\nu)$, then 
\begin{align*}
\bigg\| \bigg(\sum_{k_1, \ldots, k_m} |T(f^1_{k_1}, \ldots, f^m_{k_m})|^r \bigg)^{\frac1r}\bigg\|_{L^{p, \infty}(\nu)}
\le C \|T\|_{\rm weak} \prod_{i=1}^m \bigg\| \bigg(\sum_{k_i} |f^i_{k_i}|^r \bigg)^{\frac1r}\bigg\|_{L^{q_i}(\mu_i)}, 
\end{align*}
where $\|T\|_{\rm weak} := \|T\|_{L^{q_1}(\mu_1) \times \cdots \times L^{q_m}(\mu_m) \to L^{p, \infty}(\nu)}$. 
\end{enumerate}
\end{lemma}

\section{Quantitative weighted estimates}\label{sec:quant}
The goal of this section is to establish quantitative weighted estimates for (rough) maximal operators and singular integrals. We begin with the following interpolation result with change of measures due to Stein and Weiss \cite{SW}, which plays an important role in dealing with weighted estimates. 

\begin{theorem}[\cite{SW}]\label{thm:SW}
Let $p_0, p_1 \in [1, \infty]$, and let $w_0$ and $w_1$ be weights. If the sublinear operator $T$ satisfies 
\begin{align*}
\|Tf\|_{L^{p_i}(w_i)} \le K_i \|f\|_{L^{p_i}(w_i)}, \quad i=0, 1, 
\end{align*}
then for any $\theta \in (0, 1)$, 
\begin{align*}
\|Tf\|_{L^p(w)} \le K \|f\|_{L^p(w)} \quad\text{with}\quad K \le K_0^{1-\theta} K_1^{\theta}, 
\end{align*}
where $\frac1p=\frac{1-\theta}{p_0} + \frac{\theta}{p_1}$ and $w^{\frac1p}=w_0^{\frac{1-\theta}{p_0}} w_1^{\frac{\theta}{p_1}}$. 
\end{theorem}

The sharp maximal function$M^{\#}$ is defined by  
\begin{align*}
M^{\#}f(x) := \sup_{Q \ni x} \fint_Q |f - f_Q| \, dy, 
\quad\text{ where }\quad f_Q := \fint_Q f\, dy. 
\end{align*}

The following Fefferman-Stein inequality was shown in \cite[Remark 1.9]{CP}. 
\begin{lemma}\label{lem:MM}
For every $p \in (0, \infty)$ and $w \in A_{\infty}$, 
\begin{align*}
\|Mf\|_{L^p(w)} & \le C_{n, p} \, [w]_{A_{\infty}} \|M^{\#}f\|_{L^p(w)}, 
\end{align*}
whenever $Mf \in L^p(w)$ or $f \in L^{\infty}_c(\Rn)$. 
\end{lemma}

We present a sharp weighted vector-valued Fefferman-Stein inequality. 
\begin{lemma}\label{lem:M-vector}
For any $1 < p, r < \infty$ and $w \in A_p$, 
\begin{align*}
\bigg\|\Big(\sum_k |M f_k|^r \Big)^{\frac1r}\Big\|_{L^p(w)} 
\lesssim [w]_{A_p}^{\max\{\frac1r, \frac{1}{p-1}\}} 
\bigg\|\Big(\sum_k |f_k|^r \Big)^{\frac1r}\bigg\|_{L^p(w)}.
\end{align*}
Moreover, the exponent $\max\{\frac1r, \frac{1}{p-1}\}$ is the best possible.
\end{lemma}

\begin{proof}
This inequality was given in \cite[Theorem 1.12]{CMP12}. We here present a different proof. Let $r \in (1, \infty)$. It was proved in \cite[Theorem 1.11]{CLPR} that there exist $3^n$ dyadic lattices $\D_j$ and sparse families $\S_j \subset \D_j$ such that 
\begin{align}\label{Mfk-1}
\Big(\sum_k |M f_k(x)|^r \Big)^{\frac1r} 
\le C_{n, r} \sum_{j=1}^{3^n} \A_{\S_j}^r \bigg(\Big(\sum_k |f_k|^r \Big)^{\frac1r}\bigg)(x), \quad \text{a.e. } x \in \Rn, 
\end{align}
where 
\begin{align}\label{eq:AS}
\A_{\S}^r f(x) := \bigg(\sum_{Q \in \S} \langle |f| \rangle_Q^r \mathbf{1}_Q(x) \bigg)^{\frac1r}. 
\end{align}
It follows from \cite[Theorem2.1]{BH} that for all $p \in (1, \infty)$ and $w \in A_p$, 
\begin{align}\label{Mfk-2}
\|\A_{\S}^{\gamma} f\|_{L^p(w)} 
\le C_{n, p, \gamma} [w]_{A_p}^{\max\{\frac{1}{\gamma}, \frac{1}{p-1}\}} \|f\|_{L^p(w)}, 
\quad \gamma>0. 
\end{align}
Thus, \eqref{Mfk-1} and \eqref{Mfk-2} imply the desired estimate. 
\end{proof}

\begin{lemma}\label{lem:CZO}
Let $\mathscr{B}_1$ and $\mathscr{B}_2$ be Banach spaces, and let $\mathscr{L}(\mathscr{B}_1, \mathscr{B}_2)$ be the Banach space defined by all bounded linear operators from $\mathscr{B}_1$ to $\mathscr{B}_2$ with the operator norm $\|\cdot\|_{\mathscr{L}(\mathscr{B}_1, \mathscr{B}_2)}$. Let $T$ be a linear operator mapping $\mathscr{B}_1$-valued functions into $\mathscr{B}_2$-valued functions satisfying 
\begin{enumerate}
\item[(i)] $T$ is bounded from $L^2(\Rn, \mathscr{B}_1)$ into $L^2(\Rn, \mathscr{B}_2)$. 

\item[(ii)] There exists a kernel function $K(x) \in \mathscr{L}(\mathscr{B}_1, \mathscr{B}_2)$ such that 
\[
\|K(x-y) - K(x)\|_{\mathscr{L}(\mathscr{B}_1, \mathscr{B}_2)} \le C_K |y| |x|^{-n-1}, \quad 2|y| < |x|, 
\]
and for every $f \in L^2(\Rn, \mathscr{B}_1)$ with compact support, 
\[
Tf(x) = \int_{\Rn} K(x-y) f(y) \, dy, \qquad \text{a.e. } x \not\in \supp(f), 
\]
\end{enumerate}
Then for every $p \in (1, \infty)$ and $w \in A_p$, 
\[
\|Tf\|_{L^p(w, \mathscr{B}_2)} 
\lesssim [w]_{A_p}^{\frac72\max\{1, \frac{1}{p-1}\}} \|f\|_{L^p(w, \mathscr{B}_1)}. 
\]
\end{lemma}

\begin{proof}
It was shown in \cite[p. 41--42]{RRT} that 
\begin{align}
\label{CZO-1} & \|Tf\|_{L^{1, \infty}(\Rn, \mathscr{B}_2)} 
\lesssim C_T \, \|f\|_{L^1(\Rn, \mathscr{B}_1)}, 
\\
\label{CZO-2} & M^{\#}(\|Tf\|_{\mathscr{B}_2})(x) 
\lesssim C_T \, M_r(\|f\|_{\mathscr{B}_1})(x), \quad 2 \le r<\infty, \, x \in \Rn, 
\end{align}
where $C_T := \|T\|_{L^2(\Rn, \mathscr{B}_1) \to L^2(\Rn, \mathscr{B}_2)} + C_K$. Then interpolating between the assumption (i) and \eqref{CZO-2} yields that for any $1<r_0<r<2$, 
\begin{align}\label{CZO-3} 
\|Tf\|_{L^{r_0}(\Rn, \mathscr{B}_2)} 
&\lesssim C_T \, (r_0-1)^{-\frac{1}{r_0}} \|f\|_{L^{r_0}(\Rn, \mathscr{B}_1)}, 
\end{align}
where the implicit constant is independent of $r_0$. As argued in the proof of \eqref{CZO-2}, the inequality \eqref{CZO-3} implies for any $1<r_0<r<2$, 
\begin{align}\label{CZO-4} 
M^{\#}(\|Tf\|_{\mathscr{B}_2})(x) 
&\lesssim C_T \, (r_0-1)^{-\frac{1}{r_0}} M_r(\|f\|_{\mathscr{B}_1})(x), \quad x \in \Rn, 
\end{align}

Now let $p \in (1, \infty)$ and $w \in A_p$. Then, by Lemma \ref{lem:open}, there exists $\gamma \in (0, 2^{-n-3})$ and $ q_0 \in(1, p)$ such that $q_0 = \frac{p}{1+\varepsilon}$, $\frac{p-1}{p(1+\gamma)'}<\varepsilon<\frac{p-1}{(1+\gamma)'}$, $(1+\gamma)' \simeq [w]_{A_p}^{\max\{1, \frac{1}{p-1}\}}$, and $[w]_{A_{q_0}} \leq 2^p [w]_{A_p}$. Set $r:=p/q_0=1+\varepsilon$. If $r \ge 2$, it follows from Lemma \ref{lem:MM} and \eqref{CZO-2} that 
\begin{align*}
\|T f\|_{L^p(w, \mathscr{B}_2)}
&\le \|M(\|T f\|_{\mathscr{B}_2})\|_{L^p(w)}
\lesssim [w]_{A_p} \|M^{\#}(\|Tf\|_{\mathscr{B}_2})\|_{L^p(w)}
\\
&\lesssim [w]_{A_p} \|M_r(\|f\|_{\mathscr{B}_1})\|_{L^p(w)}
\lesssim [w]_{A_p} [w]_{A_{q_0}}^{\frac{1}{r(q_0-1)}} \|f\|_{L^p(w, \mathscr{B}_1)}
\\
&\lesssim [w]_{A_p}^{1+ \frac{1}{p-r}} \|f\|_{L^2(w, \mathscr{B}_1)} 
\le [w]_{A_p}^{\frac52 \max\{1, \frac{1}{p-1}\}} \|f\|_{L^p(w, \mathscr{B}_1)}, 
\end{align*}
since 
\begin{align*}
p-r=p-1-\varepsilon
>p-1-\frac{p-1}{(1+\gamma)'}
=\frac{p-1}{1+\gamma}
>\frac{p-1}{1+2^{-n-3}}
>\frac{p-1}{3/2}. 
\end{align*}
If $1<r<2$, we choose $r_0=1+\frac{p-1}{p(1+\gamma)'}$ and invoke Lemma \ref{lem:MM} and \eqref{CZO-4} to obtain that 
\begin{align*}
\|T f\|_{L^p(w, \mathscr{B}_2)}
&\lesssim \frac{1}{r_0-1} [w]_{A_p} \|M_r(\|f\|_{\mathscr{B}_1})\|_{L^p(w)}
\\
&\lesssim \frac{p(1+\gamma)'}{p-1} [w]_{A_p}^{1 + \frac{1}{p-r}} \|f\|_{L^p(w, \mathscr{B}_1)} 
\\
&\lesssim [w]_{A_p}^{\frac{7}{2} \max\{1, \frac{1}{p-1}\}} \|f\|_{L^p(w, \mathscr{B}_1)}. 
\end{align*}
This completes the proof. 
\end{proof}

\begin{lemma}\label{lem:phif}
Let $\varphi \in \S(\Rn)$ be such that $\int_{\Rn} \varphi \, dx=0$ and $\supp(\widehat{\varphi}) \subset \{\xi \in \Rn : c_1 \le |\xi| \le c_2\}$ for some $0<c_1<c_2 < \infty$. Set $\varphi_k(x) := 2^{kn} \varphi (2^k x)$ for any $k \in \Z$. Then for every $p \in(1, \infty)$ and $w \in A_p$, 
\begin{align}
\label{eq:phik-1} & \bigg\|\Big(\sum_{k \in \Z} |\varphi_k * f|^2\Big)^{\frac12}\bigg\|_{L^p(w)} 
\lesssim [w]_{A_p}^{\max\{\frac12, \frac{1}{p-1}\}} \|f\|_{L^p(w)}, 
\\
\label{eq:phik-33} & \bigg\|\sum_{k \in \Z} \varphi_k * f_k \bigg\|_{L^p(w)} 
\lesssim [w]_{A_p}^{\frac72 \max\{1, \frac{1}{p-1}\}}  
\bigg\|\Big(\sum_{k \in \Z} |f_k|^2\Big)^{\frac12}\bigg\|_{L^p(w)}.  
\\ 
\label{eq:phik-3} & \bigg\| \Big(\sum_{k \in \Z} |\varphi_k * f_k|^2\Big)^{\frac12} \bigg\|_{L^p(w)} 
\lesssim [w]_{A_p}^{\max\{\frac12, \frac{1}{p-1}\}} 
\bigg\|\Big(\sum_{k \in \Z} |f_k|^2\Big)^{\frac12}\bigg\|_{L^p(w)}.  
\end{align}
If we assume in addition that $\sum_{k \in \Z} |\widehat{\varphi}(2^{-k} \xi)|^2 = C_{\varphi}>0$ for all $\xi \neq 0$, then 
\begin{align}\label{eq:phik-2} 
\|f\|_{L^p(w)} \lesssim [w]_{A_p}^{\max\{1, \frac{1}{2(p-1)}\}} 
\bigg\|\Big(\sum_{k \in \Z} |\varphi_k * f|^2\Big)^{\frac12}\bigg\|_{L^p(w)}. 
\end{align}
\end{lemma}

\begin{proof}
Since $\varphi \in \S(\Rn)$, one can check that there exists $C'_{\varphi}>0$ such that for any $\beta \in (0, 1]$ and any $y \in \Rn$, $|\varphi(x)| \le C'_{\varphi} (1+|x|)^{-n-\beta}$ and
\begin{align*}
|\varphi(x-y) - \varphi(x)| 
&\le C'_{\varphi} \bigg[\frac{|y|^{\beta}}{(1+|x|)^{n+1+\beta}} + \frac{|y|^{\beta}}{(1+|x-y|)^{n+1+\beta}}\bigg]. 
\end{align*}
Recalling that $\int_{\Rn} \varphi \, dx=0$, we see that $\varphi/C'_{\varphi} \in \mathcal{C}_{\beta, 1}$, which is defined in \cite[Definition 6.2]{Wil08}. Then by \cite[Theorem 6.3]{Wil08}, 
\begin{align}\label{eq:GGf}
S_{\varphi} f(x) 
:= \Big(\sum_{k \in \Z} |\varphi_k * f(x)|^2\Big)^{\frac12} 
\le \widetilde{\sigma}_{\beta, 1} f(x) 
\lesssim \sigma_{\beta}f(x) 
\lesssim G_{\beta} f(x), \quad x \in \Rn, 
\end{align}
where the implicit constant is independent of $f$ and $x$. Thus, \eqref{eq:phik-1} follows from \eqref{eq:GGf} and the sharp weighted estimate for $G_{\beta}$ in \cite[Theorem 1.1]{Ler11}.

To show \eqref{eq:phik-33}, we will use vector-valued singular integrals. By the support of $\widehat{\varphi}$, there exist $j_0, j_1 \in \N$ such that $\supp(\widehat{\varphi}_{j+k}) \cap \supp(\widehat{\varphi}_j) = \emptyset$ whenever $k \le -j_0-1$ or $k \ge j_1+1$. This and Plancherel's identity give 
\begin{align*}
\int_{\Rn} & \Big|\sum_{k \in \Z} \varphi_k*f_k(x) \Big|^2 dx
=\int_{\Rn} \Big|\sum_{k \in \Z} \widehat{\varphi}_k(\xi) \widehat{f}_k(\xi) \Big|^2 d\xi 
\\ 
&=\sum_{k, j \in \Z} \int_{\Rn} \widehat{\varphi}_k(2^{-k}\xi) \widehat{f}_k(\xi) \widehat{\varphi}_j(2^{-j}\xi) \overline{\widehat{f}_k(\xi)} d\xi 
\\ 
&=\sum_{j \in \Z} \sum_{k=j-j_0}^{j+j_1} \int_{\Rn} \widehat{\varphi}_k(2^{-k}\xi) \widehat{f}_k(\xi) \widehat{\varphi}_j(2^{-j}\xi) \overline{\widehat{f}_k(\xi)} d\xi 
\\ 
&=\sum_{j \in \Z} \sum_{k=-j_0}^{j_1} \int_{\Rn} \widehat{\varphi}_{j+k}(2^{-j-k}\xi) 
\widehat{\varphi}_j(2^{-j}\xi) \widehat{f}_{j+k}(\xi)  \overline{\widehat{f}_k(\xi)} d\xi 
\\ 
&\lesssim \sum_{k=-j_0}^{j_1} \int_{\Rn} \sum_{j \in \Z} |\widehat{f}_{j+k}(\xi)|  |\widehat{f}_k(\xi)| d\xi 
\\ 
&\lesssim \sum_{k=-j_0}^{j_1} \bigg(\int_{\Rn} \sum_{j \in \Z} |\widehat{f}_{j+k}(\xi)|^2 d\xi\bigg)^{\frac12} 
\bigg(\int_{\Rn} \sum_{j \in \Z} |\widehat{f}_j(\xi)|^2 d\xi\bigg)^{\frac12}
\\
&\lesssim \|\{f_k\}_{k \in \Z}\|_{L^2(\Rn, \ell^2)}^2. 
\end{align*}
This means that the operator $T$ defined by $T(\{f_k\}_{k \in \Z}) := \sum_{k \in \Z} \varphi_k*f_k$, is a bounded linear operator from $L^2(\Rn, \ell^2)$ to $L^2(\Rn)$, with the kernel $K(x)=\{\varphi_k(x)\}_{k \in \Z}$ satisfying $\|\nabla K(x)\|_{\mathscr{L}(\ell^2, \mathbb{C})} \lesssim |x|^{-n-1}$ for all $x \neq 0$. Hence, Lemma \ref{lem:CZO} implies \eqref{eq:phik-33}. 

Note that the inequality \eqref{eq:phik-3} is a consequence of Lemma \ref{lem:M-vector} and that $|\varphi_k * f_k| \lesssim Mf_k$ uniformly in $k \in \Z$. 

Finally, to get \eqref{eq:phik-2}, we use Parseval's identity and $\sum_{k \in \Z} |\widehat{\varphi}_k(\xi)|^2=\sum_{k \in \Z} |\widehat{\varphi}(2^{-k} \xi)|^2 = C_{\varphi}$ to get that for any $f, g \in L^2(\Rn)$, 
\begin{align*}
\int_{\Rn} \sum_{k \in \Z} \varphi_k*f(x) \, \varphi_k*g(x) \, dx 
=C_{\varphi} \int_{\Rn} f(x) g(x) \, dx. 
\end{align*}
Then it follows that for $g \in \S(\Rn)$ with $\|g\|_{L^{p'}(w^{1-p'})}=1$, 
\begin{align*}
\bigg|\int_{\Rn} f(x) g(x) \, dx\bigg|
&\lesssim \bigg|\int_{\Rn} \sum_{k \in \Z} \varphi_k*f(x) \, \varphi_k*g(x) \, dx\bigg|
\\ 
&\le \int_{\Rn} S_{\varphi}f(x) S_{\varphi} g(x) \, dx 
\le \|S_{\varphi}f\|_{L^p(w)} \|S_{\varphi} g\|_{L^{p'}(w^{1-p'})}
\\
&\lesssim [w^{1-p'}]_{A_{p'}}^{\max\{\frac12, \frac{1}{p'-1}\}} \|S_{\varphi}f\|_{L^p(w)} 
=[w]_{A_p}^{\max\{1, \frac{1}{2(p-1)}\}} \|S_{\varphi}f\|_{L^p(w)}, 
\end{align*}
where we have used \eqref{eq:phik-1} in the last inequality. This gives at once \eqref{eq:phik-2}. 
\end{proof}

\begin{lemma}\label{lem:Mar}
Given $\varepsilon>0$ and a pairwise disjoint family of cubes $\{Q_j\}$, we set 
\begin{align}\label{Me}
\Omega := \bigcup_j Q_j 
\quad\text{ and }\quad 
\mathfrak{M}_{\varepsilon}(x) 
:= \sum_j \frac{\ell(Q_j)^{n+\varepsilon}}{|x-x_{Q_j}|^{n+\varepsilon} + \ell(Q_j)^{n+\varepsilon}}, \quad x \in \Rn.
\end{align} 
Then $\|\mathfrak{M}_{\varepsilon}\|_{L^2(w)} \lesssim [w]_{A_2} w(\Omega)^{\frac12}$ for any $w \in A_2$.
\end{lemma}

\begin{proof}
Note that 
\begin{align*}
\mathfrak{M}_{\varepsilon}(x)
\lesssim \sum_j \bigg[\frac{1}{(|x-x_{Q_j}|/\ell(Q_j))^n +1} \bigg]^{n+\varepsilon}
\lesssim \sum_j M\mathbf{1}_{Q_j}(x)^{\frac{n+\varepsilon}{n}}, 
\end{align*}
which together with Lemma \ref{lem:M-vector} gives that for any $w \in A_2$, 
\begin{align*}
&\bigg(\int_{\Rn} \mathfrak{M}_{\varepsilon}(x)^2\, w(x) \, dx \bigg)^{\frac{n}{2(n+\varepsilon)}}
\lesssim \bigg(\int_{\Rn} \Big(\sum_j M\mathbf{1}_{Q_j}(x)^{\frac{n+\varepsilon}{n}} \Big)^{\frac{n}{n+\varepsilon} \cdot \frac{2(n+\varepsilon)}{n}} w(x) \, dx \bigg)^{\frac{n}{2(n+\varepsilon)}} 
\\
&\quad \lesssim [w]_{A_{\frac{2(n+\varepsilon)}{n}}}^{\max\{\frac{n}{n+\varepsilon}, \frac{1}{\frac{2(n+\varepsilon)}{n}-1}\}}
\bigg(\int_{\Rn} \Big(\sum_j \mathbf{1}_{Q_j}(x) \Big)^{2} w(x) \, dx \bigg)^{\frac{n}{2(n+\varepsilon)}}
\le [w]_{A_2}^{\frac{n}{n+\varepsilon}} w(\Omega)^{\frac{n}{2(n+\varepsilon)}}, 
\end{align*}
where we have use the disjointness of $\{Q_j\}$. This implies the desired estimate. 
\end{proof}

Given $\Omega \in L^1(\Sn)$, the rough maximal operator $M_{\Omega}$ and singular integral $T_{\Omega}$ are defined by
\begin{align}
M_{\Omega} f(x) &:= \sup_{r>0} \fint_{B(0, r)} |\Omega(y')| |f(x-y)| \, dy, 
\\
T_{\Omega} f(x) &:= \mathrm{p. v. } \int_{\Rn} \frac{\Omega(y')}{|y|^n}f(x-y) \, dy. 
\end{align}

\begin{theorem}\label{thm:Tom}
Let $q \in (1, \infty)$ and $\Omega \in L^q(\Sn)$ be such that $\int_{\Sn} \Omega \, d\sigma=0$. Then for all $p \in (q', \infty)$ and for all $w \in A_{p/q'}$, 
\begin{align}
\label{MOTO-1} \|M_{\Omega}f\|_{L^p(w)}
&\lesssim [w]_{A_{p/q'}}^{\frac{1}{p-q'}} \|f\|_{L^p(w)}, 
\\
\label{MOTO-2} \|T_{\Omega}f\|_{L^p(w)}
&\lesssim [w]_{A_{p/q'}}^{\max\{1, \frac{1}{p/q'-1}\} + \max\{1, \frac{1}{p-q'}\}} \|f\|_{L^p(w)}. 
\end{align}
Moreover, the vector-valued inequality holds for $q>2$: 
\begin{align}
\label{MOTO-3} \bigg\|\Big(\sum_j |M_{\Omega}f_j|^2\Big)^{\frac12}\bigg\|_{L^p(w)}
&\lesssim [w]_{A_{p/q'}}^{\frac{1}{p-q'}} \bigg\|\Big(\sum_j |f_j|^2\Big)^{\frac12}\bigg\|_{L^p(w)}. 
\end{align}
\end{theorem}

\begin{proof}
By definition and H\"{o}lder's inequality, one has 
\begin{align*}
M_{\Omega}f(x) \le \|\Omega\|_{L^q(\Sn)} \, M_{q'}f(x), \quad x \in \Rn, 
\end{align*}
which together with \eqref{eq:sharp} immediately gives \eqref{MOTO-1}. Then \eqref{MOTO-3} is a consequence of \eqref{MOTO-1}, Theorem \ref{thm:lim}, and Remark \ref{rem:pp}.

To treat \eqref{MOTO-2}, we choose a radial nonnegative function $\varphi \in \mathscr{C}_c^{\infty}(\Rn)$ such that $\supp \varphi \subset \{|x|< 1/4\}$ and $\int_{\Rn} \varphi \, dx =1$. Set $\varphi_j(x) := 2^{-nj} \varphi(2^{-j} x)$ and $\nu_j(x) := \frac{\Omega(x')}{|x|^n} \mathbf{1}_{\{2^j \le |x|<2^{j+1}\}}(x)$ for each $j \in \Z$. Define 
\[
T_j f := K_j * f \quad\text{ and }\quad 
K_j := \sum_{k \in \Z} \nu_k * \varphi_{k-j}, \quad j \in \Z. 
\]
Then, 
\begin{align}\label{TjTj}
T_{\Omega} = T_1 + \sum_{j=1}^{\infty} (T_{j+1} - T_j). 
\end{align}
It was proved in \cite[p. 396]{Wat} that for some $\delta_0>0$, 
\begin{align}
\label{TTL2-1} \|T_j f\|_{L^2(\Rn)} \lesssim \|f\|_{L^2(\Rn)}, \quad j \ge 1, 
\\
\label{TTL2-2} \|(T_{j+1} - T_j) f\|_{L^2(\Rn)} \lesssim 2^{-\delta_0 j} \|f\|_{L^2(\Rn)}, \quad j \ge 1, 
\end{align}
where the implicit constants are independent of $j$. 

On the other hand, it follows from \cite[Lemma 2]{Wat} that 
\begin{equation}\label{eq:LrHor}
\text{$K_j$ satisfies the $L^q$-H\"{o}rmander condition},  
\end{equation}
which together with \cite[Theorem 2]{Wat} gives 
\begin{align}
\label{Tj-weak} & T_j \text{ is bounded from $L^1(\Rn)$ to $L^{1, \infty}(\Rn)$}, 
\\
\label{Tj-Lp} & T_j \text{ is bounded on } L^p(\Rn), \quad 1<p<\infty. 
\end{align}
In particular, \eqref{Tj-Lp} implies 
\begin{align}\label{TjLr}
T_j \text{ is bounded from $L^{q'}(\Rn)$ to $L^{q', \infty}(\Rn)$}, \quad 1<q<\infty, 
\end{align}
and the interpolation theorem, \eqref{TTL2-2}, and \eqref{Tj-Lp} yield that for some $\delta>0$, 
\begin{align}\label{TTLp} 
\|(T_{j+1} - T_j) f\|_{L^p(\Rn)} \lesssim 2^{-\delta j} \|f\|_{L^p(\Rn)}, \quad j \ge 1, \, 1<p<\infty. 
\end{align}
Hence, by \eqref{eq:LrHor}, \eqref{TjLr}, and \cite[Theorem 1.2]{Li}, we obtain that for any $f \in L^{\infty}_c(\Rn)$ there exists a sparse family $\S_j$ such that 
\begin{align}\label{TjAS}
T_j f(x) 
\lesssim \sum_{Q \in \S_j} \langle |f|^{q'} \rangle_Q^{\frac{1}{q'}} \mathbf{1}_Q(x)
=\A_{\S_j}^{\frac{1}{q'}} (|f|^{q'})(x)^{\frac{1}{q'}}, \quad\text{a.e. } x \in \Rn, 
\end{align}
where the dyadic operator $\A_{\S}^{\gamma}$ is defined in \eqref{eq:AS} and the implicit constant is independent of $j$. Accordingly, we use \eqref{Mfk-2}, \eqref{TjAS}, and a density argument to arrive at 
\begin{align}\label{TjLpv}
\|T_j f\|_{L^p(v)} 
\lesssim [v]_{A_{p/q'}}^{\max\{1, \frac{1}{p-q'}\}} \|f\|_{L^p(v)}, 
\quad\forall p \in (q', \infty), \, v \in A_{p/q'}. 
\end{align}

Now fix $p \in (q', \infty)$ and $w \in A_{p/q'}$. By Lemma \ref{lem:open}, there exists $\gamma \in (0, 1)$ such that 
\[
(1+\gamma)' = c_n [w]_{A_{p/q'}}^{\max\{1, \frac{1}{p/q'-1}\}} =: c_n B_0, 
\quad\text{and}\quad
[w^{1+\gamma}]_{A_{p/q'}} \lesssim [w]_{A_{p/q'}}^{1+\gamma},
\] 
which along with \eqref{TjLpv} implies 
\begin{align}\label{TTjLpw}
\|(T_{j+1} -T_j)f\|_{L^p(w^{1+\gamma})} 
\lesssim [w]_{A_{p/q'}}^{(1+\gamma)\max\{1, \frac{1}{p-q'}\}} \|f\|_{L^p(w^{1+\gamma})}. 
\end{align}
In light of Theorem \ref{thm:SW} with $w_0 \equiv 1$, $w_1=w^{1+\gamma}$, and $\theta=\frac{1}{1+\gamma}$, interpolating between \eqref{TTLp} and \eqref{TTjLpw} gives 
\begin{align}\label{TjLpwee}
\|(T_{j+1} -T_j) f\|_{L^p(w)} 
\lesssim 2^{-(1-\theta)\delta j} [w]_{A_{p/q'}}^{\max\{1, \frac{1}{p-q'}\}} \|f\|_{L^p(w)}, \quad j \ge 1. 
\end{align}
Note that $1-\theta = \frac{1}{(1+\gamma)'}$ and $e^{-t} < 2t^{-2}$ for any $t >0$. As a consequence, \eqref{TjTj}, \eqref{TjLpv}, and \eqref{TjLpwee} imply 
\begin{align*}
\|T_{\Omega} f\|_{L^p(w)} 
&\lesssim \sum_{j=0}^{\infty} 2^{-\frac{c'_n j}{B_0}} [w]_{A_{p/q'}}^{\max\{1, \frac{1}{p-q'}\}} \|f\|_{L^p(w)}  
\\
&= \bigg(\sum_{j \le B_0} + \sum_{j > B_0} \bigg) 2^{-\frac{c'_n j}{B_0}} [w]_{A_{p/q'}}^{\max\{1, \frac{1}{p-q'}\}} \|f\|_{L^p(w)}  
\\
&\lesssim \bigg(B_0 + \sum_{j > B_0} j^{-2}B_0^2 \bigg)  
[w]_{A_{p/q'}}^{\max\{1, \frac{1}{p-q'}\}} \|f\|_{L^p(w)}  
\\
&\lesssim B_0 \, [w]_{A_{p/q'}}^{\max\{1, \frac{1}{p-q'}\}} \|f\|_{L^p(w)}  
\\
&= [w]_{A_{p/q'}}^{\max\{1, \frac{1}{p/q'-1}\} + \max\{1, \frac{1}{p-q'}\}} \|f\|_{L^p(w)}.   
\end{align*}
This completes the proof. 
\end{proof}

\begin{theorem}\label{thm:MTT}
Let $q \in (1, \infty)$ and $\Omega \in L^q(\Sn)$. Then for all $p \in (1, q)$ and for all $w^{1-p'} \in A_{p'/q'}$, 
\begin{align}\label{MTT} 
\|M_{\Omega}f\|_{L^p(w)}
\lesssim [w^{1-p'}]_{A_{p'/q'}}^{\max\{1, \frac{1}{p'/q'-1}\} + \max\{1, \frac{1}{p'-q'}\}} \|f\|_{L^p(w)}. 
\end{align}
\end{theorem}

\begin{proof}
Fix $p \in (1, q)$ and $w^{1-p'} \in A_{p'/q'}$. For $j \in \Z$,  set $\nu_{\Omega, j}(x) := \frac{\Omega(x')}{|x|^n} \mathbf{1}_{\{2^j \le |x|<2^{j+1}\}}(x)$. Define 
\begin{align*}
\mathbf{S}_{\Omega}f(x) := \bigg(\sum_{j \in \Z} |\mathbf{T}_{\Omega, j} f(x)|^2 \bigg)^{\frac12}, 
\quad\text{where}\quad 
\mathbf{T}_{\Omega, j} f := \nu_{\Omega, j}*f. 
\end{align*}
If we set $\Omega_0(x') := |\Omega(x')| - \fint_{\Sn} |\Omega| \, d\sigma$ for any $x' \in \Sn$, there there holds 
\begin{align}\label{MMS}
\Omega_0 \in L^q(\Sn), \quad \int_{\Sn} \Omega_0 \, d\sigma =0, \quad\text{and}\quad 
M_{\Omega} f
\lesssim Mf + \mathbf{S}_{\Omega_0}(|f|). 
\end{align}
Since $w^{1-p'} \in A_{p'/q'} \subset A_{p'}$, we see that $w \in A_p$ and by \eqref{eq:sharp}, 
\begin{align}\label{MMPP}
\|Mf\|_{L^p(w)} 
\lesssim [w]_{A_p}^{\frac{1}{p-1}} \|f\|_{L^p(w)}
= [w^{1-p'}]_{A_{p'}} \|f\|_{L^p(w)}
\le [w^{1-p'}]_{A_{p'/q'}} \|f\|_{L^p(w)}. 
\end{align}

In order to estimate $\mathbf{S}_{\Omega_0}$, we define a linear operator 
\[
\mathbf{T}_{\Omega_0}^{\varepsilon} 
:= \sum_{m \in \Z} \varepsilon_m \mathbf{T}_{\Omega_0, m}, 
\quad\text{ where } \, \varepsilon := \{\varepsilon_m = \pm 1\}.
\]
Writing 
\[
k := \sum_{m \in \Z} \varepsilon_m \nu_{\Omega_0, m} 
\quad\text{ and }\quad 
k^{(m)} := \nu_{\Omega_0, m}, \quad m\in \Z, 
\]
one can verify that \cite[Lemmas 1 and 2]{Wat} hold for $k$ and $k^{(m)}$, with bounds independent of $\varepsilon$. This means that $\mathbf{T}_{\Omega_0}^{\varepsilon}$ behaves as $T_{\Omega}$ in Theorem \ref{thm:Tom}. Then by \eqref{MOTO-2}, 
\begin{align}\label{supvar-1}
\sup_{\varepsilon} \|\mathbf{T}_{\Omega_0}^{\varepsilon} f\|_{L^s(v)}
&\lesssim [v]_{A_{s/q'}}^{\max\{1, \frac{1}{s/q'-1}\} + \max\{1, \frac{1}{s-q'}\}} \|f\|_{L^s(v)}. 
\end{align}
for any $s \in (q', \infty)$ and $v \in A_{s/q'}$. By duality, \eqref{supvar-1} implies 
\begin{align}\label{supvar-2}
\sup_{\varepsilon} \|\mathbf{T}_{\Omega_0}^{\varepsilon} f\|_{L^p(w)}
&\lesssim [w^{1-p'}]_{A_{p'/q'}}^{\max\{1, \frac{1}{p'/q'-1}\} + \max\{1, \frac{1}{p'-q'}\}} \|f\|_{L^p(w)}. 
\end{align}
We would like to use \eqref{supvar-2} to bound $\mathbf{S}_{\Omega_0}$. Let $\{r_m(\cdot)\}_{m \in \N}$ be the system of Rademacher functions in $[0, 1)$. By Khintchine's inequality (cf. \cite[p. 586]{Gra}) and \eqref{supvar-2} applied to $\varepsilon(t) := \{r_m(t)\}_{m \in \Z}$, we have 
\begin{align}\label{Som}
&\|\mathbf{S}_{\Omega_0} f\|_{L^p(w)}
\simeq \bigg\|\bigg(\int_0^1 \Big|\sum_{m \in \Z} r_m(t) \mathbf{T}_{\Omega_0, m} f\Big|^p dt \bigg)^{\frac1p}\bigg\|_{L^p(w)}
\\ \nonumber 
&\quad= \bigg(\int_0^1 \|\mathbf{T}_{\Omega_0}^{\varepsilon(t)} f\|_{L^p(w)}^p dt \bigg)^{\frac1p}
\lesssim [w^{1-p'}]_{A_{p'/q'}}^{\max\{1, \frac{1}{p'/q'-1}\} + \max\{1, \frac{1}{p'-q'}\}} \|f\|_{L^p(w)}.
\end{align}
Therefore, \eqref{MTT} follows from \eqref{MMS}, \eqref{MMPP}, and \eqref{Som}.
\end{proof}

\begin{lemma}\label{lem:GH}
Let $\psi \in \mathscr{C}_c^{\infty}(\Rn)$ be a radial function such that $0 \le \psi \le 1$, $\supp \psi \subset\{1/2 \le |\xi| \le 2\}$ and $\sum_{l \in \Z} \psi (2^{-l}\xi)^2 =1$ for $|\xi| \neq 0$. Define the multiplier $\Delta_l$ by $\widehat{\Delta_l f}(\xi) = \psi(2^{-l}\xi) \widehat{f}(\xi)$. For $j \in \Z$,  set $\nu_j(x) := \frac{\Omega(x')}{|x|^n} \mathbf{1}_{\{2^j\le |x|<2^{j+1}\}}(x)$, where $\Omega$ is the same as in Theorem \ref{thm:Tom}. Then for all $p \in (q', \infty)$ and $w \in A_{p/q'}$, 
\begin{equation}\label{vukk}
\sup_{s \in \Z} \bigg\|\Big(\sum_{k \in \Z} |\nu_{k+s} \ast \Delta_{l-k}^2 f|^2\Big)^{\frac12}\bigg\|_{L^p(w)} 
\lesssim [w]_{A_{p/q'}}^{\frac52 \max\{1, \frac{1}{p/q'-1}, \frac{2}{p-1}\}}  \|f\|_{L^p(w)}.
\end{equation}
\end{lemma}

\begin{proof}
Let $p \in (q', \infty)$ and $w \in A_{p/q'}$. Observe that 
\begin{equation}\label{supnu}
\sup_{s \in \Z} \sup_{k \in \Z} |\nu_{k+s} \ast f_k| 
\le M_{\Omega}\Big(\sup_{k \in \Z} |f_k| \Big). 
\end{equation}
This and \eqref{MOTO-1} yield 
\begin{equation}\label{vksm-1}
\sup_{s \in \Z} \Big\|\sup_{k \in \Z} |\nu_{k+s} \ast f_k| \Big\|_{L^p(w)}
\lesssim [w]_{A_{p/q'}}^{\frac{1}{p-q'}} \Big\| \sup_{k \in \Z} |f_k| \Big\|_{L^p(w)}. 
\end{equation}
In light of Theorem \ref{thm:MTT}, \eqref{supnu} implies that for any $r \in (1, q)$ and $v^{1-r'} \in A_{r'/q'}$, 
\begin{equation*}
\sup_{s \in \Z} \Big\|\sup_{k \in \Z} |\nu_{k+s} \ast f_k| \Big\|_{L^r(v)}
\lesssim [v^{1-r'}]_{A_{r'/q'}}^{\max\{1, \frac{1}{r'/q'-1}\} + \max\{1, \frac{1}{r'-q'}\}} 
\Big\| \sup_{k \in \Z} |f_k| \Big\|_{L^r(v)}, 
\end{equation*}
which together with duality gives 
\begin{equation}\label{vksm-2}
\sup_{s \in \Z} \bigg\|\sum_{k \in \Z} |\nu_{k+s} \ast f_k| \bigg\|_{L^p(w)} 
\lesssim [w]_{A_{p/q'}}^{\max\{1, \frac{1}{p/q'-1}\} + \max\{1, \frac{1}{p-q'}\}} 
\bigg\| \sum_{k \in \Z} |f_k| \bigg\|_{L^p(w)}. 
\end{equation}
Then, interpolating between \eqref{vksm-1} and \eqref{vksm-2}, we obtain 
 \begin{equation*}
\sup_{s \in \Z} \bigg\|\Big(\sum_{k \in \Z} |\nu_{k+s} \ast f_k|^2 \Big)^{\frac12} \bigg\|_{L^p(w)} 
\lesssim [w]_{A_{p/q'}}^{\frac12 \max\{1, \frac{1}{p/q'-1}\} + \max\{1, \frac{1}{p-q'}\}}  
\bigg\| \Big(\sum_{k \in \Z} |f_k|^2 \Big)^{\frac12}\bigg\|_{L^p(w)}. 
\end{equation*}
Combining \eqref{eq:phik-3} with \eqref{eq:phik-1} and that $[w]_{A_p} \le [w]_{A_{p/q'}}$, this immediately implies \eqref{vukk}. 
\end{proof}

\section{Proof of main theorems}\label{sec:proof}
In this section, we will prove Theorems \ref{thm:lim} and \ref{thm:lim-Tb}. The first step is to show Theorem \ref{thm:lim}, which will follow from Theorem \ref{thm:lim-off}, a limited rang, off-diagonal extrapolation with quantitative weights norms. Before proving the latter, we present some other quantitative extrapolation.

\subsection{$A_p$ extrapolation}
We begin with the $A_p$ extrapolation with quantitative bounds.  

\begin{theorem}\label{thm:Ap}
Let $\F$ be a family of extrapolation pairs. Assume that there exist exponents $p_0 \in [1, \infty]$ such that for all weights $v^{p_0} \in A_{p_0}$, 
\begin{equation}\label{eq:Ap-1}
\|f v\|_{L^{p_0}} \leq \Phi ([v^{p_0}]_{A_{p_0}})  \|g v\|_{L^{p_0}}, \quad (f, g) \in \F,
\end{equation}
where $\Phi : [1, \infty) \to [1, \infty)$ is an increasing function. Then for all exponents $p \in (1, \infty)$ and all weights $w^p \in A_p$, 
\begin{equation}\label{eq:Ap-2}
\|f w\|_{L^p} 
\leq 2 \Phi \Big(C_p \, [w^p]_{A_p}^{\max\{1, \frac{p_0 - 1}{p -1}\}}\Big) 
\|g w\|_{L^p}, \quad (f, g) \in \F, 
\end{equation}
where $C_p=3^{n(p'+8)(p_0-p)}$ if $p<p_0$, and $C_p=3^{n(p+8)}$ if $p>p_0$. 
\end{theorem}

Theorem \ref{thm:Ap} was shown in \cite{DGPP, Duo} without the explicit constant $C_p$.  We restudy it by presenting a stronger result as follows.  

\begin{theorem}\label{thm:Apvw}
Let $q \in [1, \infty]$ and $p \in (1, \infty)$. Then for any $w^p \in A_p$, $f \in L^p(w^p)$, and $g \in L^{p'}(w^{-p'})$, there exists $v^q \in A_q$ with $[v^q]_{A_q} \le C_p \, [w^p]_{A_p}^{\max\{1, \frac{q-1}{p-1}\}}$ such that 
\begin{align}\label{eq:fpq}
\|fv\|_{L^q} \|gv^{-1}\|_{L^{q'}} \le 2 \|fw\|_{L^p} \|gw^{-1}\|_{L^{p'}}, 
\end{align}
where $C_p=3^{n(p'+8)(q-p)}$ if $p<q$, and $C_p=3^{n(p+8)}$ if $p>q$. 
\end{theorem}

\begin{proof}
Let $p \in (1, \infty)$, $w^p \in A_p$, $f \in L^p(w^p)$, and $g \in L^{p'}(w^{-p'})$. We may assume that $f$ and $g$ are nonnegative and non-trivial. Let us first consider the case $p<q$. By $w^p \in A_p$ and Theorem \ref{thm:RdF}, there exists an operator $\mathcal{R}: L^p(w^p) \to L^p(w^p)$ such that
\begin{align}\label{eq:Ap-RdF}
f \le \mathcal{R} f, \quad
\|\mathcal{R} f\|_{L^p(w^p)} \le 2 \|f\|_{L^p(w^p)}, \quad\text{ and }\quad
[\mathcal{R} f]_{A_1} \le 2 \|M\|_{L^p(w^p)}. 
\end{align}
Define
\begin{align}\label{eq:Ap-v}
v := w^{\frac{p}{q}} (\mathcal{R} f)^{-1+\frac{p}{q}}.
\end{align}
Then by Lemma \ref{lem:fac}, the last estimate in \eqref{eq:Ap-RdF}, and \eqref{eq:sharp}, 
\begin{align}\label{eq:Ap-vq}
[v^q]_{A_q} = [w^p (\mathcal{R} f)^{p-q}]_{A_q}
\le [w^p]_{A_p} [\mathcal{R} f]_{A_1}^{q-p} 
\le 3^{n(p' + 8) (q-p)} [w^p]_{A_p}^{\frac{q-1}{p-1}}.
\end{align}
On the other hand, it follows from the first two estimate in \eqref{eq:Ap-RdF} that
\begin{align}\label{eq:Ap-fv}
\|fv\|_{L^q}
=\|f (\mathcal{R} f)^{-1+\frac{p}{q}} w^{\frac{p}{q}}\|_{L^q}
\le \|(\mathcal{R} f \cdot w)^{\frac{p}{q}}\|_{L^q}
=\|\mathcal{R} f \cdot w\|_{L^p}^{\frac{p}{q}}
\le (2 \|fw\|_{L^p})^{\frac{p}{q}}.
\end{align}
In view of $p<q$, we set $\frac1r := \frac1p - \frac1q$. Then, $\frac{1}{q'}=\frac{1}{p'} + \frac1r$, and by H\"{o}lder's inequality,
\begin{align}\label{eq:Ap-gv}
\|gv^{-1}\|_{L^{q'}}
=\|(g w^{-1}) (wv^{-1})\|_{L^{q'}}
\le \|gw^{-1}\|_{L^{p'}} \|wv^{-1}\|_{L^r}.
\end{align}
Observe that
\begin{align}\label{eq:Ap-wv}
\|wv^{-1}\|_{L^r}
=\|(\mathcal{R}f \cdot w)^{p(\frac1p-\frac1q)}\|_{L^r}
=\|\mathcal{R}f \cdot w\|_{L^p}^{1 - \frac{p}{q}}
\le (2 \|fw\|_{L^p})^{1 - \frac{p}{q}},
\end{align}
where the second estimate estimate in \eqref{eq:Ap-RdF} was used in the last step. Now collecting \eqref{eq:Ap-fv}--\eqref{eq:Ap-wv}, we obtain
\begin{align*}
\|fv\|_{L^q} \|gv^{-1}\|_{L^{q'}}
\le 2 \|fw\|_{L^p} \|gw^{-1}\|_{L^{p'}}.
\end{align*}
This and \eqref{eq:Ap-vq} show \eqref{eq:fpq} in the case $p<q$.

Let us deal with the case $q<p$, which is equivalent to $p'<q'$. Also, $w^p \in A_p$ is equivalent to $w^{-p'} \in A_{p'}$.
Note that $g \in L^{p'}(w^{-p'})$ and $f \in L^p(w^p)$. Invoking \eqref{eq:fpq} for $p'$, $q'$, $g$, $f$, $w^{-1}$ in place of $p$, $q$, $f$, $g$, and $w$, respectively, one can find a weight $u^{q'} \in A_{q'}$ with 
\begin{equation}\label{eq:uuq}
[u^{-q'}]_{A_{q'}} 
\le 3^{n(p+8)(q'-p')} [w^{-p'}]_{A_{p'}}^{\frac{q'-1}{p'-1}}
\end{equation}
such that
\begin{align}\label{eq:Ap-gf}
\|gu\|_{L^{q'}}  \|f u^{-1}\|_{L^q}
\le 2 \|gw^{-1}\|_{L^{p'}} \|fw\|_{L^p}.
\end{align}
Picking $v=u^{-1}$ and using \eqref{eq:uuq}, we see that 
\begin{align*}
[v^q]_{A_q} 
= [u^{-q}]_{A_q} 
= [u^{-q'}]_{A_{q'}}^{\frac{1}{q'-1}}
&\le 3^{n(p+8) \frac{q'-p'}{q'-1}} [w^{-p'}]_{A_{p'}}^{\frac{1}{p'-1}}
< 3^{n(p+8)} [w^p]_{A_p}, 
\end{align*} 
and \eqref{eq:Ap-gf} can be rewritten as
\begin{align*}
\|f v\|_{L^q} \|gv^{-1}\|_{L^{q'}}
\le 2 \|fw\|_{L^p} \|gw^{-1}\|_{L^{p'}}.
\end{align*}
This shows \eqref{eq:fpq} in the case $q<p$.
\end{proof}

\begin{proof}[\textbf{Proof of Theorem $\ref{thm:Ap}$.}]
Let $p \in (1, \infty)$ and $w^p \in A_p$. By duality, 
\begin{align}\label{Apvw}
\|fw\|_{L^p}
= \sup_{\substack{0 \le h \in L^{p'}(w^{-p'}) \\ \|hw^{-1}\|_{L^{p'}} = 1}} |\langle f, h \rangle|. 
\end{align}
Fix a nonnegative function $h \in L^{p'}(w^{-p'})$ with $\|hw^{-1}\|_{L^{p'}} = 1$. In view of Theorem \ref{thm:Apvw}, one can find a weight $v^{p_0} \in A_{p_0}$ such that 
\begin{align}
\label{Apvw-1}
[v^{p_0}]_{A_{p_0}}   
&\le C_p \, [w^p]_{A_p}^{\max\{1, \frac{p_0 - 1}{p -1}\}},   
\\
\label{Apvw-2} \|gv\|_{L^{p_0}} \|h v^{-1}\|_{L^{p'_0}} 
&\le 2 \|gw\|_{L^p} \|hw^{-1}\|_{L^{p'}}, 
\end{align}
where $C_p=3^{n(p'+8)(p_0-p)}$ if $p<p_0$, and $C_p=3^{n(p+8)}$ if $p>p_0$. Hence, by \eqref{eq:Ap-1}, \eqref{Apvw-1}, and \eqref{Apvw-2}, 
\begin{align*}
|\langle f, h \rangle|
\le \|fv\|_{L^{p_0}} \|hv^{-1}\|_{L^{p'_0}}
&\le \Phi([v^{p_0}]_{A_{p_0}}) \|gv\|_{L^{p_0}} \|hv^{-1}\|_{L^{p'_0}}
\\
&\le 2 \Phi(C_p [w^p]_{A_p}^{\max\{1, \frac{p_0 - 1}{p -1}\}}) 
\|gw\|_{L^p} \|hw^{-1}\|_{L^{p'}}, 
\end{align*}
which along with \eqref{Apvw} yields at once \eqref{eq:Ap-2} as desired. 
\end{proof}

Next, we would like to use Theorem \ref{thm:Ap} to get additional results. 

\begin{theorem}\label{thm:Ap-res}
Let $\F$ be a family of extrapolation pairs. Assume that there exist exponents $p_0 \in (0, \infty)$ and $q_0 \in [1, \infty)$ such that for all weights $v \in A_{q_0}$, 
\begin{equation}\label{eq:Apres-1}
\|f\|_{L^{p_0}(v)} \leq \Phi ([v]_{A_{q_0}})  \|g\|_{L^{p_0}(v)}, \quad (f, g) \in \F,
\end{equation}
where $\Phi : [1, \infty) \to [1, \infty)$ is an increasing function. Then for all exponents $p \in (1, \infty)$ and all weights $w \in A_p$, 
\begin{equation}\label{eq:Apres-2}
\|f\|_{L^{pp_0/q_0}(w)} 
\leq 2^{\frac{q_0}{p_0}} \Phi \Big(C_p \, [w]_{A_p}^{\max\{1, \frac{q_0 - 1}{p -1}\}}\Big)  
\|g\|_{L^{pp_0/q_0}(w)}, \quad (f, g) \in \F, 
\end{equation}
where $C_p=3^{n(p'+8)(q_0-p)}$ if $p<q_0$, and $C_p=3^{n(p+8)}$ if $p>q_0$. 
\end{theorem}

\begin{proof}
Set 
\[
\widetilde{\F} := \big\{(F, G)= \big(f^{\frac{p_0}{q_0}}, g^{\frac{p_0}{q_0}}\big): (f, g) \in \F\big\}.
\]
Note that \eqref{eq:Apres-1} implies that for all weights $v \in A_{q_0}$, 
\begin{equation}\label{eq:Apres-3}
\|F\|_{L^{q_0}(v)} 
= \|f\|_{L^{p_0}(v)}^{\frac{p_0}{q_0}} 
\leq \Phi ([v]_{A_{q_0}})^{\frac{p_0}{q_0}}  \|g\|_{L^{p_0}(v)}^{\frac{p_0}{q_0}}
=\Phi ([v]_{A_{q_0}})^{\frac{p_0}{q_0}} \|G\|_{L^{q_0}(v)}, 
\end{equation}
for all $(F, G) \in \widetilde{\F}$. Then it follows from \eqref{eq:Apres-3} and Theorem \ref{thm:Ap} with $p_0$ replaced by $q_0$ that for all exponent $p \in (1, \infty)$ and for all weights $w \in A_p$, 
\begin{equation*}
\|F\|_{L^p(w)} 
\leq 2 \Phi \Big(C_p \, [w]_{A_p}^{\max\{1, \frac{q_0 - 1}{p -1}\}}\Big)^{\frac{p_0}{q_0}} 
\|G\|_{L^{q_0}(w)}, \quad (F, G) \in \widetilde{\F},
\end{equation*}
which can be rewritten as 
\begin{equation*}
\|f\|_{L^{pp_0/q_0}(w)} 
\leq 2^{\frac{q_0}{p_0}} \Phi \Big(C_p \, [w]_{A_p}^{\max\{1, \frac{q_0 - 1}{p -1}\}}\Big)  
\|g\|_{L^{pp_0/q_0}(w)}, \quad (f, g) \in \F, 
\end{equation*}
where $C_p=3^{n(p'+8)(q_0-p)}$ if $p<q_0$, and $C_p=3^{n(p+8)}$ if $p>q_0$. This shows \eqref{eq:Apres-2}.
\end{proof}

\begin{theorem}\label{thm:weakAp}
Let $T$ be a sublinear operator. Assume that there exists $p_0 \in [1, \infty)$ such that for all $v \in A_{p_0}$, 
\begin{align}\label{weakAp-1}
\|Tf\|_{L^{p_0, \infty}(v)} \le \Phi([v]_{A_{p_0}}) \|f\|_{L^{p_0}(v)}, 
\end{align}
where $\Phi : [1, \infty) \to [1, \infty)$ is an increasing function. Then for all $p \in (1, \infty)$ and for all $w \in A_p$, 
\begin{align}
\label{weakAp-2} \|Tf\|_{L^{p, \infty}(w)} 
&\le 2 \Phi \Big(C_p \, [w^p]_{A_p}^{\max\{1, \frac{p_0 - 1}{p -1}\}}\Big) \|f\|_{L^p(w)}, 
\\
\label{weakAp-3} \|Tf\|_{L^p(w)} 
& \le 2 \Phi \Big(C_p \, [w^p]_{A_p}^{\max\{1, \frac{3(p_0 - 1)}{p -1}\}}\Big) \|f\|_{L^p(w)}. 
\end{align}
\end{theorem}

\begin{proof}
Given an arbitrary number $\lambda>0$, we denote
\begin{align*}
\F_{\lambda} :=\{(F_{\lambda}, G) :=  (\lambda \mathbf{1}_{\{x \in \Rn: |Tf(x)| > \lambda\}}, f): f \}. 
\end{align*}
The hypothesis \eqref{weakAp-1} implies that  for all weights $v \in A_{p_0}$,
\begin{align}\label{eq:Flam}
\|F_{\lambda}\|_{L^{p_0}(v)} 
&= \lambda \, v(\{x \in \Rn: |Tf(x)| > \lambda\})^{\frac{1}{p_0}}
\le \|Tf\|_{L^{p_0, \infty}(v)}
\\ \nonumber 
&\leq \Phi([v]_{A_{p_0}}) \|f\|_{L^{p_0}(v)}
= \Phi([v]_{A_{p_0}}) \|G\|_{L^{p_0}(v)},
\end{align}
for all $(F_{\lambda}, G) \in \F_{\lambda}$. Thus, \eqref{eq:Flam} means that \eqref{eq:Ap-1} is satisfies for the family $\F_{\lambda}$. Then Theorem \ref{thm:Ap} yields that for all exponents $p \in (1, \infty)$ and all weights $w \in A_p$, 
\begin{align*}
\lambda \, w(\{x \in \Rn: |Tf(x)| > \lambda\})^{\frac1p}
= \|F_{\lambda}\|_{L^p(w)}
&\le 2 \Phi \Big(C_p \, [w^p]_{A_p}^{\max\{1, \frac{p_0 - 1}{p -1}\}}\Big) \|G\|_{L^p(w)}
\\
&=  2 \Phi \Big(C_p \, [w^p]_{A_p}^{\max\{1, \frac{p_0 - 1}{p -1}\}}\Big) \|f\|_{L^p(w)},
\end{align*}
where $C_p=3^{n(p'+8)(p_0-p)}$ if $p<p_0$, and $C_p=3^{n(p+8)}$ if $p>p_0$, 
which along with the arbitrariness of $\lambda$ implies \eqref{weakAp-2}. 

To prove \eqref{weakAp-3}, we fix $q \in (1, \infty)$ and $w \in A_q$. By Lemma \ref{lem:open}, there exist $\gamma \in (0, 1)$ and $q_0 \in (1, q)$ so that 
\begin{align}\label{AqAq-1}
q_0 = \frac{q}{1+\varepsilon}, \quad
0<\varepsilon<\frac{q-1}{(1+\gamma)'}, \quad 
(1+\gamma)' \simeq [v]_{A_q}^{\max\{1, \frac{1}{q-1}\}}, \quad 
[w]_{A_{q_0}} \le 2^q [w]_{A_q}.
\end{align}
We may assume that $\varepsilon<\frac12$ since in this case \eqref{AqAq-1} still holds. Choose $q_1 := \frac{q}{1-\varepsilon} \in (q, 2q)$ such that $\frac1q=\frac{1-\theta}{q_0} + \frac{\theta}{q_1}$ with $\theta=\frac12$. Then, 
\begin{align}\label{AqAq-2}
w \in A_q \subset A_{q_1} \quad\text{ with }\quad [w]_{A_{q_1}} \le [w]_{A_q}.
\end{align}
Then it follows from \eqref{AqAq-1}, \eqref{AqAq-2}, and \eqref{weakAp-2} (with the exact constant $C_p$, see the proof above) that 
\begin{align}
\label{weakAp-4} \|Tf\|_{L^{q_0, \infty}(w)} 
&\le 2 \Phi \Big(C_{q_0} \, [w]_{A_{q_0}}^{\max\{1, \frac{p_0 - 1}{q_0 -1}\}}\Big) \|f\|_{L^{q_0}(w)}, 
\\
\label{weakAp-5} \|Tf\|_{L^{q_1, \infty}(w)} 
&\le 2 \Phi \Big(C_{q_1} \, [w]_{A_{q_1}}^{\max\{1, \frac{p_0 - 1}{q_1 -1}\}}\Big) \|f\|_{L^{q_1}(w)},  
\end{align}
where 
\begin{equation*}
C_{q_i}=
\begin{cases}
3^{n(q'_i+8)(p_0-q_i)}, & \text{ if } q_i<p_0, 
\\ 
3^{n(q_i+8)}, &\text{ if } q_i>p_0, 
\end{cases}
\qquad 
i=0, 1. 
\end{equation*}
Additionally, by the choice of $q_0$ and $q_1$, and that $\varepsilon<(q-1)/2$, we have 
\begin{align}
\label{qwqw-1} & q_0-1=\frac{q}{1+\varepsilon} -1
> \frac{q-1}{2(1+\varepsilon)}
>\frac{q-1}{3}, \quad
q'_0 = \frac{q}{q-1-\varepsilon} < 2q', 
\\ 
\label{qwqw-2} & C_{q_i} \le C'_q, \quad\text{and}\quad
[w]_{A_{q_i}}^{\max\{1, \frac{p_0 - 1}{q_i -1}\}}
\le [w]_{A_q}^{\max\{1, \frac{3(p_0 - 1)}{q -1}\}}, \quad i=0, 1, 
\end{align}
where $C'_q$ depends only on $n$, $p_0$, and $q$. Thus, invoking \eqref{qwqw-2}, we interpolate between \eqref{weakAp-4} and \eqref{weakAp-5} to conclude  
\begin{align*}
\|Tf\|_{L^q(w)} 
&\le 2 \Phi \Big(C'_q \, [w^p]_{A_p}^{\max\{1, \frac{3(p_0 - 1)}{q -1}\}}\Big) \|f\|_{L^q(w)}. 
\end{align*}
This completes the proof.
\end{proof}

\subsection{Off-diagonal extrapolation}
We next present a quantitative off-diagonal extrapolation below, which improves Theorem \ref{thm:Ap} to the limited range case. 

\begin{theorem}\label{thm:off}
Let $\F$ be a family of extrapolation pairs and $1 \le \p_- < \p_+ \le \infty$. Assume that there exist exponents $p_0 \in [\p_-, \p_+]$ and $q_0 \in (1, \infty)$ such that for all weights $v^{p_0} \in A_{p_0/\p_-} \cap RH_{(\p_+/p_0)'}$, 
\begin{equation}\label{eq:off-1}
\|f v\|_{L^{q_0}} 
\leq \Phi \big([v^{p_0(\p_+/p_0)'}]_{A_{\tau_{p_0}}}\big)  
\|g v\|_{L^{p_0}}, \quad (f, g) \in \F,
\end{equation}
where $\Phi : [1, \infty) \to [1, \infty)$ is an increasing function. Then for all exponents $p \in (\p_-, \p_+)$ and $q \in (1, \infty)$ satisfying $\frac1p - \frac1q = \frac{1}{p_0} - \frac{1}{q_0}$,  and all weights $w^p \in A_{p/\p_-} \cap RH_{(\p_+/p)'}$, 
\begin{equation}\label{eq:off-2}
\|f w\|_{L^q} 
\leq 2^{\max\{\frac{\tau_{p_0}}{p_0}, \frac{\tau_p}{p}\}} 
\Phi \Big(C_{p, q} \, [w^{p(\p_+/p)'}]_{A_{\tau_p}}^{\max\{1, \frac{\tau_{p_0} - 1}{\tau_p -1}\}}\Big) 
\|g w\|_{L^p}, \quad (f, g) \in \F. 
\end{equation}
\end{theorem}

To show Theorem \ref{thm:off}, we present a more general result below.
 
\begin{theorem}\label{thm:off-vw}
Let $\beta \in (0, \infty)$, $p_0, q_0 \in [1, \infty)$, $p, q \in (1, \infty)$, and let $r_0, r \in (\frac{1}{\beta}, \infty)$ be such that $\frac1q - \frac{1}{q_0} = \frac1r - \frac{1}{r_0} = \frac1p - \frac{1}{p_0}$. Then for all weights $w^r \in A_{r \beta}$ and for all functions $f \in L^p(w^p)$ and $g \in L^{q'}(w^{-q'})$, there exists a weight $v^{r_0} \in A_{r_0 \beta}$ such that 
\begin{align}
[v^{r_0}]_{A_{r_0 \beta}} 
&\lesssim [w^r]_{A_{r \beta}}^{\max\{1, \frac{r_0 \beta-1}{r \beta -1}\}}, 
\\
\|fv\|_{L^{p_0}} \|g v^{-1}\|_{L^{q'_0}} 
&\le 2^{\max\{\frac{r \beta}{p}, \frac{(r \beta)'}{q'}\}} \|fw\|_{L^p} \|gw^{-1}\|_{L^{q'}}. 
\end{align}
\end{theorem}

\begin{proof}
Fix $w^r \in A_{r \beta}$, $f \in L^p(w^p)$, and $g \in L^{q'}(w^{-q'})$. We first consider the case $q<q_0$ (equivalently, $p<p_0$ and $r<r_0$). Pick  
\begin{align}\label{defhh}
h := f^{\frac{p}{r \beta}} w^{\frac{p-r}{r \beta}} 
\quad\text{ so that }\quad 
\|h\|_{L^{r\beta}(w^r)} = \|fw\|_{L^p}^{\frac{p}{r \beta}}. 
\end{align}
By $w^r \in A_{r \beta}$ and Theorem \ref{thm:RdF}, there exists an operator $\mathcal{R}: L^{r \beta}(w^r) \to L^{r \beta}(w^r)$ such that
\begin{align}\label{eq:off-RdF}
h \le \mathcal{R} h, \quad
\|\mathcal{R} h\|_{L^{r \beta}(w^r)} \le 2 \|h\|_{L^{r \beta}(w^r)}, \quad\text{ and }\quad
[\mathcal{R} h]_{A_1} \le 2 \|M\|_{L^{r \beta}(w^r)}. 
\end{align}
Define
\begin{align}\label{eq:off-vv}
v := w^{\frac{r}{r_0}} (\mathcal{R} h)^{\frac{(r-r_0)\beta}{r_0}}.
\end{align}
Then by Lemma \ref{lem:fac} part \eqref{list:fac-1}, the last inequality in \eqref{eq:off-RdF}, and \eqref{eq:sharp}, 
\begin{align}
[v^{r_0}]_{A_{r_0 \beta}}
=[w^r (\mathcal{R} h)^{r\beta - r_0 \beta}]_{A_{r_0 \beta}}
\le [w^r]_{A_{r \beta}} [\mathcal{R}h]_{A_1}^{r_0 \beta - r \beta} 
\lesssim [w^r]_{A_{r \beta}}^{\frac{r_0 \beta -1}{r \beta -1}}. 
\end{align}
It follows from \eqref{defhh}, \eqref{eq:off-RdF}, and \eqref{eq:off-vv} that 
\begin{align}\label{off-fv}
\|fv\|_{L^{p_0}} 
&= \Big\|\big(h^{\frac{r \beta}{p}} w^{\frac{r}{p}-1+\frac{r}{p_0}} \big) 
(\mathcal{R} h)^{\frac{(r-r_0) \beta}{r_0}} \Big\|_{L^{p_0}}
\le \Big\|\big[(\mathcal{R} h)^{r \beta} w^r \big]^{\frac1p - \frac1r + \frac{1}{r_0}} \Big\|_{L^{p_0}}
\\ \nonumber 
&=\|\mathcal{R}h\|_{L^{r\beta}(w^r)}^{\frac{r \beta}{p_0}} 
\le (2 \|h\|_{L^{r\beta}(w^r)})^{\frac{r \beta}{p_0}} 
= \big(2 \|fw\|_{L^p}^{\frac{p}{r \beta}} \big)^{\frac{r \beta}{p_0}} 
= 2^{\frac{r \beta}{p_0}} \|fw\|_{L^p}^{\frac{p}{p_0}}. 
\end{align}
To proceed, we set $\frac1t := \frac{1}{q} - \frac{1}{q_0}$, equivalently $\frac{1}{q'_0} = \frac{1}{q'} + \frac1t$. By H\"{o}lder's inequality, 
\begin{align}\label{off-gv}
\|gv^{-1}\|_{L^{q'_0}} 
\le \|gw^{-1}\|_{L^{q'}} \|wv^{-1}\|_{L^t}, 
\end{align}
and by \eqref{defhh}--\eqref{eq:off-vv}, 
\begin{multline}\label{off-wv}
\|wv^{-1}\|_{L^t} 
=\Big\|(\mathcal{R} h)^{\beta(1-\frac{r}{r_0})} w^{1-\frac{r}{r_0}} \Big\|_{L^t}
=\|\mathcal{R}h\|_{L^{r\beta}(w^r)}^{r\beta(\frac1r - \frac{1}{r_0})}
\le \big(2 \|h\|_{L^{r\beta}(w^r)} \big)^{r\beta(\frac1r - \frac{1}{r_0})} 
\\ 
= \big(2 \|fw\|_{L^p}^{\frac{p}{r \beta}} \big)^{r\beta(\frac1r - \frac{1}{r_0})} 
= 2^{r\beta(\frac1r - \frac{1}{r_0})} \|fw\|_{L^p}^{p(\frac1r - \frac{1}{r_0})}.   
\end{multline}
Now collecting \eqref{off-fv}, \eqref{off-gv}, and \eqref{off-wv}, we deduce that 
\begin{align}
\|fv\|_{L^{p_0}} \|gv^{-1}\|_{L^{q'_0}} 
\le 2^{\frac{r \beta}{p}}\|fw\|_{L^p} \|gw^{-1}\|_{L^{q'}}, 
\end{align}
provided $\frac1q - \frac{1}{q_0} = \frac1r - \frac{1}{r_0} = \frac1p - \frac{1}{p_0}$. This shows the case $q<q_0$. 

Next let us deal with the case $q>q_0$ (equivalently, $p>p_0$ and $r>r_0$). Set 
\begin{align}\label{def:ssrr}
s := \frac{r}{r \beta -1} \quad\text{ and }\quad 
s_0 := \frac{r_0}{r_0 \beta -1}. 
\end{align}
Recall that $w^r \in A_{r \beta}$. Then we see that 
\begin{align}\label{wssrp}
w^{-s} \in A_{s \beta}, \quad 
q'<q'_0, \quad\text{ and }\quad 
\frac1s - \frac{1}{s_0} = \frac{1}{r_0} - \frac{1}{r} = \frac{1}{p'} - \frac{1}{p'_0} = \frac{1}{q'} - \frac{1}{q'_0}. 
\end{align}
Hence, the conclusion in the preceding case applied to the tuple $(q', p', s, q'_0, p'_0, s_0, g, f, w^{-1})$  in place of $(p, q, r, p_0, q_0, r_0, f, g, w)$ gives that there exists a weight $u^{s_0} \in A_{s_0 \beta}$ so that 
\begin{align}
\label{uuss} [u^{s_0}]_{A_{s_0 \beta}} 
&\lesssim [w^{-s}]_{A_{s \beta}}^{\frac{s_0 \beta -1}{s \beta -1}}, 
\\
\label{gguu} \|gu\|_{L^{q'_0}} fu^{-1}\|_{L^{p_0}} 
&\le 2^{\frac{s\beta}{q'}} \|gw^{-1}\|_{L^{q'}} \|fw\|_{L^p}. 
\end{align}
Note that by \eqref{def:ssrr}, 
\begin{align}\label{rgrg}
(r \beta -1) (s \beta-1) =1 
\quad\text{ and }\quad 
(r_0 \beta -1) (s_0 \beta-1) =1. 
\end{align}
Pick $v := u^{-1}$. Then by \eqref{wssrp}, \eqref{uuss}, and \eqref{rgrg}, 
\begin{align*}
[v^{r_0}]_{A_{r_0 \beta}} 
&=[u^{-r_0}]_{A_{r_0 \beta}} 
=[u^{\frac{r_0}{r_0 \beta -1}}]_{A_{(r_0 \beta)'}}^{r_0 \beta -1}
=[u^{s_0}]_{A_{s_0 \beta}}^{r_0 \beta -1}
\\ 
&\lesssim [w^{-s}]_{A_{s \beta}}^{\frac{1}{s \beta -1}}
=[w^{-\frac{r}{r \beta -1}}]_{A_{(r \beta)'}}^{r \beta -1}
=[w^r]_{A_{r \beta}}, 
\end{align*}
and \eqref{gguu} can be rewritten as 
\begin{align*}
\|fv\|_{L^{p_0}} \|gv^{-1}\|_{L^{q'_0}}
\le 2^{\frac{(r \beta)'}{q'}} \|fw\|_{L^p} \|gw^{-1}\|_{L^{q'}}. 
\end{align*}
In the case $q=q_0$, taking $v:=w$, the conclusion is trivial. This completes the proof. 
\end{proof}

The following conclusion is a particular case of Theorem \ref{thm:off-vw}. 

\begin{theorem}\label{thm:lim-off-vw}
Let $1 \le \p_- < \p_+ \le \infty$, $p_0 \in [\p_-, \p_+]$, $p \in (\p_-, \p_+)$, and let $q_0, q \in (1, \infty)$ be such that $\frac1q - \frac{1}{q_0} = \frac1p - \frac{1}{p_0}$. Then for all weights $w^p \in A_{p/\p_-} \cap RH_{(\p_+/p)'}$ and for all functions $f \in L^p(w^p)$ and $g \in L^{q'}(w^{-q'})$, there exists a weight $v^{p_0} \in A_{p_0/\p_-} \cap RH_{(\p_+/p_0)'}$ such that 
\begin{align}
\label{eq:lim-off-vw-2}
[v^{p_0(\p_+/p_0)'}]_{A_{\tau_{p_0}}}   
&\lesssim [w^{p(\p_+/p)'}]_{A_{\tau_p}}^{\max\{1, \frac{\tau_{p_0} - 1}{\tau_p -1}\}},  
\\
\label{eq:lim-off-vw-3} \|fv\|_{L^{p_0}} \|g v^{-1}\|_{L^{q'_0}} 
&\le 2^{\max\{\frac{\tau_p}{p}, \frac{\tau'_p}{q'}\}} \|fw\|_{L^p} \|gw^{-1}\|_{L^{q'}}. 
\end{align}
\end{theorem}

\begin{proof}
Denote 
\begin{align}\label{rppp}
r := p(\p_+/p)', \quad 
r_0 := p_0(\p_+/p_0)', \quad\text{ and }\quad 
\beta := \frac{1}{\p_-} - \frac{1}{\p_+}. 
\end{align}
Then one can check that 
\begin{align}\label{rgtp}
r \beta = \tau_p, \quad 
r_0 \beta = \tau_{p_0}, \quad\text{ and }\quad
\frac1r - \frac{1}{r_0}
=\frac1p - \frac{1}{p_0} 
=\frac1q - \frac{1}{q_0}.
\end{align}
Let $w^p \in A_{p/\p_-} \cap RH_{(\p_+/p)'}$, $f \in L^p(w^p)$, and $g \in L^{q'}(w^{-q'})$. Then it follows from Lemma \ref{lem:weights} part \eqref{list:ApRH-2} and \eqref{rgtp} that $w^r \in A_{r \beta}$, which together with Theorem \ref{thm:off-vw} implies that there exists a weight $v^{r_0} \in A_{r_0 \beta}$ such that 
\begin{align}
\label{vrar-1} [v^{r_0}]_{A_{r_0 \beta}} 
&\lesssim [w^r]_{A_{r \beta}}^{\max\{1, \frac{r_0 \beta-1}{r \beta -1}\}}, 
\\
\label{vrar-2} \|fv\|_{L^{p_0}} \|g v^{-1}\|_{L^{q'_0}} 
&\le 2^{\max\{\frac{r \beta}{p}, \frac{(r \beta)'}{q'}\}} \|fw\|_{L^p} \|gw^{-1}\|_{L^{q'}}. 
\end{align}
In view of \eqref{rppp}, \eqref{rgtp}, and Lemma \ref{lem:weights} part \eqref{list:ApRH-2},  we conclude from \eqref{vrar-1} and \eqref{vrar-2} that $v^{p_0} \in A_{p_0/\p_-} \cap RH_{(\p_+/p_0)'}$ so that \eqref{eq:lim-off-vw-2} and \eqref{eq:lim-off-vw-3} hold. 
\end{proof}

Let us see how we deduce Theorem \ref{thm:off} from Theorem \ref{thm:lim-off-vw}.

\begin{proof}[\textbf{Proof of Theorem $\ref{thm:off}$.}]
By duality, 
\begin{align}\label{fw-dual}
\|fw\|_{L^q}
= \sup_{\substack{0 \le h \in L^{q'}(w^{-q'}) \\ \|hw^{-1}\|_{L^{q'}} = 1}} |\langle f, h \rangle|. 
\end{align}
Fix a nonnegative function $h \in L^{q'}(w^{-q'})$ with $\|hw^{-1}\|_{L^{q'}} = 1$. By Theorem \ref{thm:lim-off-vw}, there exists a weight $v^{p_0} \in A_{p_0/\p_-} \cap RH_{(\p_+/p_0)'}$ such that 
\begin{align}
\label{vvw-1}
[v^{p_0(\p_+/p_0)'}]_{A_{\tau_{p_0}}}   
&\lesssim [w^{p(\p_+/p)'}]_{A_{\tau_p}}^{\max\{1, \frac{\tau_{p_0} - 1}{\tau_p -1}\}},  
\\
\label{vvw-2} \|gv\|_{L^{p_0}} \|h v^{-1}\|_{L^{q'_0}} 
&\le 2^{\max\{\frac{\tau_p}{p}, \frac{\tau'_p}{q'}\}} \|gw\|_{L^p} \|hw^{-1}\|_{L^{q'}}. 
\end{align}
Then, in view of \eqref{vvw-1}, we use \eqref{eq:off-1} and \eqref{vvw-2} to obtain 
\begin{align*}
|\langle f, h \rangle|
&\le \|fv\|_{L^{q_0}} \|hv^{-1}\|_{L^{q'_0}}
\le \Phi([v^{p_0(\p_+/p_0)'}]_{A_{\tau_{p_0}}}) \|gv\|_{L^{q_0}} \|hv^{-1}\|_{L^{q'_0}}
\\
&\le 2^{\max\{\frac{\tau_p}{p}, \frac{\tau'_p}{q'}\}} \Phi(C [w^{p(\p_+/p)'}]_{A_{\tau_p}}^{\max\{1, \frac{\tau_{p_0} - 1}{\tau_p -1}\}}) 
\|gw\|_{L^p} \|hw^{-1}\|_{L^{q'}}. 
\end{align*}
This along with \eqref{fw-dual} gives at once \eqref{eq:off-2} as desired. 
\end{proof}

\subsection{Multilinear extrapolation}
If we use Theorem \ref{thm:off} to show Theorem \ref{thm:lim}, it requires all the exponents are Banach. Thus, we have to improve Theorem \ref{thm:off} to the non-Banach ranges as follows. But in this case, we cannot establish a ``product-type embedding" as Theorem \ref{thm:lim-off-vw}.

\begin{theorem}\label{thm:lim-off}
Let $\F$ be a family of extrapolation pairs and $1 \le \p_- < \p_+ \le \infty$. Assume that there exist exponents $p_0, q_0 \in (0, \infty)$ such that $p_0 \in [\p_-, \p_+]$ and for all weights $v^{p_0} \in A_{p_0/\p_-} \cap RH_{(\p_+/p_0)'}$, 
\begin{equation}\label{eq:lim-off-1}
\|f v\|_{L^{q_0}} 
\leq \Phi \big([v^{p_0}]_{A_{p_0/\p_-} \cap RH_{(\p_+/p_0)'}}\big) 
\|g v\|_{L^{p_0}}, \quad (f, g) \in \F,
\end{equation}
where $\Phi : [1, \infty) \to [1, \infty)$ is an increasing function. Then for all exponents $p \in (\p_-, \p_+)$ and $q \in (0, \infty)$ satisfying $\frac1p - \frac1q = \frac{1}{p_0} - \frac{1}{q_0}$,  and all weights $w^p \in A_{p/\p_-} \cap RH_{(\p_+/p)'}$, 
\begin{equation}\label{eq:lim-off-2}
\|f w\|_{L^q} 
\leq 2^{\max\{\frac{\tau_p}{p}, \frac{\tau'_p}{p_0}\}} 
\Phi \big(C_0 \, [w^p]_{A_{p/\p_-} \cap RH_{(\p_+/p)'}}^{\gamma(p, \, p_0)} \big)
\|g w\|_{L^p}, \quad (f, g) \in \F,  
\end{equation} 
where the constant $C_0$ depends only on $n$, $p$, $p_0$, $\p_-$, and $\p_+$, and 
\begin{equation*}
\gamma(p, p_0) := 
\begin{cases}
\max\big\{1, \frac{\tau_{p_0} - 1}{\tau_p -1} \big\}, & p_0<\p_+, 
\\[4pt]
\frac{p_0}{\tau_p -1} \big(\frac{1}{\p_-} - \frac{1}{\p_+}\big), & p_0 = \p_+.
\end{cases}
\end{equation*} 
\end{theorem}

\begin{proof}
Fix $p \in (\p_-, \p_+)$ and $q \in (0, \infty)$ satisfying $\frac1p - \frac1q = \frac{1}{p_0} - \frac{1}{q_0}$,  and let $w^p \in A_{p/\p_-} \cap RH_{(\p_+/p)'}$. Fix $(f, g) \in \F$. Without loss of generality we may assume that $0 < \|gw\|_{L^p} < \infty$. Indeed, if $\|gw\|_{L^p} = \infty$ there is nothing to prove, and if $\|gw\|_{L^p} = 0$, then $g=0$ a.e. and by \eqref{eq:lim-off-1} we see that $f=0$ a.e., which trivially implies \eqref{eq:lim-off-2}. We split the proof into two cases. 

\medskip {\bf Case I: $q<q_0$.} Recall that $\tau_t = \big(\frac{\p_{+}}{t} \big)' \big(\frac{t}{\p_-}-1)+1$ for any $t \in [\p_{-}, \p_+]$. Obviously,  $\tau_t$ is an increasing function in $t$. Lemma \ref{lem:weights} part \eqref{list:ApRH-2} gives 
\begin{align}\label{eq:wppAA}
w^{p(\p_+/p)'} \in A _{\tau_p}.
\end{align}  
Set
\begin{align}\label{def:hhff}
h := g^{\frac{p}{\tau_{p}}}w^{\frac{p}{\tau_{p}}[1-(\p_+/p)']}
\quad\text{ so that }\quad 
\|h\|_{L^{\tau_p}(w^{p(\p_+/p)'})} 
=\|gw\|_{L^p}^{\frac{p}{\tau_p}} 
< \infty, 
\end{align}
which along with \eqref{eq:wppAA} and Theorem \ref{thm:RdF} implies that there exists an operator $\mathcal{R}: L^{\tau_p}(w^{p(\p_+/p)'}) \to L^{\tau_p}(w^{p(\p_+/p)'})$ such that
\begin{align}\label{eq:AR-RdF}
h \le \mathcal{R} h, \quad 
\|\mathcal{R}h\|_{L^{\tau_p}(w^{p(\p_+/p)'})} \leq 2 \|h\|_{L^{\tau_p}(w^{p(\p_+/p)'})}, 
\, \text{ and }\, 
[\mathcal{R}h]_{A_1} \le 2\|M\|_{L^{\tau_p}(w^{p(\p_+/p)'})}.  
\end{align}
Then \eqref{def:hhff} and the second estimate in \eqref{eq:AR-RdF} yield 
\begin{align}\label{eq:RRhfw} 
\|\mathcal{R}h\|_{L^{\tau_p}(w^{p(\p_+/p)'})} 
\leq 2 \|gw\|_{L^p}^{\frac{p}{\tau_p}}.  
\end{align}

Assume first that $p_0<\p_+$. Pick 
\begin{align}\label{eq:vvww}
v := w^{\frac{p(\p_+/p)'}{p_0(\p_+/p_0)'}}(\mathcal{R}h)^{\frac{\tau_p - \tau_{p_0}}{p_0(\p_+/p_0)'}}. 
\end{align}
Considering $p<p_0$, \eqref{eq:wppAA}, and the last estimate in \eqref{eq:AR-RdF}, we use Lemma \ref{lem:fac} and \eqref{eq:sharp} to get $v^{p_0(\p_+/p_0)'} \in A _{\tau_{p_0}}$ with 
\begin{align}\label{eq:vvwtau}
[v^{p_0(\p_+/p_0)'}]_{A_{\tau_{p_0}}}
\le [w^{p(\p_+/p)'}]_{A_{\tau_p}} [\mathcal{R}h]_{A_1}^{\tau_{p_0} - \tau_p}
\le C_1 [w^{p(\p_+/p)'}]_{A_{\tau_p}}^{\frac{\tau_{p_0} - 1}{\tau_p -1}}, 
\end{align}
where the constant $C_1$ depends only on $n$, $p$, $p_0$, $\p_-$, and $\p_+$, which together with Lemma \ref{lem:weights} part \eqref{list:ApRH-2} implies 
\begin{equation}\label{eq:vvvRH}
v^{p_0} \in A_{p_0/\p_-} \cap RH_{(\p_+/p_0)'}.
\end{equation}
On the other hand, note that 
\begin{align}
\label{ppp-1}
&\frac{1}{p(\p_+/p)'} - \frac{1}{p_0(\p_+/p_0)'}
= \frac1p - \frac{1}{p_0}
= \frac1q - \frac{1}{q_0} ,  
\\ 
\label{ppp-2} 
&\frac{\tau_p}{p(\p_+/p)'} 
= \frac{1}{\p_-} - \frac{1}{\p_+} 
= \frac{\tau_{p_0}}{p_0(\p_+/p_0)'}, 
\end{align}
provided 
\begin{align}\label{taup}
\tau_p = \bigg(\frac{\p_{+}}{p}\bigg)' \bigg(\frac{p}{\p_{-}}-1 \bigg) + 1
=\frac{\frac{1}{\p_{-}}-\frac{1}{p}}{\frac{1}{p}-\frac{1}{\p_{+}}}+1
=\frac{\frac{1}{\p_{-}}-\frac{1}{\p_{+}}}{\frac{1}{p}-\frac{1}{\p_{+}}}, 
\end{align}
which also implies 
\begin{align}\label{ppp-3}
\frac{\tau_p}{p} + \frac{\tau_p - \tau_{p_0}}{p_0(\p_+/p_0)'}
&=\tau_p \bigg[\frac1p + \frac{1-\tau_{p_0}/\tau_p}{p_0(\p_+/p_0)'} \bigg]
= \tau_p \bigg[\frac1p + \bigg(\frac{1}{p_0}-\frac{1}{\p_+}\bigg)\bigg(1-\frac{\frac1p-\frac{1}{\p_+}}{\frac{1}{p_0}-\frac{1}{\p_+}}\bigg)\bigg]
= \frac{\tau_p}{p_0}. 
\end{align}
By \eqref{def:hhff}, the first estimate in \eqref{eq:AR-RdF}, \eqref{ppp-1}, and \eqref{ppp-3}, 
\begin{align}\label{eq:Apfvv}
\|gv\|_{L^{p_0}}
&= \bigg\|h^{\frac{\tau_p}{p}} w^{(\p_+/p)'-1+\frac{p(\p_+/p)'}{p_0(\p_+/p_0)'}} 
(\mathcal{R}h)^{\frac{\tau_p - \tau_{p_0}}{p_0(\p_+/p_0)'}} \bigg\|_{L^{p_0}}
\\ \nonumber
&\le \bigg\|(\mathcal{R}h)^{\frac{\tau_p}{p} + \frac{\tau_p - \tau_{p_0}}{p_0(\p_+/p_0)'}} 
w^{p(\p_+/p)' [\frac1p - (\frac{1}{p(\p_+/p)'} - \frac{1}{p_0(\p_+/p_0)'})]} \bigg\|_{L^{p_0}}
=\|\mathcal{R}h\|_{L^{\tau_p}(w^{p(\p_+/p)'})}^{\frac{\tau_p}{p_0}}. 
\end{align}
To proceed, we denote $\frac1r := \frac1q - \frac{1}{q_0}>0$. Then in light of \eqref{eq:vvww}, \eqref{ppp-1}, and \eqref{ppp-2}, it follows from H\"{o}lder's inequality that 
\begin{align*}
\|fw\|_{L^q}
&= \bigg\|\Big[f \, w^{\frac{p(\p_+/p)'}{p_0(\p_+/p_0)'}} 
(\mathcal{R}h)^{\frac{\tau_p - \tau_{p_0}}{p_0(\p_+/p_0)'}} \Big] 
\Big[(\mathcal{R} h)^{\frac{\tau_p}{p(\p_+/p)'}} w \Big]^{(1-\frac{p(\p_+/p)'}{p_0(\p_+/p_0)'})} \bigg\|_{L^q}
\\ 
&\le \|fv\|_{L^{q_0}} \Big\|\big[(\mathcal{R} h)^{\frac{\tau_p}{p(\p_+/p)'}} w \big]^{1-\frac{p(\p_+/p)'}{p_0(\p_+/p_0)'}} \Big\|_{L^r}
\\ 
&= \|fv\|_{L^{q_0}} \Big\|\big[(\mathcal{R} h)^{\tau_p} w^{p(\p_+/p)'} \big]^{\frac1q - \frac{1}{q_0}} \Big\|_{L^r}
\\ 
&= \|fv\|_{L^{q_0}} \|\mathcal{R} h\|_{L^{\tau_p}(w^{p(\p_+/p)'})}^{\tau_p(\frac1q - \frac{1}{q_0})}.  
\end{align*}
Furthermore, invoking \eqref{eq:vvwtau}, \eqref{eq:vvvRH}, and \eqref{eq:lim-off-1}, we arrive at  
\begin{align}\label{eq:lim-off-3}
\|fw\|_{L^q}
&\le \Phi \big([v^{p_0(\p_+/p_0)'}]_{A_{\tau_{p_0}}}\big)  \|g v\|_{L^{p_0}} 
\|\mathcal{R} h\|_{L^{\tau_p}(w^{p(\p_+/p)'})}^{\tau_p(\frac1q - \frac{1}{q_0})}
\\ \nonumber
&\le \Phi \big([v^{p_0(\p_+/p_0)'}]_{A_{\tau_{p_0}}}\big)  
\|\mathcal{R} h\|_{L^{\tau_p}(w^{p(\p_+/p)'})}^{\tau_p(\frac{1}{p_0} + \frac1q - \frac{1}{q_0})}
\\ \nonumber
&\le 2^{\frac{\tau_p}{p}} \Phi \Big(C_1 [w^{p(\p_+/p)'}]_{A_{\tau_p}}^{\frac{\tau_{p_0} - 1}{\tau_p -1}}\Big)  
\|gw\|_{L^p}, 
\end{align}
where we have used \eqref{eq:Apfvv}, \eqref{eq:RRhfw}, and that $\frac1p - \frac1q = \frac{1}{p_0} - \frac{1}{q_0}$.

Let us next treat the case $p_0=\p_+$. Choose $v := (\mathcal{R}h)^{\frac{1}{\p_+} - \frac{1}{\p_-}}$. Then it follows from Lemma \ref{lem:weights} part \eqref{list:ApRH-1} that 
\begin{align}\label{vrhp}
v^{p_0} = (\mathcal{R}h)^{1-\frac{p_0}{\p_-}} \in A_{p_0/\p_-} \cap RH_{\infty} = A_{p_0/\p_-} \cap RH_{(\p_+/p_0)'}
\end{align}
with 
\begin{align}
\max\big\{[v^{p_0}]_{A_{p_0/\p_-}}, [v^{p_0}]_{RH_{\infty}} \big\} 
\le [\mathcal{R}h]_{A_1}^{\frac{p_0}{\p_-} - 1} 
\lesssim [w^{p(\p_+/p)'}]_{A_{\tau_p}}^{\frac{p_0}{\tau_p-1}(\frac{1}{\p_-} - \frac{1}{\p_+})}, 
\end{align}
where we have used the last estimate in \eqref{eq:AR-RdF} and \eqref{eq:sharp}. In the current scenario, 
\begin{align}
\label{expp-1} & \frac{\tau_p}{p} + \frac{1}{\p_+} - \frac{1}{\p_-}
=\tau_p \bigg[\frac1p - \bigg(\frac1p - \frac{1}{\p_+}\bigg) \bigg]
=\frac{\tau_p}{\p_+}
=\frac{\tau_p}{p_0}, 
\\
\label{expp-2} &p_0[(\p_+/p)' -1]
=\frac{p_0p}{\p_+-p}
=\frac{\p_+ p}{\p_+ - p}
=p(\p_+/p)', 
\\ 
\label{expp-3} &\frac1r := \frac1q - \frac{1}{q_0} 
= \frac1p - \frac{1}{p_0} 
= \frac1p - \frac{1}{\p_+}
=\frac{1}{p(\p_+/p)'}, 
\\
\label{expp-4} & \text{and}\quad 
r \bigg(\frac{1}{\p_-} -\frac{1}{\p_+} \bigg)
=\frac{\frac{1}{\p_-} -\frac{1}{\p_+}}{\frac1p -\frac{1}{\p_+}}
=\tau_p. 
\end{align} 
In view of \eqref{def:hhff}, \eqref{expp-1}, and \eqref{expp-2}, there holds 
\begin{align}\label{gvph}
\|gv\|_{L^{p_0}} 
= \Big\|h^{\frac{\tau_p}{p}} w^{(\p_+/p)'-1} (\mathcal{R} h)^{\frac{1}{\p_+} - \frac{1}{\p_-}}\Big\|_{L^{p_0}}
\le \|\mathcal{R} h\|_{L^{\tau_p}(w^{p(\p_+/p)'})}^{\frac{\tau_p}{p_0}}. 
\end{align}
Hence, invoking \eqref{vrhp}--\eqref{gvph},  H\"{o}lder's inequality, and \eqref{eq:lim-off-1}, we deduce 
\begin{align*}
\|fw\|_{L^q}
&= \Big\|\Big[f \, (\mathcal{R}h)^{\frac{1}{\p_+} - \frac{1}{\p_-}} \Big] 
\Big[(\mathcal{R}h)^{\frac{1}{\p_-} - \frac{1}{\p_+}} w \Big] \Big\|_{L^q}
\\ 
&\le \|fv\|_{L^{q_0}} \big\|(\mathcal{R}h)^{\frac{1}{\p_-} - \frac{1}{\p_+}} w\big\|_{L^r}
\\ 
&= \|fv\|_{L^{q_0}} \|\mathcal{R} h\|_{L^{\tau_p}(w^{p(\p_+/p)'})}^{\tau_p(\frac1p - \frac{1}{p_0})}
\\ 
&\le \Phi \big(\max\{[v^{p_0}]_{A_{p_0/{\p_-}}},\, [v^{p_0}]_{RH_{(\p_+/p_0)'}}\}\big) \|g v\|_{L^{p_0}} 
\|\mathcal{R} h\|_{L^{\tau_p}(w^{p(\p_+/p)'})}^{\tau_p(\frac1p - \frac{1}{p_0})}
\\ \nonumber
&\le \Phi \big(\max\{[v^{p_0}]_{A_{p_0/{\p_-}}},\, [v^{p_0}]_{RH_{(\p_+/p_0)'}}\}\big)    
\|\mathcal{R} h\|_{L^{\tau_p}(w^{p(\p_+/p)'})}^{\frac{\tau_p}{p}}
\\ \nonumber
&\le 2^{\frac{\tau_p}{p}} 
\Phi \Big(C_1 [w^{p(\p_+/p)'}]_{A_{\tau_p}}^{\frac{p_0}{\tau_p-1}(\frac{1}{\p_-} - \frac{1}{\p_+})}\Big)  
\|gw\|_{L^p}, 
\end{align*}
where \eqref{eq:RRhfw} was used in the last step. \black 

\medskip {\bf Case II: $q_0<q$.} By Lemma \ref{lem:weights} parts \eqref{list:ApRH-2} and \eqref{list:ApRH-3}, 
\begin{align}\label{wsap}
w^{-s} \in A_{\tau'_p} \quad\text{ with }\quad 
[w^{-s}]_{A_{\tau'_p}} = [w^{p(\p_+/p)'}]_{A_{\tau_p}}^{\frac{1}{\tau_p-1}}, 
\end{align}
where $s=p'(\p'_-/p')'=\frac{1}{\frac{1}{\p_-} - \frac1p}$. This and Theorem \ref{thm:RdF} yield that there exists an operator $\mathcal{R}: L^{\tau'_p}(w^{-s}) \to L^{\tau'_p}(w^{-s})$ such that for any nonnegative function $\widetilde{h} \in L^{\tau'_p}(w^{-s})$, 
\begin{align}\label{eq:RRdF}
\widetilde{h} \le \mathcal{R} \widetilde{h}, \quad 
\|\mathcal{R} \widetilde{h}\|_{L^{\tau'_p}(w^{-s})} \leq 2 \|\widetilde{h}\|_{L^{\tau'_p}(w^{-s})}, 
\, \text{ and }\, 
[\mathcal{R} \widetilde{h}]_{A_1} \le 2 \|M\|_{L^{\tau'_p}(w^{-s})}.  
\end{align}

Write $\frac1r := \frac{1}{q_0} - \frac1q = \frac{1}{p_0} - \frac1p>0$, equivalently, $\frac{q}{q-q_0}=\frac{r}{q_0}$. By duality there exists a nonnegative function $h \in L^{\frac{q}{q-q_0}}(w^q)$ with $\|h\|_{L^{\frac{q}{q-q_0}}(w^q)} \le 1$ such that 
\begin{align}\label{HH-1}
\|fw\|_{L^q}^{q_0} 
=\|f^{q_0}\|_{L^{\frac{q}{q_0}}(w^q)}
= \int_{\Rn} f^{q_0} h \, w^q \, dx. 
\end{align}
Setting $H := \mathcal{R} \Big(h^{\frac{r}{\tau'_p q_0}} w^{\frac{s+q}{\tau'_p}} \Big)^{\frac{\tau'_p q_0}{r}} 
w^{-\frac{(s+q)q_0}{r}}$, 
we utilize \eqref{eq:RRdF} to obtain that $h \le H$, and by \eqref{eq:sharp} and \eqref{wsap}, 
\begin{align}\label{HHA1}
\big[H^{\frac{r}{\tau'_p q_0}} w^{\frac{s+q}{\tau'_p}} \big]_{A_1} 
= \Big[\mathcal{R} \Big(h^{\frac{r}{\tau'_p q_0}} w^{\frac{s+q}{\tau'_p}} \Big)\Big]_{A_1} 
\lesssim [w^{-s}]_{A_{\tau'_p}}^{\frac{1}{\tau'_p -1}} 
=[w^{p(\p_+/p)'}]_{A_{\tau_p}}, 
\end{align}
provided that 
\[
\|h^{\frac{r}{\tau'_p q_0}} w^{\frac{s+q}{\tau'_p}}\|_{L^{\tau'_p}(w^{-s})} 
= \|h\|_{L^{\frac{r}{q_0}}(w^q)}^{\frac{r}{\tau'_p q_0}}
= \|h\|_{L^{\frac{q}{q-q_0}}(w^q)}^{\frac{r}{\tau'_p q_0}}
\le 1, 
\]
which also gives 
\begin{align}\label{HH-2}
\|H\|_{L^{\frac{r}{q_0}}(w^q)}
= \Big\| \mathcal{R} \Big(h^{\frac{r}{\tau'_p q_0}} w^{\frac{s+q}{\tau'_p}} \Big) \Big\|_{L^{\tau'_p}(w^{-s})}^{\frac{\tau'_p q_0}{r}}
\le 2^{\frac{\tau'_p q_0}{r}} \Big\|h^{\frac{r}{\tau'_p q_0}} w^{\frac{s+q}{\tau'_p}} \Big\|_{L^{\tau'_p}(w^{-s})}^{\frac{\tau'_p q_0}{r}}
\le 2^{\frac{\tau'_p q_0}{r}}. 
\end{align}
Now picking $v := w^{\frac{q}{q_0}} \, H^{\frac{1}{q_0}}$, we see that by \eqref{HH-2}
\begin{align}\label{HH-3}
\|vw^{-1}\|_{L^r}
= \|H\|_{L^{\frac{r}{q_0}}(w^q)}^{\frac{1}{q_0}}
\le 2^{\frac{\tau'_p}{r}}
\le 2^{\frac{\tau'_p}{p_0}}. 
\end{align}

To proceed, we observe that $p_0 < p <\p_+$ and use \eqref{taup} to deduce that 
\begin{align}
\label{HH-4}
& \tau'_p \, p_0(\p_+/p_0)'/r
=\frac{\tau_p}{\tau_p -1} \frac{\frac{1}{p_0} - \frac1p}{\frac{1}{p_0} - \frac{1}{\p_+}}
=\frac{\tau_p}{\tau_p -1} \bigg(1- \frac{\tau_{p_0}}{\tau_p} \bigg)
=\frac{\tau_p - \tau_{p_0}}{\tau_p -1}, 
\\ 
\label{HH-5}
&\frac{\tau_p -1}{p(\p_+/p)'}
=\frac{1}{\p_-} - \frac1p
=\frac1s, 
\quad\text{ and }\quad 
\frac{\tau_{p_0} -1}{p_0(\p_+/p_0)'}
=\frac{1}{\p_-} - \frac{1}{p_0}
=: \frac{1}{s_0}, 
\end{align}
which in turn implies 
\begin{align}\label{HH-6}
\frac{q}{q_0} - \frac{q}{r} - \frac{s}{r} 
=1-\frac{s}{r}
=s \bigg(\frac1s - \frac1r\bigg)
=s \bigg(\frac{1}{\p_-} - \frac{1}{p_0}\bigg)
=\frac{s}{s_0} 
=\frac{p(\p_+/p)'}{p_0(\p_+/p_0)'} \frac{\tau_{p_0} -1}{\tau_p -1}. 
\end{align}
Hence, it follows from \eqref{HH-4} and \eqref{HH-6} that 
\begin{align*}
v^{p_0(\p_+/p_0)'} 
&=w^{(\frac{q}{q_0} - \frac{s}{r} - \frac{q}{r})p_0(\p_+/p_0)'} \, 
\big(H^{\frac{r}{\tau'_p q_0}} w^{\frac{s+q}{\tau'_p}} \big)^{\tau'_p p_0(\p_+/p_0)'/r} 
\\
&=\big(w^{p(\p_+/p)'}\big)^{\frac{\tau_{p_0} -1}{\tau_p -1}} 
\big(H^{\frac{r}{\tau'_p q_0}} w^{\frac{s+q}{\tau'_p}} \big)^{\frac{\tau_p - \tau_{p_0}}{\tau_p -1}}, 
\end{align*}
which along with \eqref{eq:wppAA}, \eqref{HHA1}, and Lemma \ref{lem:fac} part \eqref{list:fac-2}, yields 
\begin{align}\label{HHvv}
[v^{p_0(\p_+/p_0)'}]_{A_{\tau_{p_0}}} 
\le \big[w^{p(\p_+/p)'}\big]_{A_{\tau_p}}^{\frac{\tau_{p_0} -1}{\tau_p -1}}
\big[H^{\frac{r}{\tau'_p q_0}} w^{\frac{s+q}{\tau'_p}} \big]_{A_1}^{\frac{\tau_p - \tau_{p_0}}{\tau_p -1}}
\le C_2 \big[w^{p(\p_+/p)'}\big]_{A_{\tau_p}}, 
\end{align}
where the constant $C_2$ depends only on $n$, $p$, $p_0$, $\p_-$, and $\p_+$. By Lemma \ref{lem:weights} part \eqref{list:ApRH-2}, this means that 
\begin{align}\label{hvap}
v^{p_0} \in A_{p_0/\p_-} \cap RH_{(\p_+/p_0)'}.
\end{align}

With \eqref{HH-1} and \eqref{hvap} in hand, the hypothesis \eqref{eq:lim-off-1} implies 
\begin{align}\label{eq:lim-off-4}
\|fw\|_{L^q} 
&\le \bigg(\int_{\Rn} f^{q_0} H \, w^q \, dx \bigg)^{\frac{1}{q_0}} 
= \|f v\|_{L^{q_0}}
\le \Phi \big([v^{p_0(\p_+/p_0)'}]_{A_{\tau_{p_0}}}\big) \|g v\|_{L^{p_0}}
\\ \nonumber
&\le \Phi \big([v^{p_0(\p_+/p_0)'}]_{A_{\tau_{p_0}}}\big) \|g w\|_{L^p} \|vw^{-1}\|_{L^r}
\le 2^{\frac{\tau'_p}{p_0}} \Phi \big(C_2 [w^{p(\p_+/p)'}]_{A_{\tau_p}}\big) \|g w\|_{L^p}, 
\end{align}
where \eqref{HH-3} and \eqref{HHvv} were used in the last inequality. As a consequence, \eqref{eq:lim-off-2} follows at once from \eqref{eq:lim-off-3} and \eqref{eq:lim-off-4}. 
\end{proof}

\medskip 
\begin{proof}[\textbf{Proof of Theorem $\ref{thm:lim}$.}]
Fix $v_i^{q_i} \in A_{q_i/\p_i^-} \cap RH_{(\p_i^+/q_i)'}$, $i=2, \ldots, m$. Set 
\begin{align*}
\F_1 := \bigg\{(F, G) := \bigg(\frac{f \prod_{i=2}^m v_i}{\prod_{i=2}^m \|f_i v_i\|_{L^{q_i}}}, f_1 \bigg)
: (f, f_1, \ldots, f_m) \in \F \bigg\}.
\end{align*} 
By hypothesis \eqref{eq:lim-1}, we see that for every $v_1^{q_1} \in A_{q_1/\p_1^-} \cap RH_{(\p_1^+/q_1)'}$ 
\begin{align*}
\|F v_1\|_{L^q} = \frac{\|f v\|_{L^q}}{\prod_{i=2}^m \|f_i v_i\|_{L^{q_i}}} 
&\le \prod_{i=1}^m \Phi_i \big([v_i^{q_i}]_{A_{q_i/\p_i^-} \cap RH_{(\p_i^+/q_i)'}}\big)  \|f_1 v_1\|_{L^{q_1}}
\\ 
&= \prod_{i=1}^m \Phi_i \big([v_i^{q_i}]_{A_{q_i/\p_i^-} \cap RH_{(\p_i^+/q_i)'}}\big) 
\|G v_1\|_{L^{q_1}}, \quad (F, G) \in \F_1, 
\end{align*}
where $\frac1q = \sum_{i=1}^m \frac{1}{q_i}$ and $v=\prod_{i=1}^m v_i$. This verifies the hypothesis \eqref{eq:lim-off-1} for the family $\F_1$. Then Theorem \ref{thm:lim-off} implies that for every $p_1 \in (\p_1^-, \p_1^+)$ and every $w_1^{p_1} \in A_{p_1/\p_1^-} \cap RH_{(\p_1^+/p_1)'}$, 
\begin{multline}\label{eq:FG-1}
\frac{\|fw_1 \prod_{i=2}^m v_i\|_{L^{s_1}}}{\prod_{i=2}^m \|f_i v_i\|_{L^{q_i}}} 
= \|Fw_1\|_{L^{s_1}}  
\le \mathfrak{N}_1 
\prod_{i=2}^m \Phi_i \big([v_i^{q_i}]_{A_{q_i/\p_i^-} \cap RH_{(\p_i^+/q_i)'}}\big) \|G w_1\|_{L^{p_1}} 
\\
= \mathfrak{N}_1 
\prod_{i=2}^m \Phi_i \big([v_i^{q_i}]_{A_{q_i/\p_i^-} \cap RH_{(\p_i^+/q_i)'}}\big) 
\|f_1 w_1\|_{L^{p_1}}, \quad (F, G) \in \F_1, 
\end{multline}
where $\frac{1}{s_1} - \frac{1}{p_1} = \frac1q - \frac{1}{q_1}$, 
\begin{align}
& \mathfrak{N}_1 := 
2^{\max\{\frac{\tau_{p_1}}{p_1}, \frac{\tau'_{p_1}}{q_1}\}} 
\Phi_1 \Big(C_1 \, [w_1^{p_1}]_{A_{p_1/\p_1^-} \cap RH_{(\p_1^+/p_1)'}}^{\gamma_1(p_1, q_1)}\Big), 
\\ 
& \gamma_1(p_1, q_1) := 
\begin{cases}
\max\big\{1, \frac{\tau_{q_1} - 1}{\tau_{p_1} -1} \big\}, & q_1<\p_1^+, 
\\[4pt]
\frac{q_1}{\tau_{p_1} -1} \big(\frac{1}{\p_1^-} - \frac{1}{\p_1^+} \big), & q_1 = \p_1^+.
\end{cases}
\end{align} 
Considering \eqref{eq:FG-1}, we have 
\begin{align}\label{eq:FG-2}
\bigg\|fw_1\prod_{i=2}^m v_i \bigg\|_{L^{s_1}} 
\le \mathfrak{N}_1 \|f_1 w_1\|_{L^{p_1}} 
\prod_{i=2}^m \Phi_i  \big([v_i^{q_i}]_{A_{q_i/\p_i^-} \cap RH_{(\p_i^+/q_i)'}}\big) \|f_i v_i\|_{L^{q_i}}, 
\end{align}
for all $(f, f_1, \ldots, f_m) \in \F$, for all $p_1 \in (\p_1^-, \p_1^+)$, for all $w_1^{p_1} \in A_{p_1/\p_1^-} \cap RH_{(\p_1^+/p_1)'}$, and for all $v_i^{q_i} \in A_{q_i/\p_i^-} \cap RH_{(\p_i^+/q_i)'}$, $i=2, \ldots, m$. 

Now fix $p_1 \in (\p_1^-, \p_1^+)$, $w_1^{p_1} \in A_{p_1/\p_1^-} \cap RH_{(\p_1^+/p_1)'}$, and $v_i^{q_i} \in A_{q_i/\p_i^-} \cap RH_{(\p_i^+/q_i)'}$, $i=3, \ldots, m$. Set 
\begin{align*}
\F_2 := \bigg\{(F, G) := \bigg(\frac{f w_1 \prod_{i=3}^m v_i}{\|f_1 w_1\|_{L^{p_1}} \prod_{i=3}^m \|f_i v_i\|_{L^{q_i}}}, f_2 \bigg): (f, f_1, \ldots, f_m) \in \F \bigg\}.
\end{align*}
It follows from \eqref{eq:FG-2} that for every $v_2^{q_2} \in A_{q_2/\p_2^-} \cap RH_{(\p_2^+/q_2)'}$,  
\begin{align*}
\|Fv_2\|_{L^{s_1}} 
&=\frac{\|fw_1 \prod_{i=2}^m v_i\|_{L^{s_1}}}{\|f_1w_1\|_{L^{p_1}} \prod_{i=3} \|f_i v_i\|_{L^{q_i}}} 
\\
&\le \mathfrak{N}_1 
\prod_{i=2}^m \Phi_i  \big([v_i^{q_i}]_{A_{q_i/\p_i^-} \cap RH_{(\p_i^+/q_i)'}}\big) \|f_2 v_2\|_{L^{q_2}}
\\
&= \mathfrak{N}_1 
\prod_{i=2}^m \Phi_i  \big([v_i^{q_i}]_{A_{q_i/\p_i^-} \cap RH_{(\p_i^+/q_i)'}}\big) 
\|G v_2\|_{L^{q_2}}, \quad (F, G) \in \F_2. 
\end{align*}
Invoking Theorem \ref{thm:lim-off} applied to $\F_2$, we have that for every $p_2 \in (\p_2^-, \p_2^+)$ and every $w_2^{p_2} \in A_{p_2/\p_2^-} \cap RH_{(\p_2^+/p_2)'}$, 
\begin{align}\label{fws-2}
&\frac{\|fw_1w_2 \prod_{i=3}^m v_i\|_{L^{s_2}}}{\|f_1w_1\|_{L^{p_1}} \prod_{i=3} \|f_i v_i\|_{L^{q_i}}} 
=\|Fw_2\|_{L^{s_2}} 
\\ \nonumber
&\quad\le \mathfrak{N}_1 \mathfrak{N}_2  
\prod_{i=3}^m \Phi_i  \big([v_i^{q_i}]_{A_{q_i/\p_i^-} \cap RH_{(\p_i^+/q_i)'}}\big) \|G w_2\|_{L^{p_2}} 
\\ \nonumber
&\quad= \mathfrak{N}_1 \mathfrak{N}_2  
\prod_{i=3}^m \Phi_i  \big([v_i^{q_i}]_{A_{q_i/\p_i^-} \cap RH_{(\p_i^+/q_i)'}}\big) 
\|f_2 w_2\|_{L^{p_2}}, \quad (F, G) \in \F_2, 
\end{align}
where $\frac{1}{s_2} - \frac{1}{p_2} = \frac{1}{s_1} - \frac{1}{q_2}$, 
\begin{align}
& \mathfrak{N}_2 := 
2^{\max\{\frac{\tau_{p_2}}{p_2}, \frac{\tau'_{p_2}}{q_2}\}} 
\Phi_2 \Big(C_2 \, [w_2^{p_2}]_{A_{p_2/\p_2^-} \cap RH_{(\p_2^+/p_2)'}}^{\gamma_2(p_2, q_2)}\Big), 
\\ 
& \gamma_2(p_2, q_2) := 
\begin{cases}
\max\big\{1, \frac{\tau_{q_2} - 1}{\tau_{p_2} -1} \big\}, & q_2<\p_2^+, 
\\[4pt]
\frac{q_2}{\tau_{p_2} -1} \big(\frac{1}{\p_2^-} - \frac{1}{\p_2^+} \big), & q_2 = \p_2^+.
\end{cases}
\end{align} 
It follows from \eqref{fws-2} that for every $p_i \in (\p_i^-, \p_i^+)$,  for every $w_i^{p_i} \in A_{p_i/\p_i^-} \cap RH_{(\p_i^+/p_i)'}$, $i=1, 2$, and for every $v_j^{q_j} \in A_{q_j/\p_j^-} \cap RH_{(\p_j^+/q_j)'}$, $j=3, \ldots, m$, 
\begin{align*}
\bigg\|fw_1w_2 \prod_{j=3}^m v_j \bigg\|_{L^{s_2}}  
\le \prod_{i=1}^2 \mathfrak{N}_i \|f_i w_i\|_{L^{p_i}} 
\prod_{j=3}^m \Phi_j \big([v_j^{q_j}]_{A_{q_j/\p_j^-} \cap RH_{(\p_j^+/q_j)'}}\big) \|f_j w_j\|_{L^{q_j}}, 
\end{align*}
for all $(f, f_1, \ldots, f_m) \in \F$.

Inductively, one can show that for each $k \in \{1, \ldots, m\}$, for every $p_i \in (\p_i^-, \p_i^+)$,  for every $w_i^{p_i} \in A_{p_i/\p_i^-} \cap RH_{(\p_i^+/p_i)'}$, $i \in \{1, \ldots, k\}$, and for every $v_j^{q_j} \in A_{q_j/\p_j^-} \cap RH_{(\p_j^+/q_j)'}$, $j=\{k+1, \ldots, m\}$, 
\begin{align}\label{fws-22}
\bigg\|f \prod_{i=1}^k w_i \prod_{j=k+1}^m v_j \bigg\|_{L^{s_k}}  
\le \prod_{i=1}^k \mathfrak{N}_i \|f_i w_i\|_{L^{p_i}} 
\prod_{j=k+1}^m \Phi_j \big([v_j^{q_j}]_{A_{q_j/\p_j^-} \cap RH_{(\p_j^+/q_j)'}}\big) \|f_j v_j\|_{L^{q_j}}, 
\end{align}
for all $(f, f_1, \ldots, f_m) \in \F$, where $s_0 := q$, 
\begin{align}
\label{ss-k} & \frac{1}{s_k} - \frac{1}{p_k} = \frac{1}{s_{k-1}} - \frac{1}{q_k}, 
\\
& \mathfrak{N}_k := 
2^{\max\{\frac{\tau_{p_k}}{p_k}, \frac{\tau'_{p_k}}{q_k}\}} 
\Phi_k \Big(C_k \, [w_k^{p_k}]_{A_{p_k/\p_k^-} \cap RH_{(\p_k^+/p_k)'}}^{\gamma_k(p_k, q_k)}\Big), 
\\ 
& \gamma_k(p_k, q_k) := 
\begin{cases}
\max\big\{1, \frac{\tau_{q_k} - 1}{\tau_{p_k} -1} \big\}, & q_k<\p_k^+, 
\\[4pt]
\frac{q_k}{\tau_{p_k} -1} \big(\frac{1}{\p_k^-} - \frac{1}{\p_k^+} \big), & q_k = \p_k^+.
\end{cases}
\end{align} 
To conclude the proof, we take $\frac{1}{s_m} = \frac1p := \sum_{i=1}^m \frac{1}{p_i}$, and then \eqref{ss-k} is satisfied.  
The inequality \eqref{fws-22} immediately gives \eqref{eq:lim-2} as desired.

It remains to show the vector-valued inequality \eqref{eq:lim-3}. Fix $r_i \in (\p_i^-, \p_i^+)$, $i=1, \ldots, m$, and set $\frac{1}{r} = \sum_{i=1}^m \frac{1}{r_i}$. Given $N \in \N$, we define
\begin{align*}
\F_{\vec{r}}^N := \Big\{(F, F_1, \ldots, F_m)
:= &\bigg(\Big(\sum_{|k| < N} |f^k|^r \Big)^{\frac1r}, \Big(\sum_{|k| < N} |f^k_1|^{r_1} \Big)^{\frac{1}{r_1}},  \ldots,
\\
&\qquad \Big(\sum_{|k| < N} |f^k_m|^{r_m} \Big)^{\frac{1}{r_m}} \bigg): \{(f^k, f^k_1, \cdots, f^k_m)\}_k \subset \F \bigg\}.
\end{align*}
By \eqref{eq:lim-2}, for all $(F, F_1, \ldots, F_m) \in \F_{\vec{r}}^N$, and for all weights $v_i^{r_i} \in A_{r_i/\p_i^-} \cap RH_{(\p_i^+/r_i)'}$, $i=1, \ldots, m$,
\begin{align}\label{eq:FL-1}
\|F\|_{L^r(v^r)}
&= \bigg\| \Big(\sum_{|k| < N}  |f^k|^r \Big)^{\frac1r} \bigg\|_{L^r(v^r)}
= \bigg(\sum_{|k|<N} \|f^k\|_{L^r(v^r)}^r \bigg)^{\frac1r}
\\ \nonumber
&\le \bigg(\sum_{|k|<N} \prod_{i=1}^m \mathfrak{C}_{i, 1} 
\Phi_i \big(C_i [v_i^{r_i}]_{A_{r_i/\p_i^-} \cap RH_{(\p_i^+/r_i)'}}^{\gamma_i(r_i, q_i)} \big)^r  
\|f^k_i\|_{L^{r_i}(v_i^{r_i})}^r \bigg)^{\frac1r}
\\ \nonumber
&\le \prod_{i=1}^m \mathfrak{C}_{i, 1} 
\Phi_i \big(C_i [v_i^{r_i}]_{A_{r_i/\p_i^-} \cap RH_{(\p_i^+/r_i)'}}^{\gamma_i(r_i, q_i)}\big) 
\bigg(\sum_{|k|<N}  \|f^k_i\|_{L^{r_i}(v_i^{r_i})}^{r_i} \bigg)^{\frac{1}{r_i}}
\\ \nonumber
&= \prod_{i=1}^m \mathfrak{C}_{i, 1} 
\Phi_i \big(C_i [v_i^{r_i}]_{A_{r_i/\p_i^-} \cap RH_{(\p_i^+/r_i)'}}^{\gamma_i(r_i, q_i)}\big) \|F_i\|_{L^{r_i}(v_i^{r_i})}, 
\end{align}
where $ \mathfrak{C}_{i, 1} := 2^{\max\{\frac{\tau_{r_i}}{r_i}, \frac{\tau'_{r_i}}{q_i}\}}$.
This corresponds to \eqref{eq:lim-1} for the family $\F_{\vec{r}}^N$ and the exponent $\vec{r}=(r_1, \ldots, r_m)$. Then the estimate \eqref{eq:lim-2} applied to $\F_{\vec{r}}^N$ gives that for all exponents $p_i \in (\p_i^{-}, \p_i^{+})$ and all weights $w_i^{p_i} \in A_{p_i/\p_i^{-}} \cap RH_{(\p_i^{+}/p_i)'}$, $i=1, \ldots, m$,
\begin{align}\label{eq:FL-2}
\|F\|_{L^p(w^p)} 
\le \prod_{i=1}^m \mathfrak{C}_{i, 1} \mathfrak{C}_{i, 2} \, \Phi_i 
\big(C'_i \, [w_i^{p_i}]_{A_{p_i/\p_i^-} \cap RH_{(\p_i^+/p_i)'}}^{\gamma_i(p_i, r_i) \gamma_i(r_i, q_i)}\big) 
\|F_i\|_{L^{p_i}(w_i^{p_i})},  
\end{align}
for all $(F, F_1, \ldots, F_m) \in \F_{\vec{r}}^N$, where $\mathfrak{C}_{i, 2} := 2^{\max\{\frac{\tau_{p_i}}{p_i}, \frac{\tau'_{p_i}}{r_i}\}}$. The estimate \eqref{eq:FL-2} in turn implies
\begin{equation}\label{eq:FL-3}
\bigg\| \Big(\sum_{|k|<N} |f^k|^r \Big)^{\frac1r}\bigg\|_{L^p(w^p)}
\leq  \prod_{i=1}^m \mathfrak{C}'_i \, \Phi_i 
\big(C'_i \, [w_i^{p_i}]_{A_{p_i/\p_i^-} \cap RH_{(\p_i^+/p_i)'}}^{\gamma_i(p_i, r_i) \gamma_i(r_i, q_i)}\big) 
\bigg\| \Big(\sum_k |f^k_i|^{r_i} \Big)^{\frac{1}{r_i}}\bigg\|_{L^{p_i}(w_i^{p_i})},
\end{equation}
for all $\{(f^k, f^k_1, \cdots, f^k_m)\}_k \subset \F$, where $\mathfrak{C}'_i := 2^{\max\{\frac{\tau_{p_i}}{p_i}, \frac{\tau'_{p_i}}{r_i}\} + \max\{\frac{\tau_{r_i}}{r_i}, \frac{\tau'_{r_i}}{q_i}\}}$, and the constant $C'_i$ depends only on $n$, $p_i$, $q_i$, $r_i$, $\p_i^-$, and $\p_i^+$. Letting $N \to \infty$, we conclude \eqref{eq:lim-3} as desired.
\end{proof}

\begin{proof}[\textbf{Proof of Theorem $\ref{thm:lim-Tb}$.}]
Let $s_i \in (\p_i^-, \p_i^+)$, $i=1, \ldots, m$, be such that $\frac1s := \sum_{i=1}^m \frac{1}{s_i} \le 1$. It follows from \eqref{eq:lim-Tb-1} and Theorem \ref{thm:lim} that for all $v_i^{s_i} \in A_{s_i/\p_i^-} \cap RH_{(\p_i^+/s_i)'}$, $i=1, \ldots, m$, 
\begin{equation}\label{eq:CC-1}
\|T(\vec{f})\|_{L^s(v^s)} 
\leq C_0 \prod_{i=1}^m \Phi_i 
\big(C_i \, [v_i^{s_i}]_{A_{s_i/\p_i^-} \cap RH_{(\p_i^+/s_i)'}}^{\gamma_i(s_i, q_i)}\big) 
\|f_i\|_{L^{s_i}(v_i^{s_i})}, 
\end{equation}
where both $C_0$ and $C_i$ depend only on $n$, $s_i$, $q_i$, $\p_i^-$, and $\p_i^+$. 

Fix $\b = (b_1, \ldots, b_m) \in \BMO^m$ and multi-index $\alpha \in \N^m$. Given $v_i^{s_i} \in A_{s_i/\p_i^-} \cap RH_{(\p_i^+/s_i)'}$, $i=1, \ldots, m$, in light of Lemma \ref{lem:weights} part \eqref{list:ApRH-2}, we see that 
\begin{equation}\label{visi-1}
v_i^{s_i(\p_i^+/s_i)'} \in A_{\tau_{s_i}}, 
\end{equation} 
which together with Lemma \ref{lem:open} yields that there exists $\eta_i \in (1, 2)$ such that 
\begin{align}\label{visi-2}
\eta_i' \simeq [v_i^{s_i(\p_i^+/s_i)'}]_{A_{\tau_{s_i}}}^{\max\{1, \frac{1}{\tau_{s_i}-1}\}}, 
\quad \text{and}\quad 
[v_i^{\eta_i s_i(\p_i^+/s_i)'}]_{A_{\tau_{s_i}}}^{\frac{1}{\eta_i}} 
\le 2^{\tau_{s_i}} [v_i^{s_i(\p_i^+/s_i)'}]_{A_{\tau_{s_i}}}. 
\end{align}
Then in view of \eqref{eq:CC-1}--\eqref{visi-2}, Theorem \ref{thm:TTb} applied to $p:=s \ge 1$, $p_i:=s_i$, $r_i := \tau_{s_i}$, and $\theta_i:=s_i(\p_i^+/s_i)$, gives that for all $v_i^{s_i} \in A_{s_i/\p_i^-} \cap RH_{(\p_i^+/s_i)'}$, $i=1, \ldots, m$, 
\begin{multline}\label{eq:CC-2}
\|[T, \b]_{\alpha}(\vec{f})\|_{L^s(v^s)} 
\leq C_0 \prod_{i=1}^m (\eta'_i)^{\alpha_i} \Phi_i 
\big(C_i \, [v_i^{\eta_i s_i(\p_i^+/s_i)'}]_{A_{\tau_{s_i}}}^{\frac{1}{\eta_i} \gamma_i(s_i, q_i)}\big) 
\|b_i\|_{\BMO}^{\alpha_i} \|f_i\|_{L^{s_i}(v_i^{s_i})} 
\\ 
\leq C_0 \prod_{i=1}^m [v_i^{s_i(\p_i^+/s_i)'}]_{A_{\tau_{s_i}}}^{\alpha_i \max\{1, \frac{1}{\tau_{s_i}-1}\}} 
\Phi_i \big(C_i \, [v_i^{s_i(\p_i^+/s_i)'}]_{A_{\tau_{s_i}}}^{\gamma_i(s_i, q_i)}\big) 
\|b_i\|_{\BMO}^{\alpha_i} \|f_i\|_{L^{s_i}(v_i^{s_i})}, 
\end{multline}
where $C_i$ depends only on $n$, $s_i$, $q_i$, $\p_i^-$, and $\p_i^+$, and $C_0$ depends only on the same parameters and additionally on $\alpha$. 

Observe that for each $i=1, \ldots, m$, 
\begin{align}\label{eq:CC-3}
\widetilde{\Phi}_i(t)
:=  t^{\alpha_i \max\{1, \frac{1}{\tau_{s_i}-1}\}} \Phi_i(C_i \, t^{\gamma_i(s_i, q_i)}) 
\text{ is an increasing function}. 
\end{align}
Now with \eqref{eq:CC-2} and \eqref{eq:CC-3} in hand, we use Theorem \ref{thm:lim} applied to $s_i$ and $C_0^{\frac1m} \widetilde{\Phi}_i$ in place of $q_i$ and $\Phi_i$ to deduce that for all exponents $p_i, r_i \in (\p_i^{-}, \p_i^{+})$ and for all weights $w_i^{p_i} \in A_{p_i/\p_i^{-}} \cap RH_{(\p_i^{+}/p_i)'}$, $i=1, \ldots, m$, 
\begin{align*}
\|[T, \b]_{\alpha}(\vec{f})\|_{L^p(w^p)} 
\leq C_0 \prod_{i=1}^m \widetilde{\Phi}_i \big(C'_i \, [w_i^{p_i(\p_i^+/p_i)'}]_{A_{\tau_{p_i}}}^{\gamma_i(p_i, s_i)}\big) 
\|b_i\|_{\BMO}^{\alpha_i} \|f_i\|_{L^{p_i}(w_i^{p_i})}, 
\end{align*}
where $C_0$ depends only on $\alpha$, $n$, $p_i$, $q_i$, $s_i$, $\p_i^-$, and $\p_i^+$, $C'_i$ depends only on $n$, $p_i$, $s_i$, $\p_i^-$, and $\p_i^+$, and 
\begin{align*}
\bigg\| \Big(\sum_k | [T, \b]_{\alpha}(\vec{f}^k)|^r \Big)^{\frac1r} \bigg\|_{L^p(w^p)}
& \leq C  \prod_{i=1}^m \widetilde{\Phi}_i \big(C''_i \, [w_i^{p_i}]_{A_{p_i/\p_i^-} \cap RH_{(\p_i^+/p_i)'}}^{\gamma_i(p_i, r_i) \gamma_i(r_i, s_i)}\big)  
\\
&\qquad\qquad\times \|b_i\|_{\BMO}^{\alpha_i}
\bigg\| \Big(\sum_k |f^k_i|^{r_i} \Big)^{\frac{1}{r_i}}\bigg\|_{L^{p_i}(w_i^{p_i})},
\end{align*}
where $\frac{1}{r} = \sum_{i=1}^m \frac{1}{r_i}$, $C$ depends only on $\alpha$, $n$, $p_i$, $q_i$, $r_i$, $s_i$, $\p_i^-$, and $\p_i^+$, and $C''_i$ depends only on $n$, $p_i$, $r_i$, $s_i$, $\p_i^-$, and $\p_i^+$. This completes the proof of Theorem \ref{thm:lim-Tb}.
\end{proof}

\section{Applications}\label{sec:app}
This section is dedicated to using extrapolation to prove quantitative weighted inequalities for a variety of operators. This also shows that extrapolation theorems are useful and powerful.

\subsection{Bilinear Bochner-Riesz means}

Given $\delta \in \R$, the bilinear Bochner-Riesz means of order $\delta$ is defined by 
\begin{align*}
\B^{\delta}(f_1, f_2)(x) := \int_{\R^{2n}} (1-|\xi_1|^2-|\xi_2|^2)^{\delta}_{+} \, 
\widehat{f_1}(\xi_1) \, \widehat{f_2}(\xi_2) e^{2\pi i x \cdot(\xi_1 + \xi_2)} d\xi_1 d\xi_2. 
\end{align*}

\begin{theorem}\label{thm:BR-1}
Let $n \ge 2$ and $\delta \ge n-1/2$. Then for all $p_i \in (1, \infty)$, for all $w_i^{p_i} \in A_{p_i}$, for all $\b=(b_1, b_2) \in \BMO^2$, and for each multi-index $\alpha \in \N^2$, 
\begin{align}
\label{eq:Bn-1} \|\B^{\delta}(f_1, f_2)\|_{L^p(w^p)} 
& \lesssim \prod_{i=1}^2 [w_i^{p_i}]_{A_{p_i}}^{\beta_i(\delta)} \|f_i\|_{L^{p_i}(w_i^{p_i})}, 
\\
\label{eq:Bn-2} \|[\B^{\delta}, \b]_{\alpha}(f_1, f_2)\|_{L^p(w^p)} 
& \lesssim \prod_{i=1}^2 [w_i^{p_i}]_{A_{p_i}}^{\eta_i(\delta)}
\|b_i\|_{\BMO}^{\alpha_i} \|f_i\|_{L^{p_i}(w_i^{p_i})}, 
\end{align}
whenever $\frac1s := \frac{1}{s_1} + \frac{1}{s_2} \le 1$ with $s_1, s_2 \in (1, \infty)$, where $w=w_1 w_2$, $\frac1p=\frac{1}{p_1}+\frac{1}{p_2}$, 
\begin{equation*}
\beta_i(\delta) =
\begin{cases}
\frac{1}{p_i-1}, & \delta>n-1/2, 
\\ 
\max\{1, \frac{1}{p_i-1}\}, & \delta=n-1/2, 
\end{cases}
\end{equation*}  
and 
\begin{equation*}
\eta_i(\delta) =
\begin{cases}
(\alpha_i+\frac{1}{p_i-1}) \max\{1, \frac{1}{s_i-1}, \frac{1}{p_i-1}, \frac{s_i-1}{p_i-1}\}, & \delta>n-1/2, 
\\ 
(\alpha_i+1)\max\{1, \frac{1}{s_i-1}, \frac{1}{p_i-1}, \frac{s_i-1}{p_i-1}\}, & \delta=n-1/2. 
\end{cases}
\end{equation*}  
\end{theorem}

\begin{proof}
Let us first consider the case $\delta>n-1/2$. In this case, it was shown in \cite[Lemma 3.1]{JSS} that 
\begin{align}\label{eq:BMM}
|\B^{\delta}(f_1, f_2)(x)| 
\lesssim Mf_1(x) Mf_2(x), \qquad x \in \Rn, 
\end{align}
where the implicit constant is independent of $x$, $f_1$, and $f_2$. Combining \eqref{eq:BMM} with \eqref{eq:sharp} and H\"{o}lder's inequality, we obtain that for all $p_i \in (1, \infty)$ and for all $w_i^{p_i} \in A_{p_i}$, 
\begin{align}\label{eq:Bdel-1} 
\|\B^{\delta}(f_1, f_2)\|_{L^p(w^p)} 
& \lesssim \prod_{i=1}^2 [w_i^{p_i}]_{A_{p_i}}^{\frac{1}{p_i-1}} \|f_i\|_{L^{p_i}(w_i^{p_i})}, 
\end{align}
where $w=w_1 w_2$ and $\frac1p=\frac{1}{p_1}+\frac{1}{p_2}$. Then it follows from \eqref{eq:Bdel-1}, Theorem \ref{thm:lim-Tb}, and Remark \ref{rem:pp} that for all $\b=(b_1, b_2) \in \BMO^2$ and for each multi-index $\alpha \in \N^2$, 
\begin{align*}
\|[\B^{\delta}, \b]_{\alpha}(f_1, f_2)\|_{L^p(w^p)} 
\lesssim \prod_{i=1}^2 [w_i^{p_i}]_{A_{p_i}}^{(\alpha_i+\frac{1}{p_i-1}) \max\{1, \frac{1}{s_i-1}, \frac{1}{p_i-1}, \frac{s_i-1}{p_i-1}\}}
\|b_i\|_{\BMO}^{\alpha_i} \|f_i\|_{L^{p_i}(w_i^{p_i})}, 
\end{align*}
whenever $s_1, s_2 \in (1, \infty)$ satisfy $\frac1s := \frac{1}{s_1} + \frac{1}{s_2} \le 1$. 

Next, we turn to the case $\delta=n-1/2$. Given $\varepsilon_1 \in (0, \frac12)$, and $\varepsilon_2>0$, we write 
\begin{align}\label{eq:dtt}
\delta(\theta) := (1+\varepsilon_1)(1-\theta) + \theta(n-1/2 + \varepsilon_2),\quad \theta \in (0, 1).
\end{align} 
We first claim that for any $u_1, u_2 \in A_2$, 
\begin{align}\label{eq:Ba-claim}
\|\B^{\delta(\theta)}(f_1, f_2)\|_{L^1(u^{\theta})} 
\le \phi_1(\varepsilon_1)^{1-\theta} \phi_2(\varepsilon_2)^{\theta} 
\prod_{i=1}^2 [u_i]_{A_2}^{\theta} \|f_i\|_{L^2(u_i^{\theta})}, 
\quad\forall \theta \in (0, 1), 
\end{align}
where $u=u_1 u_2$ and the constant $ \phi_1, \phi_2$ are non-negative function and $\phi_2$ is increasing. Indeed, \eqref{eq:Ba-claim} can be obtained by following the proof of \cite[Theorem 1.8]{JSS}. We here only mention the difference: 
\begin{align*}
\sup_{t \in \R} |\psi(it)| 
&\le \phi_1(\varepsilon_1) \|h\|_{L^{\infty}(\Rn)} \prod_{i=1}^2 \|f_i\|_{L^2(\Rn)}, 
\\ 
\sup_{t \in \R} |\psi(1+it)| 
&\le \phi_2(\varepsilon_2) \|h\|_{L^{\infty}(\Rn)} \prod_{i=1}^2 [u_i]_{A_2} \|f_i\|_{L^2(\Rn)}, 
\end{align*}
provided the sharp estimate for the Hardy-Littlewood maximal operator in \eqref{eq:sharp}.  

Now let $v_1^2, v_2^2 \in A_2$, $v:=v_1 v_2$, and by Lemma \ref{lem:open}, there exists $\gamma \in (0, 2^{-n-3})$ such that 
\begin{align}\label{eq:aa22}
[v_i^{2(1+\gamma)}]_{A_2} 
\le 2^{2(1+\gamma)} [v_i^2]_{A_2}^{1+\gamma}. 
\end{align}
Then, \eqref{eq:Ba-claim} applied to $u_i=v_i^{2(1+\gamma)}$, $i=1, 2$, gives the for any $\theta \in (0, 1)$, 
\begin{multline}\label{eq:Baa}
\|\B^{\delta(\theta)}(f_1, f_2)\|_{L^1(v^{2(1+\gamma)\theta})} 
\le \phi_1(\varepsilon_1)^{1-\theta} \phi_2(\varepsilon_2)^{\theta}  
\prod_{i=1}^2 [v_i^{2(1+\gamma)}]_{A_2}^{\theta} \|f_i\|_{L^2(v_i^{2(1+\gamma)\theta})}
\\ 
\le \phi_1(\varepsilon_1)^{1-\theta} \phi_2(\varepsilon_2)^{\theta} 2^{4(1+\gamma) \theta}
\prod_{i=1}^2 [v_i^2]_{A_2}^{(1+\gamma)\theta} \|f_i\|_{L^2(v_i^{2(1+\gamma)\theta})}, 
\end{multline}
where \eqref{eq:aa22} was used in the last step. Picking $\theta =(1+\gamma)^{-1}$, $\varepsilon_1=1/4$, and $\varepsilon_2=(n-7/4) \gamma$, we utilize \eqref{eq:dtt} and \eqref{eq:Baa} to deduce that $\delta(\theta)=n-1/2$ and 
\begin{align}\label{eq:BBaa}
\|\B^{n-1/2}(f_1, f_2)\|_{L^1(v^2)} 
\lesssim \prod_{i=1}^2 [v_i^2]_{A_2} \|f_i\|_{L^2(v_i^2)}, 
\end{align}
where we had used that $\phi_1(\varepsilon_1)^{1-\theta} \phi_2(\varepsilon_2)^{\theta} \le \max\{1, \phi_1(1/4), \phi_2(n)\}$, and the implicit constant depends only on $n$. 

Having proved \eqref{eq:BBaa} and invoking Theorems \ref{thm:lim} and \ref{thm:lim-Tb} applied to $\p_i^-=1$, $\p_i^+=\infty$, $q_i=2$, and $\Phi_i(t)=t$, we conclude \eqref{eq:Bn-1} and \eqref{eq:Bn-2}. 
\end{proof}

The next result considers the case $\delta<n-\frac12$, which can be viewed as a complement of Theorem \ref{thm:BR-1}. 

\begin{theorem}\label{thm:BR-2}
Let $n \ge 2$, $0<\delta<n-\frac12$, and $0<\delta_1,\delta_2\le \frac n2$ be such that $\delta_1+\delta_2<\delta$.  Set $\p_1^- := \frac{2n}{n+2\delta_1}$, $\p_2^- := \frac{2n}{n+2\delta_2}$, and $\p_1^+ = \p_2^+ :=2$. Then for all $w_i^2 \in A_{2/\p_i^-}\cap RH_{(\p_i^+/2)'}$, $i=1, 2$, 
\begin{align}\label{eq:Bn-21} 
\|\B^{\delta}(f_1, f_2)\|_{L^1(w)} 
& \lesssim \prod_{i=1}^2 [w_i^2]_{A_{2/\p_i^-} \cap RH_{(\p_i^+/2)'}} 
\|f_i\|_{L^2(w_i^2)}. 
\end{align}
Moreover, for all $p_i \in (\p_i^-, \p_i^+)$ and for all $w_i^{p_i} 
\in A_{p_i/\p_i^-}\cap RH_{(\p_i^+/p_i)'}$, $i=1, 2$, 
\begin{align}\label{eq:Bn-22} 
\|\B^{\delta}(f_1, f_2)\|_{L^p(w^p)} 
& \lesssim \prod_{i=1}^2 [w_i^{p_i}]_{A_{p_i/\p_i^-} \cap RH_{(\p_i^+/p_i)'}}^{\gamma_i(p_i, 2)} 
\|f_i\|_{L^{p_i}(w_i^{p_i})}, 
\end{align}
where $w=w_1 w_2$ and $\frac1p=\frac{1}{p_1}+\frac{1}{p_2}$.  
\end{theorem}

\begin{proof}
We modify the proof of \cite[Theorem 2]{LW} into the current setting. First, choose a nonnegative function $\varphi \in \mathscr{C}_c^{\infty}(0,\,\infty)$ satisfying $\supp \varphi \subset (\frac12,\,2)$ and $\sum_{j \in \Z} \varphi(2^j s)=1$ for any $s>0$. For each $j \ge 0$, we define the bilinear operator
\begin{equation*}
T_j(f_1, f_2) 
:= \int_{0}^{\infty}\int_{0}^{\infty}\varphi_j^\delta(\lambda_1,\lambda_2)
R_{\lambda_1}f_1 \, R_{\lambda_2}f_2 \,\lambda_1^{n-1}\lambda_2^{n-1} \, d\lambda_1 \, d\lambda_2,
\end{equation*}
where
\begin{align*}
\varphi_j^\delta(s_1, s_2) & := (1-s_1^2-s_2^2)_+^\delta \, \varphi(2^j(1-s_1^2-s_2^2)),
\\
R_{\lambda}f(x) & := \int_{\mathbb{S}^{n-1}}\widehat f(\lambda \omega) 
e^{2\pi ix\cdot \lambda \omega} \, d\sigma(\omega),\ \ \lambda >0.
\end{align*}
Here $d\sigma$ is the surface measure on $\mathbb{S}^{n-1}$. Then one has 
\begin{equation}\label{eq:BR-0}
\B^\delta=\sum_{j=0}^{\infty}T_j.
\end{equation}
Given $j\ge0$, let $B_j=\{x\in\Rn: |x| < 2^{j(1+\gamma)}\}$ with $\gamma>0$ chosen later, and split the kernel function $K_j$ of $T_j$ into four parts:
\begin{align*}
&K_j^1(y_1, y_2) := K_j(y_1, y_2) \mathbf{1}_{B_j}(y_1) \mathbf{1}_{B_j}(y_2),\ \ \ 
K_j^2(y_1, y_2) := K_j(y_1, y_2) \mathbf{1}_{B_j}(y_1) \mathbf{1}_{B_j^c}(y_2),
\\
&K_j^3(y_1, y_2) := K_j(y_1, y_2) \mathbf{1}_{B_j^c}(y_1) \mathbf{1}_{B_j}(y_2),\ \ \ 
K_j^4(y_1, y_2) := K_j(y_1, y_2) \mathbf{1}_{B_j^c}(y_1) \mathbf{1}_{B_j^c}(y_2). 
\end{align*}
Letting $T_j^\ell$ denote the bilinear operator with kernel $K_j^\ell$, $\ell=1,2,3,4$, we see that 
\begin{align}\label{eq:BRTj}
T_j = T_j^1 + T_j^2 + T_j^3 + T_j^4.
\end{align}
Note that a straightforward calculation gives 
\begin{equation}\label{eq:BR-11}
|K_j(x_1,x_2)|
\lesssim 2^{-j\delta}2^{-j}(1+2^{-j}|x_1|)^{-N}(1+2^{-j}|x_2|)^{-N}, 
\quad\forall N>0, 
\end{equation}
and 
\begin{align}\label{eq:BR-13}
\int_{B_j^c} \frac{|f(x-y)|}{(1+2^{-j}|y|)^N} dy 
&=\sum_{k=0}^{\infty} \int_{2^{k+j(1+\gamma)} \le |y| < 2^{k+1+j(1+\gamma)}}
\frac{|f(x-y)|}{(1+2^{-j}|y|)^N} dy 
\\ \nonumber
&\le \sum_{k=0}^{\infty} \frac{(2^{k+2+j(1+\gamma)})^n}{(1+2^{k+j\gamma})^N} 
\fint_{Q(x, 2^{k+2+j(1+\gamma)})} |f(y)| \, dy
\\ \nonumber
&\lesssim \sum_{k=0}^{\infty} 2^{-k(N-n)} 2^{-j(N\gamma-(1+\gamma)n)} Mf(x) 
\\ \nonumber
&\lesssim 2^{-j(N\gamma-(1+\gamma)n)}Mf(x), \quad\text{ provided } N>n.  
\end{align}

In view of Theorem \ref{thm:lim}, it suffices to prove \eqref{eq:Bn-21}. Now let $q_1=q_2=2$, $v_1^2 \in A_{q_1/\p_1^-} \cap RH_{(\p_1^+/q_1)'} = A_{1+\frac{2\delta_1}{n}} \cap RH_{\infty}$, and $v_2^2 \in A_{q_2/\p_2^-} \cap RH_{(\p_2^+/q_2)'} 
= A_{1+\frac{2\delta_2}{n}} \cap RH_{\infty}$. Considering \eqref{eq:BR-0}--\eqref{eq:BRTj}, we are reduced to showing that there exists $\varepsilon>0$ such that 
\begin{equation}\label{eq:BR-12}
\|T_j^{\ell}(f_1, f_2)\|_{L^1(v)}
\lesssim 2^{-\varepsilon j} \prod_{i=1}^2 [v_i^2]_{A_{2/\p_i^-} \cap RH_{(\p_i^+/2)'}} 
\|f_i\|_{L^2(v_1^2)}, \quad j \ge 0, \,  \ell=1, 2, 3, 4. 
\end{equation}
To control $T_j^4$, note that $v_1^2 \in A_{1 + \frac{2\delta_1}{n}} \subset A_2$ and $v_2^2 \in A_{1+\frac{2\delta_2}{n}} \subset A_2$ since $\max\{n+2\delta_1, n+2\delta_2\} \le 2n$. Using \eqref{eq:BR-11}, Cauchy-Schwarz inequality, \eqref{eq:BR-13}, \eqref{eq:sharp}, and \eqref{def:wpApRH}, we have
\begin{align}\label{eq:BR-14}
\|T_j^4(f_1, f_2)\|_{L^1(v)}
&\lesssim \int_{\Rn} \int_{B_j^c} \int_{B_j^c} \frac{|f_1(x-y_1)|}{(1+2^{-j}|y_1|)^N} 
\frac{|f_2(x-y_2)|}{(1+2^{-j}|y_2|)^N} dy_1 dy_2\,v(x)dx
\\ \nonumber
&\le 
\biggl(\int_{\Rn}\biggl(\int_{B_j^c} \frac{|f_1(x-y_1)|}{(1+2^{-j}|y_1|)^N} dy_1\biggr)^2v_1^2(x)dx \biggr)^{\frac12}
\\ \nonumber
&\qquad\quad\times
\biggl(\int_{\Rn}\biggl(\int_{B_j^c} \frac{|f_2(x-y_2)|}{(1+2^{-j}|y_2|)^N} dy_2\biggr)^2v_2^2(x)dx \biggr)^{\frac12} 
\\ \nonumber
&\lesssim 2^{-j[N\gamma - (1+\gamma)n]} \|Mf_1\|_{L^2(v_1^2)} \|Mf_2\|_{L^2(v_2^2)}
\\ \nonumber
&\lesssim 2^{-\varepsilon j} [v_1^2]_{A_2} [v_2^2]_{A_2} \|f_1\|_{L^2(v_1^2)} \|f_2\|_{L^2(v_2^2)}
\\ \nonumber
&\le 2^{-\varepsilon j} \prod_{i=1}^2 [v_i^2]_{A_{2/\p_i^-} \cap RH_{(\p_i^+/2)'}} \|f_i\|_{L^2(v_i^2)}, 
\end{align}
where in the second-to-last inequality we have picked $N>0$ large enough so that $N\gamma>(1+\gamma)n$, and then taken   $0<\varepsilon<N\gamma-(1+\gamma)n$. Similarly, 
\begin{align}\label{eq:BR-15} 
\|T_j^3(f_1, f_2)\|_{L^1(v)}
&\lesssim \biggl(\int_{\Rn}\biggl(\int_{|y_1| \ge 2^{j(1+\gamma)}}
\frac{|f_1(x-y_1)|}{(1+2^{-j}|y_1|)^N} dy_1\biggr)^2 v_1^2(x) dx\biggr)^{\frac12}
\\ \nonumber 
&\qquad\quad\times \biggl(\int_{\Rn}\biggl(\int_{|y_2| < 2^{j(1+\gamma)}}
\frac{|f_2(x-y_2)|}{(1+2^{-j} |y_2|)^N}dy_2 \biggr)^2v_2^2(x)dx\biggr)^{\frac12}
\\ \nonumber
&\lesssim 2^{-j[N\gamma-(1+\gamma)n]} 2^{j(1+\gamma)n} \|Mf_1\|_{L^2(v_1^2)} \|Mf_2\|_{L^2(v_2^2)} 
\\ \nonumber
&\lesssim 2^{-j[N\gamma-2(1+\gamma)n]} [v_1^2]_{A_2} [v_2^2]_{A_2} \|f_1\|_{L^2(v_1^2)} \|f_2\|_{L^2(v_2^2)}
\\ \nonumber 
&\le 2^{-\varepsilon j} \prod_{i=1}^2 [v_i^2]_{A_{2/\p_i^-} \cap RH_{(\p_i^+/2)'}} \|f_i\|_{L^2(v_i^2)}, 
\end{align}
where we have chosen $N>0$ sufficiently large so that $N\gamma - 2(1+\gamma)n > \varepsilon$. Symmetrically to $T_j^3(f_1, f_2)$, there holds 
\begin{equation*}\label{eq:BR-16}
\|T_j^2(f_1, f_2)\|_{L^1(v)}
\lesssim 2^{-\varepsilon j} \prod_{i=1}^2 [v_i^2]_{A_{2/\p_i^-} \cap RH_{(\p_i^+/2)'}} \|f_i\|_{L^2(v_i^2)}. 
\end{equation*}

Finally, to prove \eqref{eq:BR-12} for $T_j^1$, we proceed as follows. For fixed $y\in \Rn$, set $B_j(y, r)=\{x\in\Rn: |x-y|\le 2^{j(1+\gamma)} r\}$ with $r>0$, and split $f_1$ and $f_2$ into three parts, respectively: 
\[
f_1 = f_{1, 1} + f_{1, 2} + f_{1, 3}, \quad\text{and}\quad 
f_2 = f_{2, 1} + f_{2, 2} + f_{2, 3},
\] 
where
\begin{align*}
&f_{1, 1} := f_1 \mathbf{1}_{B_j(y,\frac34)},\quad 
f_{1, 2} := f_1 \mathbf{1}_{B_j(y,\frac54)\setminus B_j(y,\frac34)}, \quad
f_{1, 3} := f_1 \mathbf{1}_{B_j(y,\frac54)^c},
\\
&f_{2, 1} := f_2 \mathbf{1}_{B_j(y,\frac34)}, \quad 
f_{2, 2} := f_2 \mathbf{1}_{B_j(y,\frac54)\setminus B_j(y,\frac34)}, \quad  
f_{2, 3} := f_2 \mathbf{1}_{B_j(y,\frac54)^c}.
\end{align*}
We should mention that each $f_{1, i}$ and $f_{2, i}$, $i=1, 2, 3$, depend on the variable $y$. Let $x\in B_j(y, \frac14)$. Since $f_{1,3}$ is supported on $\Rn \setminus B_j(y,\frac54)$, it follows from $f_{1, 3}(x-y_1) \ne 0$ that $|x-y_1-y| \ge \frac54 2^{j(1+\gamma)}$, and so $|y_1| \ge 2^{j(1+\gamma)}$. Noting that the kernel $K_j^1$ is supported on $B_j\times B_j$, we get 
$T_j^1(f_{1, 3}, f_2)=0$. Similarly, $T_j^1(f_1, f_{2, 3})=0$. Hence, for any $x\in B_j(y, \frac14)$, 
\begin{align}\label{eq:TjTj1}
T_j^1(f_1, f_2)(x) 
& = T_j^1(f_{1, 1}, f_{2, 1})(x) + T_j^1(f_{1, 1}, f_{2, 2})(x) 
\\ \nonumber
&\qquad+ T_j^1(f_{1, 2}, f_{2, 1})(x) + T_j^1(f_{1, 2}, f_{2, 2})(x).
\end{align}
Since $f_{1, 2}$ and $f_{2, 2}$ are supported on $B_j(y,\frac54)\setminus B_j(y,\frac34)$, it follows from $f_{1, 2}(x-y_1) f_{2, 2}(x-y_2) \ne 0$ that $|y_1| \ge 2^{j(1+\gamma)-1}$ and $|y_2| \ge 2^{j(1+\gamma) - 1}$. Then, repeating the proof of \eqref{eq:BR-14} yields  
\begin{equation}\label{eq:BR-17}
\|T_j^1(f_{1, 2}, f_{2, 2})v\|_{L^1(B_j(y,\frac14))}
\le 2^{-\varepsilon j} \prod_{i=1}^2 [v_i^2]_{A_{2/\p_i^-} \cap RH_{(\p_i^+/2)'}} \|f_{i, 2}\|_{L^2(v_i^2)}. 
\end{equation}
Since $f_{1, 1}$ is supported on $B_j(y,\frac34)$, it follows from $f_{1, 1}(x-y_1) f_{2, 2}(x-y_2)\ne 0$ that $|y_1| \le 2^{j(1+\gamma)}$ and $|y_2| \ge 2^{j(1+\gamma)-1}$. Thus, we calculate much as in \eqref{eq:BR-15} to get 
\begin{equation}\label{eq:BR-18} 
\|T_j^1(f_{1,1}, f_{2,2})v\|_{L^1(B_j(y,\frac14))}
\le 2^{-\varepsilon j} \prod_{i=1}^2 [v_i^2]_{A_{2/\p_i^-} \cap RH_{(\p_i^+/2)'}} \|f_{i, i}\|_{L^2(v_i^2)}.  
\end{equation}
Symmetrically, 
\begin{equation}\label{eq:BR-19}
\|T_j^1(f_{1, 2}, f_{2,1})v\|_{L^1(B_j(y,\frac14))}
\le 2^{-\varepsilon j} \prod_{i=1}^2 [v_i^2]_{A_{2/\p_i^-} \cap RH_{(\p_i^+/2)'}} \|f_{i, 3-i}\|_{L^2(v_i^2)}. 
\end{equation}
It remains to consider $T_j^1(f_{1,1}, f_{2,1})$. Given $\m\in L^\infty(\R)$, set $T_{\m} h := \int_0^1 \m(\lambda) \, R_\lambda h \, \lambda^{n-1} \, d\lambda$. 
Then \cite[Lemma 3.1]{LW} states that 
\begin{equation}\label{eq:BR-20}
\|T_{\m} h\|_{L^2(\Rn)} \lesssim \|\m\|_{L^\infty(\Rn)} \|h\|_{L^p(\Rn)}, \quad 1\le p\le 2.
\end{equation}
Let $r=1+2\delta_2/n$. Then $v_2^2 \in A_r \cap RH_{\infty}$. Using \eqref{eq:BR-20} and H\"older's inequality, we have that for $h \in L^2(v_2^2)$ with $\supp h \subset B_j(y,\frac34)$, 
\begin{align}\label{eq:BR-21} 
&\|(T_{\m} h) v_2\|_{L^2(B_j(y,\frac14))}  
\le \|T_{\m} h\|_{L^2(B_j(y,\frac14))} \Bigl(\esssup_{B_j(y,\frac14)} v_2^2 \Bigr)^{\frac12} 
\\
& \lesssim \|\m\|_{L^\infty(\Rn)} [v_2^2]_{RH_{\infty}}^{\frac12} \|h\|_{L^{\frac2r}(B_j(y,\frac34))}
\biggl(\fint_{B_j(y,\frac14)} v_2^2 \, dz\biggr)^{\frac12}
\nonumber
\\
&\lesssim \|\m\|_{L^\infty(\Rn)} [w_2]_{RH_{\infty}}^{\frac12} \|h v_2\|_{L^2(B_j(y,\frac34)} |B_j(y,3/4)|^{\frac{r-1}{2}}
\nonumber
\\
&\qquad \times \biggl(\fint_{B_j(y,\frac34)}v_2^{2(1-r')}dz\biggr)^{\frac{r-1}{2}} 
\biggl(\fint_{B_j(y,\frac14)} v_2^2 \, dz\biggr)^{\frac12}
\nonumber
\\
&\lesssim 2^{j(1+\gamma)\delta_2} \|\m\|_{L^\infty(\Rn)} [v_2^2]_{RH_{\infty}}^{\frac12} [v_2^2]_{A_{1+2\delta_2/n}}^{\frac12} 
\|h v_2\|_{L^2(B_j(y,\frac34)} \nonumber
\\
&\le 2^{j(1+\gamma)\delta_2} \|\m\|_{L^\infty(\Rn)} 
[v_2^2]_{A_{q_2/\p_2^-} \cap RH_{(\p_+/q_2)'}}
\|h v_2\|_{L^2(B_j(y,\frac34)}, 
\nonumber
\end{align}
where the definition \eqref{def:wpApRH} was used in the last inequality. 
Similarly, 
\begin{align}\label{eq:BR-22}
\|(T_{\m} h) v_1\|_{L^2(B_j(y,\frac14))} 
\lesssim 2^{j(1+\gamma)\delta_1} \|\m\|_{L^\infty(\Rn)} 
[v_1^2]_{A_{q_1/\p_2^-} \cap RH_{(\p_+/q_1)'}} \|h v_1\|_{L^2(B_j(y,\frac34))}. 
\end{align}
Observe that 
\begin{align}\label{eq:BR-24} 
T_j^1 (f_{1,1}, f_{2,1})(x) = T_j (f_{1,1}, f_{2,1})(x), \quad x \in B_j(y, 1/4). 
\end{align}
As argued in \cite[(3.7)]{LW}, we utilize \eqref{eq:BR-21}--\eqref{eq:BR-24} to get that for any fixed $0<\kappa<\delta$, 
\begin{align}\label{eq:BR-25}
\|T_j^1(f_{1,1}, f_{2,1})v\|_{L^1(B_j(y,\frac14))}
&\lesssim 2^{-j(\delta-\kappa) + j(1+\gamma)(\delta_1+\delta_2)}
\prod_{i=1}^2 [v_i^2]_{A_{2/\p_i^-} \cap RH_{(\p_i^+/2)'}} \|f_{i, 1}\|_{L^2(v_i^2)}
\\ \nonumber
&\lesssim 2^{-\varepsilon j} \prod_{i=1}^2 [v_i^2]_{A_{2/\p_i^-} \cap RH_{(\p_i^+/2)'}} \|f_{i, 1}\|_{L^2(v_i^2)}, 
\end{align}
provided choosing $\kappa, \gamma, \varepsilon$ small enough so that 
$\delta-\kappa -(1+\gamma)(\delta_1+\delta_2)>\varepsilon $. 
Summing \eqref{eq:TjTj1}--\eqref{eq:BR-19} and \eqref{eq:BR-25} yields
\begin{align}\label{eq:BR-27}
\|T_j^1(f_1, f_2)v\|_{L^1(B_j(y,\frac14))}
\lesssim 2^{-\varepsilon j} \prod_{i=1}^2 [v_i^2]_{A_{2/\p_i^-} \cap RH_{(\p_i^+/2)'}}  
\|f_i \mathbf{1}_{B_j(y,\frac54)}\|_{L^2(v_i^2)}.
\end{align}
Now, integrating the both sides of \eqref{eq:BR-27} with respect to $y$, using Cauchy-Schwarz inequality, and interchanging the order of integration, we conclude 
\begin{equation*}
\|T_j^1(f_1, f_2)\|_{L^1(v)}
\lesssim 2^{-\varepsilon j} \prod_{i=1}^2 [v_i^2]_{A_{2/\p_i^-} \cap RH_{(\p_i^+/2)'}} \|f_i\|_{L^2(v_i^2)}. 
\end{equation*}
This shows \eqref{eq:BR-12} for $T_j^1$ and completes the whole proof. 
\end{proof}

\subsection{Bilinear rough singular integrals}
Given $\Omega \in L^q(\mathbb{S}^{2n-1})$ with $1 \le q \le \infty$ and $\int_{\mathbb{S}^{2n-1}} \Omega\, d\sigma=0$,  we define the rough bilinear singular integral operator $T_{\Omega}$ by
\begin{equation*}
T_{\Omega}(f, g)(x)=\mathrm{p.v.} \int_{\R^{2n}} K_{\Omega}(x-y, x-z) f(y) g(z) dy dz,
\end{equation*}
where the rough kernel $K_{\Omega}$ is given by $K_{\Omega}(y, z) = \frac{\Omega ((y, z)/|(y, z)|)}{|(y, z)|^{2n}}$.

A typical example of the rough bilinear operators is the Calder\'{o}n commutator defined in \cite{P.C} as
\begin{align*}
\mathcal{C}_a (f)(x) := \text{p.v.} \int_{\R} \frac{A(x)-A(y)}{|x-y|^2} f(y) dy,
\end{align*}
where $a$ is the derivative of $A$. C. Calder\'{o}n \cite{C.C} established the boundedness of $\mathcal{C}_a$ in the full range of exponents $1<p_1, p_2< \infty$. It was shown in \cite{P.C} that the Calder\'{o}n commutator can be written as
\begin{align*}
\mathcal{C}_a(f)(x):= \text{p.v.} \int_{\R \times \R} K(x-y, x-z) f(y) a(z) \, dydz,
\end{align*}
where the kernel is given by 
\begin{align*}
K(y, z)=\frac{e(z)-e(z-y)}{y^2}=\frac{\Omega((y, z)/|(y,z)|)}{|(y, z)|^2},
\end{align*}
where $e(t)=1$ if $t>0$ and $e(t)=0$ if $t<0$. Observe that $K(y, z)$ is odd and homogeneous of degree $-2$ whose restriction on $\mathbb{S}^1$ is $\Omega(y,z)$. It is also easy to check that $\Omega$ is odd and bounded, and hence Theorems \ref{thm:rough}--\ref{thm:rough-2} below can be applied to $\mathcal{C}_a$.

\begin{theorem}\label{thm:rough}
Let $\Omega \in L^{\infty}(\mathbb{S}^{2n-1})$ and $\int_{\mathbb{S}^{2n-1}} \Omega\, d\sigma=0$. Then for all $p_i \in (1, \infty)$, for all $w_i^{p_i} \in A_{p_i}$, for all $\b=(b_1, b_2) \in \BMO^2$, and for each multi-index $\alpha \in \N^2$, 
\begin{align}
\label{eq:TOLp-1} \|T_{\Omega}(f_1, f_2)\|_{L^p(w^p)}
&\lesssim \prod_{i=1}^2 [w_i^{p_i}]_{A_{p_i}}^{\frac32 \max\{1, \frac{1}{p_i-1}\}} \|f_i\|_{L^{p_i}(w_i^{p_i})},
\\
\label{eq:TOLp-2} \|[T_{\Omega}, \b]_{\alpha}(f_1, f_2)\|_{L^p(w^p)}
&\lesssim  \prod_{i=1}^2 [w_i^{p_i}]_{A_{p_i}}^{(\alpha_i+\frac32)\max\{1, \frac{1}{s_i-1}, \frac{1}{p_i-1}, \frac{s_i-1}{p_i-1}\}}
\|b_i\|_{\BMO}^{\alpha_i} \|f_i\|_{L^{p_i}(w_i^{p_i})},
\end{align}
whenever $\frac1s := \frac{1}{s_1} + \frac{1}{s_2} \le 1$ with $s_1, s_2 \in (1, \infty)$, where $w=w_1 w_2$ and $\frac1p=\frac{1}{p_1}+\frac{1}{p_2}$.  
\end{theorem}

\begin{proof}
Picking $r_1=r_2=r_3=1$ and $p_1=p_2=q_1=q_2=2$, we see that \eqref{eq:rqpp} holds and $p_i \in (1, \infty)$, $i=1, 2$. Then Lemma \ref{lem:ApRHApr} gives that 
\begin{align}\label{eq:TO-A22}
[\vec{w}]_{A_{(2, 2)}} \le [w_1^2]_{A_2}^{\frac12} [w_2^2]_{A_2}^{\frac12}.
\end{align}
On the other hand, it was proved in \cite{CHS} that for every $\vec{w}=(w_1, w_2) \in A_{(2, 2)}$,
\begin{align}\label{eq:TO-1}
\|T_{\Omega}\|_{L^2(w_1^2) \times L^2(w_2^2) \to L^1(w)} 
\lesssim \|\Omega\|_{L^{\infty}} [\vec{w}]_{A_{(2, 2)}}^3
\lesssim \|\Omega\|_{L^{\infty}} [w_1^2]_{A_2}^{\frac32} [w_2^2]_{A_2}^{\frac32}, 
\end{align}
where \eqref{eq:TO-A22} was used in the last step. Thus, \eqref{eq:TOLp-1} and \eqref{eq:TOLp-2} follow at once from \eqref{eq:TO-1} and Theorems \ref{thm:lim} and \ref{thm:lim-Tb} applied to $\p_i^-=1$, $\p_i^+=\infty$, $q_i=2$, $\Phi_i(t)=t^{\frac32}$. 
\end{proof}

\begin{theorem}\label{thm:rough-2}
Let $\Omega \in L^q(\mathbb{S}^{2n-1})$ with $q>\frac43$ and $\int_{\mathbb{S}^{2n-1}} \Omega\, d\sigma=0$. Let $\pi_q < \p_i^- < \p_i^+ \le \infty$, $i=1, 2$, be such that $\frac{1}{\pi'_q} < \frac{1}{\p_+} := \frac{1}{\p_i^+} + \frac{1}{\p_2^+} <1$, where $\pi_q := \max\big\{\frac{24n+3q-4}{8n+3q-4}, \frac{24n+q}{8n+q}\big\}$. Then for all $p_i \in (\p_i^-, \p_i^+)$, for all $w_i^{p_i} \in A_{p_i/\p_i^-} \cap RH_{(\p_i^+/p_i)'}$, for all $\b=(b_1, b_2) \in \BMO^2$, and for each multi-index $\alpha \in \N^2$, 
\begin{align}
\label{eq:TOLq-1} & \|T_{\Omega}(f_1, f_2)\|_{L^p(w^p)}
\lesssim \prod_{i=1}^2 [w_i^{p_i(\p_i^+/p_i)'}]_{A_{\tau_{p_i}}}^{\theta (\frac{1}{p_i} - \frac{1}{\p_i^+})}
\|f_i\|_{L^{p_i}(w_i^{p_i})},
\\ 
\label{eq:TOLq-2}
& \|[T_{\Omega}, \b]_{\alpha} (f_1, f_2)\|_{L^p(w^p)}
\lesssim \prod_{i=1}^2 \Psi_i \big([w_i^{p_i(\p_i^+/p_i)'}]_{A_{\tau_{p_i}}}^{\gamma_i(p_i, s_i)}\big) 
\|b_i\|_{\BMO}^{\alpha_i} \|f_i\|_{L^{p_i}(w_i^{p_i})}, 
\end{align}
whenever $\frac1s := \frac{1}{s_1} + \frac{1}{s_2} \le 1$ with $s_i \in (\p_i^-, \p_i^+)$,  where $w=w_1 w_2$, $\frac1p = \frac{1}{p_1} + \frac{1}{p_2}$, $\frac{1}{r} = \frac{1}{r_1} + \frac{1}{r_2}$, 
\begin{align*}
\theta = \max_{i=1, 2} \bigg\{\frac{\frac{1}{\p_i^-}}{\frac{1}{\p_i^-} - \frac{1}{p_i}}, \frac{1-\frac{1}{\p_+}}{\frac1p - \frac{1}{\p_+}}\bigg\}, 
\quad \text{ and }\quad  
\Psi_i(t) := t^{\alpha_i \max\{1, \frac{1}{\tau_{s_i}-1}\} + \theta (\frac{1}{p_i} - \frac{1}{\p_i^+})}.
\end{align*}
\end{theorem}

\begin{proof}
By assumption, $p_0 := \min\{\p_1^-, \p_2^-, \p'_+\} > \pi_q$, which together with \cite[Theorem 1.1]{GWX} gives 
\begin{equation}\label{eq:TO-sparse}
|\langle T_{\Omega}(f_1, f_2), f_3 \rangle| 
\lesssim \sup _{\S: \text{ sparse}} \Lambda_{\S, (p_0, p_0, p_0)}(f_1, f_2, f_3)
\le \sup _{\S: \text{ sparse}} \Lambda_{\S, (\p_1^-, \p_2^-, \p'_+)}(f_1, f_2, f_3),
\end{equation}
for all $f_1, f_2, f_3 \in \mathscr{C}_c^{\infty}(\Rn)$. This and Theorem \ref{thm:m-linear} below imply \eqref{eq:TOLq-1} and \eqref{eq:TOLq-2} as desired. 
\end{proof}

\begin{remark}
In Theorem \ref{thm:rough}, the exponent $\p_i^- > \pi_q$ can be relaxed to $\p_i^- \ge \pi_q$, at the cost of a larger exponent appearing in \eqref{eq:TOLq-1} and \eqref{eq:TOLq-2}. Indeed, to get the first inequality in \eqref{eq:TO-sparse}, it requires that $p_0$ is strictly greater than $\pi_q$. When $\p_i^- = \pi_q$ and $w_i^{p_i} \in A_{p_i/\p_i^-} \cap RH_{(\p_i^+/p_i)'}$, Lemma \ref{lem:open} implies that there exists $\p_i^- < \widetilde{\p}_i^- < p_i$ such that $w_i^{p_i} \in A_{p_i/\widetilde{\p}_i^-} \cap RH_{(\p_i^+/p_i)'}$, $i=1, 2$. Then $p_0 := \min\{\widetilde{\p}_1^-, \widetilde{\p}_2^-, \p'_+\} > \pi_q$. Combining this with Lemma \ref{lem:open} and the result in the case $\p_i^- > \pi_q$, we can formulate similar estimates as in Theorem \ref{thm:rough}. Details are left to the reader. 
\end{remark}

Recall that a family $\S$ of cubes is called sparse if for every cube $Q \in \S$, there exists $E_Q \subset Q$ such that $|E_Q | \geq \eta |Q|$ for some $0<\eta<1$ and the collection $\{E_Q\}_{Q \in \S}$ is pairwise disjoint.

Given a sparse family $\S$ and $\vec{\s}=(\s_1, \ldots, \s_{m+1})$ with $\s_i \geq 1$, $i=1, \dots, m+1$, we define the $(m+1)$-sparse form
\begin{equation*}
\Lambda_{\S, \vec{\s}}(f_1, \ldots, f_{m+1})
:=\sum_{Q \in \S} |Q| \prod_{i=1}^{m+1} \bigg(\fint_{Q} |f_i|^{\s_i} dx \bigg)^{\frac{1}{\s_i}}.
\end{equation*}
We are interested in those operators $T$ that dominated by certain sparse form
\begin{equation}\label{eq:sparse-dom}
|\langle T(f_1, \ldots, f_m), f_{m+1}\rangle|
\leq C(\vec{\s}) \sup _{\S: \text{ sparse}} \Lambda_{\S, \vec{\s}}(f_1, \ldots, f_{m+1}),
\end{equation}
for all $f_1, \ldots, f_{m+1} \in \mathscr{C}_c^{\infty}(\Rn)$.

\begin{theorem}\label{thm:m-linear} 
Let $1 \leq \p_i^- < \p_i^+ \leq \infty$, $i=1, \dots, m$. Assume that the operator $T$ satisfies \eqref{eq:sparse-dom} for the exponents $\vec{\s} = (\p_1^-, \ldots, \p_m^-, \p'_+)$, where $\frac{1}{\p_+} := \sum_{i=1}^m \frac{1}{\p_i^+} < 1$. Then for all exponents $p_i, r_i \in (\p_i^{-}, \p_i^{+})$ and for all weights $w_i^{p_i} \in A_{p_i/\p_i^{-}} \cap RH_{(\p_i^{+}/p_i)'}$,
\begin{align}\label{eq:Tf-Lp}
\|T(\vec{f})\|_{L^p(w^p)}
&\lesssim \prod_{i=1}^m [w_i^{p_i(\p_i^+/p_i)'}]_{A_{\tau_{p_i}}}^{\theta(\frac{1}{p_i} - \frac{1}{\p_i^+})} 
\|f_i\|_{L^{p_i}(w_i^{p_i})},
\\
\label{eq:Tf-vector}
\bigg\|\Big(\sum_j |T(\vec{f}^j)|^r \Big)^{\frac1r}\bigg\|_{L^p(w^p)}
&\lesssim \prod_{i=1}^m [w_i^{p_i(\p_i^+/p_i)'}]_{A_{\tau_{p_i}}}^{\theta(\frac{1}{p_i} - \frac{1}{\p_i^+}) \gamma_i(p_i, r_i)} 
\bigg\|\Big(\sum_j |f_i^j|^{r_i}\Big)^{\frac{1}{r_i}}\bigg\|_{L^{p_i}(w_i^{p_i})},
\end{align}
where 
\begin{align*}
w=\prod_{i=1}^m w_i, \quad 
\frac1p = \sum_{i=1}^m \frac{1}{p_i}, \quad
\frac{1}{r} = \sum_{i=1}^m \frac{1}{r_i}, 
\quad \text{ and }\quad  
\theta = \max_{1\le i \le m} \bigg\{\frac{\frac{1}{\p_i^-}}{\frac{1}{\p_i^-} - \frac{1}{p_i}}, \frac{1-\frac{1}{\p_+}}{\frac1p - \frac{1}{\p_+}}\bigg\}.
\end{align*}

If in addition $T$ is an $m$-linear linear operator, then for the same exponents and weights as above, for all $\mathbf{b}=(b_1, \ldots, b_m) \in \BMO^m$, and for each multi-index $\alpha$,
\begin{align}
\label{eq:Tbf-Lp} & \|[T, \b]_{\alpha} (\vec{f})\|_{L^p(w^p)}
\lesssim \prod_{i=1}^m \Psi_i \big([w_i^{p_i(\p_i^+/p_i)'}]_{A_{\tau_{p_i}}}^{\gamma_i(p_i, s_i)}\big) 
\|b_i\|_{\BMO}^{\alpha_i} \|f_i\|_{L^{p_i}(w_i^{p_i})},
\\
\label{eq:Tbf-vector} & \bigg\|\Big(\sum_{j}\big|[T, \b]_{\alpha}(\vec{f}^j)\big|^r \Big)^{\frac1r}\bigg\|_{L^p(w^p)}
\lesssim \prod_{i=1}^m \Psi_i \big([w_i^{p_i(\p_i^+/p_i)'}]_{A_{\tau_{p_i}}}^{\gamma_i(p_i, r_i) \gamma_i(r_i, s_i)}\big) 
\\ \nonumber
&\qquad\qquad\qquad\qquad\qquad\qquad\qquad\qquad\quad\times \|b_i\|_{\BMO}^{\alpha_i}
\bigg\|\Big(\sum_{j} \big|f_i^j\big|^{r_i}\Big)^{\frac{1}{r_i}}\bigg\|_{L^{p_i}(w_i^{p_i})}, 
\end{align}
whenever $\frac1s := \sum_{i=1}^m \frac{1}{s_i} \le 1$ with $s_i \in (\p_i^-, \p_i^+)$, where 
$\Psi_i(t) := t^{\alpha_i \max\{1, \frac{1}{\tau_{s_i}-1}\} + \theta (\frac{1}{p_i} - \frac{1}{\p_i^+})}$.  
\end{theorem}

\begin{proof}
Let $p_i \in (\p_i^{-}, \p_i^{+})$ and $w_i^{p_i} \in A_{p_i/\p_i^{-}} \cap RH_{(\p_i^{+}/p_i)'}$, $i=1, \ldots, m$. 
By density, we may assume that $f_1, \ldots, f_m \in \mathscr{C}_c^{\infty}(\Rn)$ in this sequel. By Lemma \ref{lem:ApRHApr}, one has $\vec{w} \in A_{\vec{p}, \vec{r}}$ with 
\begin{align}\label{eq:Aprbdd-1}
[\vec{w}]_{A_{\vec{p}, \vec{\s}}}
\le \prod_{i=1}^m [w_i^{p_i(\p_i^+/p_i)'}]_{A_{\tau_{p_i}}}^{\frac{1}{p_i} - \frac{1}{\p_i^+}}.
\end{align}
Then it follows from \eqref{eq:Aprbdd-1} and \cite[Corollary 4.2]{Nie} that 
\begin{align}\label{eq:Aprbdd-2}
\|T(f_1, \ldots, f_m)\|_{L^p(w^p)}
\lesssim [\vec{w}]_{A_{\vec{p}, \vec{\s}}}^{\theta} \prod_{i=1}^m  \|f_i\|_{L^{p_i}(w_i^{p_i})}, 
\end{align}
Thus, \eqref{eq:Tf-Lp} is a consequence of \eqref{eq:Aprbdd-1} and \eqref{eq:Aprbdd-2}. With Theorem \ref{thm:lim} and Remark \ref{rem:pp} in hand, the estimate \eqref{eq:Tf-Lp} in turn gives \eqref{eq:Tf-vector}. Additionally, \eqref{eq:Tbf-Lp} and \eqref{eq:Tbf-vector} follow from \eqref{eq:Tf-Lp}, Theorem \ref{thm:lim-Tb}, and Remark \ref{rem:pp}. The proof is complete.
\end{proof}

We close the subsection with the following remark, which shows Theorems \ref{thm:lim}--\ref{thm:lim-Tb} and Theorem \ref{thm:m-linear} contain a lot of applications. Details are left to the interested reader. 

\begin{remark}
Now let us present some examples in terms of the hypothesis in Theorem \ref{thm:m-linear}.
\begin{enumerate}
\item[$\bullet$] In \cite{BFP}, Bernicot et al. established a bilinear sparse domination $\Lambda_{\S, p_0, q'_0}$ for singular non-integral operators under certain assumptions. This verifies the hypothesis \eqref{eq:sparse-dom} for $r_1=p_0$ and $r_2=q'_0$. Note also that our extrapolation theorems above can be extended to spaces of homogeneous type since the corresponding sharp estimate for the Hardy-Littlewood operator \eqref{eq:sharp} was established in \cite[Proposition 7.13]{HK}. 

\item[$\bullet$] For Bochner-Riesz means $\B^{\alpha}$ in $\R^2$,  the authors \cite{BBL} proved a similar spare bilinear form  to \eqref{eq:sparse-dom} with $r_1=6/5$ and $r_2=2$ whenever $\alpha>1/6$. Much as before, one can not only recover \cite[Theorem 1.2]{BBL}, but also obtain quantitative weighted estimates and vector-valued inequalities. 

\item[$\bullet$] Bui et. al \cite{BCDH} studied the Schr\"{o}dinger operator $L=\Delta+V$ on $\Rn$ with $n \geq 3$, where $V \in RH_q$ and $q \in (n/2, n)$. Letting $p_0=\big(\frac{1}{q}-\frac{1}{n}\big)^{-1}$ and $K(x, y)$ be the kernel of the Riesz transform $L^{-1/2}\nabla$, we see that $K$ satisfies the Bui-Duong's condition (cf. \cite[Theorem 5.6]{BCDH}). The latter implies $L^r$-H\"{o}rmander condition (cf. \cite[Proposition 3.2]{Li}).  Then, combining the $L^p$ bounds for $\nabla L^{-1/2}$ with $p \in (1, p_0]$ (cf. \cite{Shen}) and the pointwise sparse domination in \cite{Li}, we use a duality argument to conclude that there exists a sparse family $\S$ such that $|\langle \nabla L^{-1/2}f, g \rangle| \lesssim \Lambda_{\S, 1, p_0}(f, g)$. That is, the hypothesis \eqref{eq:sparse-dom} is satisfied for the Riesz transform $\nabla L^{-1/2}$. 

\item[$\bullet$] For the $m$-linear Calder\'{o}n-Zygmund operators and the corresponding maximal truncation, pointwise sparse dominations were obtained in \cite{CR, DHL}, which immediately implies \eqref{eq:sparse-dom} with $\vec{r}=(1, \ldots, 1)$. Then one can improve Corollaries 8.2 and 8.3 in \cite{GM} to the quantitative weighted estimates.  

\item[$\bullet$] Let $1 \leq r < \infty$ and $\mathfrak{g}$ be the square function with the kernel $K_t$ satisfies the $m$-linear $L^r$-H\"{o}rmander condition defined in \cite{CY}.  Under the assumption that $\mathfrak{g}$ is bounded from $L^r(\Rn) \times \cdots \times L^r(\Rn)$ to $L^{r/m, \infty}(\Rn)$,   Cao and Yabuta \cite{CY} obtained a pointwise control of $\mathfrak{g}$ by $\Lambda_{\S, \vec{r}}$, where $\vec{r}=(r, \ldots, r, 1)$. Then, the square function $\mathfrak{g}$ verifies \eqref{eq:sparse-dom}. 

\item[$\bullet$] The operators satisfying \eqref{eq:sparse-dom} also include the discrete cubic Hilbert transform \cite{CKL}
and oscillatory integrals \cite{LS}. 
\end{enumerate}
\end{remark}

\subsection{Multilinear Fourier multipliers}
Given $s, m \in \N$, a function $\sigma \in \mathscr{C}^s(\R^{nm} \setminus \{0\})$ is said to belong to $\mathcal{M}^s(\R^{nm})$ if 
\begin{align*}
\big|\partial_{\xi_1}^{\alpha_1} \cdots \partial_{\xi_m}^{\alpha_m}  \sigma(\vec{\xi}) \big|
&\leq C_{\alpha} (|\xi_1|+\cdots+|\xi_m|)^{-\sum_{i=1}^m |\alpha_i|}, \quad \forall \vec{\xi} \in \R^{nm} \setminus\{0\}, 
\end{align*}
for each multi-index $\alpha=(\alpha_1, \ldots, \alpha_m)$ with $\sum_{i=1}^m |\alpha_i| \le s$. 

Given $s \in \R$, the (usual) Sobolev space $W^s(\R^{nm})$ is defined by the norm
\begin{align*} 
\|f\|_{W^s(\R^{nm})} :=\bigg(\int_{\R^{nm}} (1+|\vec{\xi}|^2)^{s}|\widehat{f}(\vec{\xi})|^2 d\vec{\xi}\bigg)^{\frac12}, 
\end{align*}
where $\widehat{f}$ is the Fourier transform in all the variables. For $\vec{s}=(s_1,\ldots,s_m) \in \R^m$, the Sobolev space of product type $W^{\vec{s}}(\R^{nm})$ consists of all $f \in \S'(\R^{nm})$ such that 
\begin{align*} 
\|f\|_{W^{\vec{s}}(\R^{nm})} 
:=\bigg(\int_{\R^{nm}} (1+|\xi_1|^2)^{s_1} \cdots 
(1+|\xi_m|^2)^{s_m} |\widehat{f}(\vec{\xi})|^2 d\vec{\xi}\bigg)^{\frac12} < \infty. 
\end{align*} 
Given a function $\sigma$ on $\R^{nN}$, we set 
\begin{align}\label{def:mj}
\sigma_j(\vec{\xi}) :=\Psi(\vec{\xi}) \sigma(2^j \vec{\xi}), \quad j \in \Z, 
\end{align} 
where $\Psi \in \S(\R^{nm})$ satisfy $\supp \Psi \subset \{1/2 \le |\vec{\xi}| \le 2 \}$, and $\sum_{k \in \Z} \Psi(2^{-k}\vec{\xi}) =1$ for all $\vec{\xi} \in \R^{nm} \setminus\{0\}$. Denote  
\begin{align*}
\mathcal{W}^s(\R^{nm})
 &:= \big\{\sigma \in L^{\infty}(\R^{nm}): \sup_{j \in \Z} \|\sigma_j\|_{W^s(\R^{nm})}<\infty\big\}, 
\\
\mathcal{W}^{\vec{s}}(\R^{nm}) 
&:= \big\{\sigma \in L^{\infty}(\R^{nm}): \sup_{j \in \Z} \|\sigma_j\|_{W^{\vec{s}}(\R^{nm})}<\infty\big\}. 
\end{align*}
Then one has 
\begin{align*}
\mathcal{M}^{s}(\R^{nm}) \subsetneq \mathcal{W}^{s}(\R^{nm}) 
\subsetneq \mathcal{W}^{(\frac{s}{m},\ldots,\frac{s}{m})}(\R^{nm}). 
\end{align*}

For a bounded function $\sigma$ on $\R^{nm}$, the $m$-linear Fourier multiplier $T_{\sigma}$ is defined by 
\begin{align*}
T_{\sigma} & (\vec{f})(x)
:= \int_{\R^{nm}} e^{2\pi ix \cdot (\xi_1+\dots+\xi_m)}
\sigma(\vec{\xi}) \widehat{f}_1(\xi_1) \cdots \widehat{f}_m(\xi_m) \, d\vec{\xi}, 
\end{align*}
for all $f_1,\ldots f_m \in \S(\Rn)$. 

Be means of extrapolation theorems, we improve Theorems 1.2 (i) and 6.2 in \cite{FT} to the weighted estimates with quantitative bounds. We can also establish the corresponding weighted estimates for the higher order commutators and vector-valued inequalities as follows. 

\begin{theorem}\label{thm:FT-theorem}
Let $m \ge 2$, $n/2 < s_i \le n$, $i=1, \ldots, m$. Assume that $\sigma \in \mathcal{W}^{\vec{s}}(\R^{nm})$. Then for every $p_i>n/s_i$, for every $w_i^{p_i} \in A_{p_i s_i/n}$, $i=1, \ldots, m$, for all $\b=(b_1,\ldots, b_m) \in \BMO^m$, and for each multi-index $\alpha \in \N^m$, 
\begin{align}
\label{eq:FT-1}
\|T_{\sigma}(\vec f)\|_{L^p(w^p)} 
&\lesssim \prod_{i=1}^m [w_i^{p_i}]_{A_{p_i s_i/n}}^{\frac32 \gamma_i(p_i, 2m)} 
\|f_i\|_{L^{p_i}(w_i^{p_i})}, 
\\
\label{eq:FT-2} 
\|[T_{\sigma}, \b]_{\alpha}(\vec{f})\|_{L^p(w^p)} 
&\lesssim \prod_{i=1}^m [w_i^{p_i}]_{A_{p_is_i/n}}^{(\alpha_i + \frac32) \gamma_i(p_i, 2m)} 
\|b_i\|_{\BMO}^{\alpha_i} \|f_i\|_{L^{p_i}(w_i^{p_i})}, 
\end{align}
where $\frac1p = \sum_{i=1}^m \frac{1}{p_i}$ and $w=\prod_{i=1}^m w_i$.  

Moreover, for any $r \in (n/s_i, 2]$, $i=1, \ldots,m$, 
\begin{align*}
\bigg\| \bigg(\sum_{k_1, \ldots, k_m} |T_{\sigma}(f^1_{k_1}, \ldots, f^m_{k_m})|^r \bigg)^{\frac1r}\bigg\|_{L^p(w^p)}
&\lesssim \prod_{i=1}^m [w_i^{p_i}]_{A_{p_i s_i/n}}^{\beta_i(r)}  
\bigg\| \bigg(\sum_{k_i} |f^i_{k_i}|^r \bigg)^{\frac1r}\bigg\|_{L^{p_i}(w_i^{p_i})}, 
\\ 
\bigg\| \bigg(\sum_{k_1, \ldots, k_m} |[T_{\sigma}, \b]_{\alpha}(f^1_{k_1}, \ldots, f^m_{k_m})|^r \bigg)^{\frac1r}\bigg\|_{L^p(w^p)}
& \lesssim \prod_{i=1}^m [w_i^{p_i}]_{A_{p_i s_i/n}}^{\eta_i(r)} \|b_i\|_{\BMO}^{\alpha_i} 
\bigg\| \bigg(\sum_{k_i} |f^i_{k_i}|^r \bigg)^{\frac1r}\bigg\|_{L^{p_i}(w_i^{p_i})}, 
\end{align*}
where 
\begin{equation*}
\beta_i(r) := 
\begin{cases} 
\frac32 \gamma_i(p_i, 2m), & r=2, 
\\
\frac32 \gamma_i(p_i, q_i) \gamma_i(q_i, 2m), & r \neq 2, 
\end{cases} 
\qquad 
\eta_i(r) := 
\begin{cases}
(\alpha_i + \frac32) \gamma_i(p_i, 2m), & r=2, 
\\
(\alpha_i + \frac32) \gamma_i(p_i, q_i) \gamma_i(q_i, 2m), & r \neq 2, 
\end{cases} 
\end{equation*}
provided $q_i \in (n/s_i, r)$, $i=1, \ldots, m$. 
\end{theorem}

\begin{proof}
We borrow some ideas from \cite{FT}, but now we can give a proof without using the weighted Hardy space argument. Let $\p_i^- := n/s_i$ and $\p_i^+ := \infty$ for each $i=1, \ldots, m$. Let $q=2$ and $q_i=2m$ for $1\le i \le m$. Then, $q_i \in (\p_i^-, \p_i^+)$. Checking the proof of \cite[Theorem 6.2]{FT}, we can obtain that for any weight $v_i^{q_i} \in A_{q_is_i/n}=A_{q_i/\p_i^-} \cap RH_{(\p_i^+/q_i)'}$, $i=1, \ldots, m$, 
\begin{equation}\label{FT-Tm}
\|T_{\sigma}(\vec f)\|_{L^q(v^q)} 
\lesssim \prod_{i=1}^m [v_i^{q_i}]_{A_{q_is_i/n}}^{\frac32} \|f_i\|_{L^{q_i}(v_i^{q_i})}.
\end{equation}
Thus, \eqref{eq:FT-1} follows from \eqref{FT-Tm} and Theorem \ref{thm:lim} applied to $\Phi_i(t) = t^{3/2}$. 

Note that in the current scenario, $\gamma_i(q_i, q_i)=1$, $\tau_{q_i}=2ms_i/n$, and hence, 
\[
\widetilde{\Phi}_i(t) 
:= t^{\alpha_i \max\{1, \frac{1}{\tau_{q_i}-1}\}} \Phi_i(C_i \, t^{\gamma_i(q_i, q_i)})
= C_i^{\frac32}  \, t^{\frac32 + \alpha_i \max\{1, \frac{1}{2Ns_i/n-1}\}} 
= C_i^{\frac32}  \, t^{\frac32 + \alpha_i}. 
\]
Then in view of \eqref{FT-Tm}, Theorem \ref{thm:lim-Tb} applied to $s_i=q_i=2m$ implies \eqref{eq:FT-2}.

On the other hand, Lemma \ref{lem:MZ} and \eqref{eq:FT-1} give that for every $q_i>n/s_i$, for every $w_i^{q_i} \in A_{q_i s_i/n}$, $i=1, \ldots, m$,
\begin{align*}
\bigg\| \bigg(\sum_{k_1, \ldots, k_m} |T_{\sigma}(f^1_{k_1}, \ldots, f^m_{k_m})|^r \bigg)^{\frac1r}\bigg\|_{L^q(v^q)}
\lesssim \prod_{i=1}^m [v_i^{q_i}]_{A_{q_i s_i/n}}^{\frac32 \gamma_i(q_i, 2m)} 
\bigg\| \bigg(\sum_{k_i} |f^i_{k_i}|^r \bigg)^{\frac1r}\bigg\|_{L^{q_i}(v_i^{q_i})}, 
\end{align*}
provided $r=2$ or $r \in (n/s_i, 2)$ and $q_i \in (n/s_i, r)$, where $\frac1q=\sum_{i=1}^m \frac{1}{q_i}$ and $v=\prod_{i=1}^m v_i$. Therefore, the vector-valued inequalities above follow from Theorem \ref{thm:lim} applied to $\Phi_i(t) = t^{\frac32 \gamma_i(q_i, 2m)}$. 
\end{proof}

\begin{theorem}\label{thm:Ta} 
Let $1 \leq r \leq 2$ and $s_1, s_2>1/r$. Let $\sigma$ be a bounded function on $\R^2$ satisfying 
\begin{equation*}
\sup _{j \in \Z} \big\|(I-\Delta_{\xi_1})^{\frac{s_1}{2}} 
(I-\Delta_{\xi_2})^{\frac{s_2}{2}} \sigma_j\big\|_{L^r(\R^2)} < \infty, 
\end{equation*}
where $\sigma_j$ is given in \eqref{def:mj} with $n=1$. Assume that $1 \leq \p_1^{-}, \p_2^{-}<\infty$ and 
$\max\limits _{1 \leq i \leq 2}\frac{1}{s_i} < \min\limits _{1 \leq i \leq 2} \p_i^{-}$.  Then for all exponents $p_i \in (\p_i^{-}, \infty)$ and all weights $w_i^{p_i} \in A_{p_i/\p_i^-}$, $i=1, 2$, 
\begin{equation*}
\|T_{\sigma}(\vec{f})\|_{L^p(w^p)} 
\lesssim \prod_{i=1}^2 [w_i^{p_i}]_{A_{p_i/\p_i^-}}^{\frac{15}{p_i} + 2 \max\{\frac12, \frac{1}{p_i-\p_i^-}\}} 
\|f_i\|_{L^{p_i}(w_i^{p_i})}, 
\end{equation*}
where $\frac{1}{p} = \frac{1}{p_1} + \frac{1}{p_2} \ge 1$ and $w=w_1 w_2$.   
\end{theorem}

\begin{proof}
We will use the same notation as \cite{GHNY}. By the same argument as \cite[p. 970]{GHNY}, we are deduced to showing the boundedness of $T_{\sigma_1}$ and $T_{\sigma_2}$, which satisfy 
\begin{align}
\label{eq:Ta-1}
|\Delta_j^{\theta}(T_{\sigma_1}(f_1, f_2))| 
&\lesssim M(|f_1|^{\rho})^{\frac{1}{\rho}}
M\big(|\Delta_j^{\eta} f_2|^{\rho}\big)^{\frac{1}{\rho}}, 
\\
\label{eq:Ta-2} T_{\sigma_2}(f_1, f_2)
&=\sum_{j \in \Z} T_{\sigma_2}(f_1, \Delta_j^{\theta} f_2), 
\\ 
\label{eq:Ta-2-j} T_{\sigma_2}(f_1, \Delta_j^{\theta} f_2)
&\lesssim M\big(|\Delta_j^{\zeta} f_1|^{\rho}\big)^{\frac{1}{\rho}} 
M\big(|\Delta_j^{\theta} f_2|^{\rho}\big)^{\frac{1}{\rho}}. 
\end{align}
Here, $\rho \in (1, 2)$ satisfies $\max\limits _{i=1, 2}\frac{1}{s_i} < \rho < \min\{\p_1^{-}, \p_2^{-}, r\}$ if $r>1$, and $\rho=1$ if $r=1$. The multiplier $\Delta_j^{\theta}$ is defined by $\widehat{\Delta_j^{\theta} f}=\widehat{\theta}(2^{-j}\cdot) \widehat{f}$, for each $j \in \Z$, where $\theta \in \mathcal{S}(\R)$ satisfies $\supp(\widehat{\theta}) \subset \{ \xi \in \R: 1/c_0 \leq |\xi| \leq c_0 \}$, for some $c_0>1$, and $\sum_{j \in \Z} \widehat{\theta}(2^{-j}\xi)=C_{\theta}$ for all $\xi \in \R \setminus\{0\}$. Considering the same property of $\Delta_j^{\theta}$ and $\Delta_j^{\zeta}$, we will suppress $\theta$ and $\zeta$ in this sequel.

Let $w_i^{p_i} \in A_{p_i/\p_i^-}$, $i=1, 2$. By the choice of $\rho$, we have 
\begin{align}\label{eq:wp-Apr}
w_i^{p_i} \in A_{p_i/\rho} \subset A_{p_i} \quad\text{ with }\quad 
[w_i^{p_i}]_{A_{p_i}} \le [w_i^{p_i}]_{A_{p_i/\rho}} 
\le [w_i^{p_i}]_{A_{p_i/\p_i^-}}, \quad i=1, 2. 
\end{align} 
Let us control $T_{\sigma_1}$ and $T_{\sigma_2}$. Invoking \eqref{eq:Ta-2}--\eqref{eq:wp-Apr}, \eqref{eq:sharp}, and Lemma \ref{lem:M-vector}, we use  H\"{o}lder's inequality to conclude that 
\begin{align*}
\|T_{\sigma_2} (f_1, f_2)\|_{L^p(w^p)} 
&\lesssim \bigg\| \sum_{j \in \Z} M\big(|\Delta_j f_1|^{\rho}\big)^{\frac{1}{\rho}} 
M\big(|\Delta_j f_2|^{\rho}\big)^{\frac{1}{\rho}} \bigg\|_{L^p(w^p)}
\\
&\leq \prod_{i=1}^2 \bigg\|\Big(\sum_{j \in \Z} M\big(|\Delta_j f_i|^{\rho}\big)^{\frac{2}{\rho}} \Big)^{\frac12}
\bigg\|_{L^{p_i}(w_i^{p_i})} 
\\
&\lesssim \prod_{i=1}^2 [w_i^{p_i}]_{A_{p_i/\rho}}^{\max\{\frac12, \frac{1}{p_i-\rho}\}}
\bigg\|\Big(\sum_{j \in \Z} |\Delta_j f_i|^2 \Big)^{\frac12} \bigg\|_{L^{p_i}(w_i^{p_i})}  
\\
&\lesssim \prod_{i=1}^2 [w_i^{p_i}]_{A_{p_i/\rho}}^{\max\{\frac12, \frac{1}{p_i-\rho}\}}
[w_i^{p_i}]_{A_{p_i}}^{\max\{\frac12, \frac{1}{p_i-1}\}} \|f_i\|_{L^{p_i}(w_i^{p_i})}
\\
&\lesssim \prod_{i=1}^2 [w_i^{p_i}]_{A_{p_i/\p_i^-}}^{2 \max\{\frac12, \frac{1}{p_i-\p_i^-}\}}
\|f_i\|_{L^{p_i}(w_i^{p_i})}, 
\end{align*}
where the inequality \eqref{eq:phik-1} was used in the second-to-last step. To estimate $T_{\sigma_1}$, we note that by Lemma \ref{lem:ApRH-m}, 
\begin{align}\label{eq:Tswp}
[w^p]_{A_2} 
\le [w^p]_{A_{p/\p_-}}
\le \prod_{i=1}^2 [w_i^{p_i}]_{A_{p_i/\p_i^-}}^{\frac{p}{p_i}}, 
\end{align}
since $\frac{1}{\p_-} := \frac{1}{\p_1^-} + \frac{1}{\p_2^-} \le 2$ and $p \le 1 \le 2\p_-$. Therefore, in view of Lemma \ref{lem:LPAinfty} applied to $r=2$ and $v=w^p$, \eqref{eq:Tswp}, and \eqref{eq:Ta-1}, we proceed as above to obtain 
\begin{align*}
& \|T_{\sigma_1}(f_1, f_2)\|_{L^p(w^p)} 
\lesssim [w^p]_{A_2}^{\frac{15}{p}} \bigg\|\bigg(\sum_{j \in \Z}|\Delta_j 
(T_{\sigma_1}(f_1, f_2))|^2 \bigg)^{\frac12}\bigg\|_{L^p(w^p)} 
\\
&\quad \lesssim \prod_{i=1}^2 [w_i^{p_i}]_{A_{p_i/\p_i^-}}^{\frac{15}{p_i}}
\bigg\|M(|f_1|^{\rho})^{\frac{1}{\rho}} 
\bigg(\sum_{j \in \Z}M\big(|\Delta_j f_2|^{\rho} \big)^{\frac{2}{\rho}}\bigg)^{\frac12} \bigg\|_{L^p(w^p)}
\\ 
&\quad \lesssim \prod_{i=1}^2 [w_i^{p_i}]_{A_{p_i/\p_i^-}}^{\frac{15}{p_i}}
\big\|M(|f_1|^{\rho})^{\frac{1}{\rho}} \big\|_{L^{p_1}(w_1^{p_1})}
\bigg\|\bigg(\sum_{j \in \Z}M\big(|\Delta_j f_2|^{\rho} 
\big)^{\frac{2}{\rho}}\bigg)^{\frac12} \bigg\|_{L^{p_2}(w_2^{p_2})} 
\\
&\quad \lesssim [w_1^{p_1}]_{A_{p_1/\rho}}^{\frac{15}{p_1}+\frac{1}{p_1-\rho}} 
[w_2^{p_2}]_{A_{p_2/\rho}}^{\frac{15}{p_2}+\max\{\frac12, \frac{1}{p_2-\rho}\}} 
\|f_1\|_{L^{p_1}(w_1^{p_1})} 
\bigg\|\bigg(\sum_{j \in \Z} |\Delta_j f_2|^2 \bigg)^{\frac12} \bigg\|_{L^{p_2}(w_2^{p_2})} 
\nonumber \\ 
&\quad \lesssim \prod_{i=1}^{m} [w_i^{p_i}]_{A_{p_i/\p_i^-}}^{\frac{15}{p_i} + 2 \max\{\frac12, \frac{1}{p_i-\p_i^-}\}}
\|f_i\|_{L^{p_i}(w_i^{p_i})}. 
\end{align*}
This completes the proof. 
\end{proof}

In this subsection, we always choose $\phi \in \S(\Rn)$ with $\int_{\Rn} \phi \, dx=1$, and set $\phi_t(x) := t^{-n} \phi(x/t)$ for any $x \in \Rn$ and $t>0$. And let $\psi,\,\Phi\in \S(\Rn)$ satisfy $0 \le \widehat{\psi}(\xi)\le \mathbf{1}_{\{1/2 \le |\xi| \le 2\}}$, $\widehat{\psi}(\xi) \ge 0$ for $1/2 \le |\xi| \le 2$, $\sum_{j \in \Z} \widehat{\psi}(2^j\xi)=1$ for $|\xi|\ne0$, and $\mathbf{1}_{\{1/2 \le |\xi| \le 2\}} \le \widehat{\Phi}(\xi) \le \mathbf{1}_{\{1/3 \le |\xi| \le 3\}}$. Denote $\psi_j(x)=2^{-jn}\psi(x/2^j)$ and $\Phi_j(x)=2^{-jn}\Phi(x/2^j)$ for each $j \in \Z$.

\begin{lemma}\label{lem:LPAinfty}
For all $0<p \le 1 \le r \le 2$ and for all $v \in A_r$, 
\begin{align*}
\|f\|_{L^p(v)} 
\le \bigg\|\sup_{t>0} |\phi_t*f| \bigg\|_{L^p(v)} 
\lesssim [v]_{A_2}^{\frac{11+2r'}{p}}
\bigg\|\bigg(\sum_{j \in \Z} |\Delta_j f|^2 \bigg)^{\frac12} \bigg\|_{L^p(v)}. 
\end{align*}
\end{lemma}

\begin{proof}
It suffices to show the second inequality since $|f(x)| \leq \sup_{t>0} |\phi_t*f(x)|$ for all $x \in \Rn$. By Lemma \ref{lem:NNN}--\ref{lem:LPA} below and estimates in \cite[p. 588]{Bui}, we  have 
\begin{align*}
\|f\|_{H^p(v)} 
&= \Big\|\sum_j \Phi_j \ast \psi_j \ast f\Big\|_{H^p(v)} 
\lesssim [v]_{A_2}^{\frac{9}{p}+\frac{2r'}{p}} \Big\|
\sup_{t>0} \Big(\sum_j |\phi_t*\psi_j \ast f|^2\Big)^{\frac12} \Big\|_{L^p(v)} 
\\
&\le [v]_{A_2}^{\frac{9}{p}+\frac{2r'}{p}} \Big\| 
\Big(\sum_j \sup_{t>0} |\phi_t*\psi_j \ast f|^2\Big)^{\frac12} \Big\|_{L^p(v)}
\lesssim [v]_{A_2}^{\frac{9}{p}+\frac{2r'}{p}} 
\Big\| \Big(\sum_j |\psi_{j \lambda}^{*}f|^2 \Big)^{\frac12}\Big\|_{L^{p}(v)}
\\
&\lesssim [v]_{A_2}^{\frac{9}{p}+\frac{2r'}{p}}
\bigg\|\Big[\sum_j M(|\psi_j \ast f|^s)(x)^{\frac{2}{s}}\Big]^{\frac{s}{2}} 
\bigg\|_{L^{p/s}(v)}^{\frac1s}
\\
&\lesssim [v]_{A_2}^{\frac{9}{p}+\frac{2r'}{p}}  
[v]_{A_{p/s}}^{\max\{\frac{1}{2}, \frac{1}{p-s}\}} 
\bigg\|\Big[\sum_j |\psi_j \ast f(x)|^2 \Big]^{\frac{1}{2}} \bigg\|_{L^p(v)}  
\end{align*}
where we used Lemma \ref{lem:M-vector} and that 
$\lambda>\max\{\frac{nr}{p}, \frac{n}{2}\} = \frac{nr}{p}$, so 
$s := \frac{n}{\lambda}<\frac{p}{r}$ and $[v]_{A_{p/s}} \leq [w]_r$. 
If we take $\frac{n}{\lambda}=\frac{p(1-\varepsilon)}{r}$ for some 
$\varepsilon \in (0, 1)$, then $p-s=p-\frac{n}{\lambda} 
= p(1-\frac{1-\varepsilon}{r}) \ge p(1-(1 - \varepsilon) ) = p \varepsilon$. 
This means $\max\{\frac{1}{2}, \frac{1}{p-s}\} < \frac{1}{p\varepsilon}$. 
Consequently, taking $\varepsilon=1/2$, we get the desired estimate.  
\end{proof}

We use the maximal operators $N, N^+, N^*$ defined in \cite{Bui}. Moreover, 
given a sequence $\mathbf{f}=\{f_j\}$, a function $u$ on $\R^{n+1}_+$, and 
$\alpha, \kappa>0$, we define
\begin{align*}
N_{\kappa}^{**} \mathbf{f}(x) 
& :=  \sup_{y \in \Rn, t>0} \bigg(\sum_j |\phi_t*f_j(x)|^q \bigg)^{\frac1q} 
\bigg(\frac{t}{t+|x-y|}\bigg)^{\kappa}, 
\\
\widetilde{N}_{\alpha} u(x) 
&:=  \sup_{|x-y|<\alpha t} |u(y, t)|, \quad 
\widetilde{N}_{\kappa}^{**} u(x) :=  
\sup_{y \in \Rn, t>0} |u(y, t)| \bigg(\frac{t}{t+|x-y|}\bigg)^{\kappa}. 
\end{align*}

\begin{lemma}\label{lem:NNN}
For any $p \in (0, \infty)$, $r \in (1, \infty)$, and $w \in A_r$, 
\begin{align}
\label{NN-1} \|\widetilde{N}^{**}_{\kappa} u\|_{L^p(w)} 
& \lesssim [w]_{A_r}^{\frac{r'}{p}} \|\widetilde{N}_1 u\|_{L^p(w)}, 
\\ 
\label{NN-2} \|N^* \mathbf{f}\|_{L^p(w)} 
& \lesssim [w]_{A_r}^{\frac{r'}{p}} \|N \mathbf{f}\|_{L^p(w)}, 
\\ 
\label{NN-3} \|N \mathbf{f}\|_{L^p(w)} 
& \lesssim [w]_{A_r}^{\frac{r'}{p}} \|N^+ \mathbf{f}\|_{L^p(w)}. 
\end{align}
\end{lemma}

\begin{proof}
The inequality \eqref{NN-1} follows from the following
\begin{align*}
\widetilde{N}^{**}_{\kappa} u \lesssim \sup_{m \in \N} 2^{-mn} \widetilde{N}_{2^m} u 
\quad\text{ and }\quad 
w(\{\widetilde{N}_{\beta}u>\eta\}) 
\lesssim (1+\beta/\alpha)^{nr} [w]_{A_r}^{r'} w(\{\widetilde{N}_{\alpha}u>\eta\}), 
\end{align*}
for all $\eta>0$, where the first estimate is trivial and the second one is contained in \cite{GW}.  The inequality \eqref{NN-2} is a consequence of  \eqref{NN-1} and the pointwise estimate $N^*\mathbf{f} \lesssim \widetilde{N}^{**}_{\kappa} \mathbf{f}$. 

To show \eqref{NN-3}, we trace the proof of $\|N \mathbf{f}\|_{L^p(w)} \lesssim \|N^+ \mathbf{f}\|_{L^p(w)}$ in \cite{Bui}. 
Firstly, by \eqref{NN-1} we have 
\begin{equation*}
\|N^{**}_\lambda f\|_{L^p(w)}
\lesssim [w]_{A_r}^{\frac{r'}{p}} \|N \mathbf{f}\|_{L^p(w)}.
\end{equation*}
Setting $\widetilde{N}_\mu u_i(x) := \sup_{t>0, |x-y|<\mu t} \big(\sum_{j\in\Z}|\phi_t^{(i)}*f_j(y)|^q \big)^{\frac1q}$, where $\phi^{(i)}=\frac{\partial \phi}{\partial x_i}$ and $\mu>1$, we use \eqref{NN-2} to get 
\begin{equation}\label{eq:Nnn-2}
\|\widetilde{N}_\mu u_i\|_{L^p(w)}
\lesssim [w]_{A_r}^{\frac{r'}{p}} \|N \mathbf{f}\|_{L^p(w)}.
\end{equation}
Since $r>1$ and $w\in A_r$, Lemma \ref{lem:open} gives that $r>\inf\{\rho>0: w\in A_\rho\}$. So, for $s\in(0,1]$ with $p/s = r > \inf\{\rho>0: w\in A_\rho\}$, and $\delta>0$ satisfying $\Gamma_\delta(y)\subset \Gamma_\mu(x)$ for all $(y,t) \in \Gamma_1(x)$, we get 
\begin{equation*}
N{\mathbf f}(x)^s\le (1+1/\delta)^n M((N^+ {\bold f})^s)(x) +
\delta^s\sum_{i=1}^{n} \widetilde{N}_\mu u_i(x)^s.
\end{equation*}
Hence, taking $L^{p/s}(w)$-norm of both sides of the above, and using
\eqref{eq:Nnn-2}, we see that 
\begin{equation*}
\|N \mathbf{f}\|^s_{L^p(w)} 
\le C_1 (1+1/\delta)^n[w]_{A_{p/s}}^{\frac{1}{p/s-1}} \|N^+ \mathbf{f}\|^s_{L^p(w)}
+ C_2 \delta^s\|N \mathbf{f}\|_{L^p(w)}.
\end{equation*}
Choosing $\delta$ so small that $C_2 \delta^s<1/2$, we obtain 
\begin{equation*}
\|N \mathbf{f}\|_{L^p(w)} 
\lesssim [w]_{A_{p/s}}^{\frac{1}{p-s}} \|N^+ \mathbf{f}\|_{L^p(w)}
=[w]_{A_{r}}^{\frac{r'}{p}} \|N^+ \mathbf{f}\|_{L^p(w)}.
\end{equation*}
This completes the proof of \eqref{NN-3}. 
\end{proof}

\begin{lemma}\label{lem:LPA}
Then for any $p \in (0, 1]$ and $w \in A_2$,  
\begin{align*}
\bigg\|\sup_{0<t<\infty} \Big|\phi_t \ast \Big(\sum_j \Phi_j \ast f_j \Big) \Big|\bigg\|_{L^p(w)}
\lesssim [w]_{A_2}^{\frac9p} \|N^*\mathbf{f}\|_{L^p(w)}. 
\end{align*}
\end{lemma}

\begin{proof}
Fix $w \in A_2$ and $\lambda>0$. It suffices to show 
\begin{align}\label{Jlam}
\mathcal{J}_{\lambda} 
&:= w \Big(\Big\{x \in \Rn: \sup\limits_{t>0} \Big|\phi_t * \Big(\sum_j \Phi_j \ast f_j \Big)(x) \Big| > \lambda \Big\} \Big)
\\ \nonumber
&\lesssim [w]_{A_2}^9 \bigg\{\lambda^{-2} \int _{\Rn \backslash \Omega_{\lambda}}\sum_j |f_{j}(x)|^2 w(x) \, dx 
+ w(\Omega_{\lambda}) \bigg\}, 
\end{align}
where the implicit constant is independent of $\lambda$, and 
$\Omega_{\lambda} := \{N^* \mathbf{f}>\lambda\}$ 
(cf. \cite[p. 190]{RRT}). 

It follows from Whitney decomposition that one can find a pairwise disjoint family of cubes $\{Q_j\}$ such that 
$\Omega_{\lambda} = \bigcup_k Q_k$ and $\dist(\Rn \backslash \Omega_{\lambda}, Q_k) \simeq \ell(Q_k)$. Then we choose a sequence of nonnegative functions $\{\varphi_k\}_k$ such that $\mathbf{1}_{\Omega_{\lambda}} = \sum_k \varphi_k$, with the following properties 
\[
\supp(\varphi_k) \subset \frac{6}{5} Q_k, \quad 
a_k := \int_{\Rn} \varphi_k\, dx \simeq |Q_k|, \quad
\|\partial^{\alpha}\varphi_k \|_{L^{\infty}(\Rn)} \lesssim \ell(Q_k)^{-|\alpha|}. 
\]  
Setting
\[
\widetilde{f}_j(x) := f_j(x) \mathbf{1}_{\Rn \backslash \Omega_{\lambda}} + \sum_k b_{k, j} \, \varphi_k 
\quad\text{ and }\quad 
b_{k, j} := \frac{1}{a_k} \int_{\Rn} f_j(x) \varphi_k(x) \, dx, 
\]
we see that for all $x \in \Rn$, 
\begin{align*}
\sum_j |\widetilde{f}_j(x)|^2 
\lesssim \sum_j |f_j(x)|^2 \mathbf{1}_{\Rn \setminus \Omega_{\lambda}} 
+ \sum_j |b_{k, j}(x)|^2 
\lesssim \lambda^2 + N^* \mathbf{f}(x_j) 
\lesssim \lambda^2,  
\end{align*}
where $x_j \in C_0 Q_j \cap (\Rn \setminus \Omega_{\lambda}) \neq \emptyset$ for all $j$ and for some $C_0>0$, which follows from the construction of Whitney decomposition of $\Omega$. 

Writing 
\begin{align*}
\mathcal{J}'_{\lambda} 
& := w \Big(\Big\{x \in \Rn \setminus \Omega_{\lambda}: \sup\limits_{t>0} \Big|\phi_t * 
\Big(\sum_j \Phi_j \ast (f_j - \widetilde{f}_j) \Big)(x) \Big| > \lambda \Big\} \Big), 
\end{align*} 
and observing that 
\[
\sup\limits_{t>0} \Big|\phi_t * \Big(\sum_j \Phi_j \ast (f_j - \widetilde{f}_j) \Big)(x) \Big|
\lesssim \lambda \, \mathfrak{M}_1(x), \quad x \in \Rn \backslash \Omega_{\lambda}, 
\] 
where $\mathfrak{M}_1(x)$ is defined in \eqref{Me}, we invoke Lemma \ref{lem:Mar} to deduce 
\begin{align}\label{Jlam-1}
\mathcal{J}'_{\lambda} 
\lesssim \|\mathfrak{M}_1\|_{L^2(w)}^2 
\lesssim [w]_{A_2}^2 w(\Omega_{\lambda}). 
\end{align}
By Chebyshev's inequality and \eqref{eq:sharp}, 
\begin{align}
\mathcal{J}''_{\lambda} \label{eq:Nnn-4}
& := w \Big(\Big\{x \in \Rn: \sup\limits_{t>0} \Big|\phi_t * 
\Big(\sum_j \Phi_j \ast \widetilde{f}_j \Big)(x) \Big| > \lambda \Big\} \Big)
\\
&\leq \lambda^{-2} \bigg\|\sup\limits_{t>0} \Big|\phi_t * \Big(\sum_j \Phi_j \ast \widetilde{f}_j \Big) \Big| \bigg\|_{L^2(w)}^2 
\lesssim \lambda^{-2} \bigg\|M\Big(\sum_j \Phi_j \ast \widetilde{f}_j \Big) \bigg\|_{L^2(w)}^2 \nonumber
\\
&\lesssim \lambda^{-2} [w]_{A_2}^2 \bigg\|\sum_j \Phi_j \ast \widetilde{f}_j \bigg\|_{L^2(w)}^2 
= \lambda^{-2} [w]_{A_2}^2 \int_{\Rn} \Big|\sum_j \Phi_j \ast \widetilde{f}_j(x) \Big|^2 w(x) \, dx.\nonumber
\end{align}
To control the last term, we let $T$ be the singular integral with $\mathscr{L}(\ell^2(\Z),\mathbb C)$-valued kernel $\Phi=\{\Phi_j\}_{j \in \Z}$ defined by $T(\mathbf{g}) := \sum_{j\in\Z} \Phi_j*g_j$ for good $\ell^2$-valued functions $\mathbf{g}=\{g_j\}_{j \in \Z}$. One can check that $T$ is bounded from $L^2(\Rn,\ell^2)$ to $L^2(\Rn, \ell^2)$, $\|\Phi\|_{\mathscr{L}(\ell^2(\Z),\mathbb C)}\lesssim |x|^{-n}$, and $\|\nabla\Phi\|_{\mathscr L(\ell^2(\Z),\mathbb C)}\lesssim |x|^{-n-1}$ (cf. \cite[p. 165]{Tri}). Hence, this, Lemma \ref{lem:CZO}, and \eqref{eq:Nnn-4} yield 
\begin{align}\label{Jlam-2} 
\mathcal{J}''_{\lambda}
&\lesssim \lambda^{-2} [w]_{A_2}^9  \int_{\Rn} \sum_j |\widetilde{f_j}(x)|^2 w(x)dx
\\ \nonumber 
&\leq \lambda^{-2} [w]_{A_2}^9 
\int_{\Rn \backslash \Omega_{\lambda}} \sum_j |f_j(x)| ^2 w(x) dx + [w]_{A_2}^9 w(\Omega_{\lambda}).
\end{align}
As a consequence, \eqref{Jlam} immediately follows from \eqref{Jlam-1} and \eqref{Jlam-2}. 
\end{proof}

\subsection{Weighted jump inequalities for rough operators}

Let $\mathcal{F} := \{F_t(x)\}_{t>0}$ be a family of Lebesgue measurable functions defined on $\Rn$. Given $\lambda>0$, we introduce the $\lambda$-jump function $N_\lambda(\mathcal F)$ of $\mathcal F$, its value at $x$ is the supremum over all $N$ such that there exist $s_1<t_1\leq s_2<t_2\leq\dotsc\leq s_N<t_N$ with
\[
|F_{t_k}(x)-F_{s_k}(x)|>\lambda, \quad\forall k=1, \ldots, N.
\] 
Given $\rho>0$, the value of the strong $\rho$-variation function $\mathcal{V}_\rho(\mathcal F)$ at $x$ is defined by
\begin{equation*}
\mathcal{V}_\rho(\mathcal F)(x) 
:= \sup_{\{t_k\}_{k \ge 0}} 
\bigg( |F_{t_0}(x)|^{\rho} + \sum_{k \geq 1} |F_{t_k}(x)-F_{t_{k-1}}(x)|^{\rho} \bigg)^{\frac{1}{\rho}},
\end{equation*}
where the supremum runs over all increasing sequences $\{t_k\}_{k \geq 0}$.

Given $\Omega \in L^1(\Sn)$ and $\varepsilon>0$, the truncated singular integral operator $T_{\varepsilon}$ is defined by
\begin{equation*}
T_{\Omega, \varepsilon} f(x) := \int_{|y| \ge \varepsilon}\frac{\Omega(y')}{|y|^n}f(x-y)dy. 
\end{equation*}
The principal value singular integral and its maximal version are defined by 
\begin{equation*}
T_{\Omega} f(x) := \lim_{\varepsilon \to 0^+} T_{\Omega, \varepsilon} f(x)
\quad\text{ and }\quad 
T_{\Omega, \#} f(x) := \sup_{\varepsilon>0} |T_{\Omega, \varepsilon} f(x)|, \quad x \in \Rn. 
\end{equation*}
In this sequel, we write $\mathcal{T} := \{T_{\Omega, \varepsilon}\}_{\varepsilon>0}$.

\begin{theorem}\label{thm:jump}
Let $\rho>2$ and $\Omega \in L^q(\Sn)$ with $q \in (1, \infty)$ be such that $\int_{\Sn} \Omega \, d\sigma=0$. Then for all $p \in (q', \infty)$ and for all $w \in A_{p/q'}$, 
\begin{align}\label{JTf}
\|\mathbb{T}f\|_{L^p(w)}
\lesssim [w]_{A_{p/q'}}^{7\max\{1, \frac{2}{p/q'-1}\}} \|f\|_{L^p(w)}, 
\end{align}
where $\mathbb{T} \in \big\{\sup\limits_{\lambda>0} \lambda\sqrt{N_{\lambda} \circ \mathcal{T}}, \mathcal{V}_{\rho} \circ \mathcal{T}, T_{\Omega, \#}\big\}$.
\end{theorem}

It suffices to show \eqref{JTf} for $\mathbb{T}=\sup\limits_{\lambda>0} \lambda\sqrt{N_{\lambda} \circ \mathcal{T}}$, which immediately implies \eqref{JTf} for $\mathbb{T} \in \{\mathcal{V}_{\rho} \circ \mathcal{T}, T_{\Omega, \#}\}$ since the following pointwise domination holds
\begin{equation*}
T_{\Omega, \#} f(x) 
\le \mathcal{V}_{\rho}(\mathcal{T} f)(x) 
\le \sup_{\lambda>0} \lambda \sqrt{N_\lambda(\mathcal{T} f)(x)}, \qquad x \in \Rn, 
\end{equation*} 
provided that $\ell^{2,\infty}(\mathbb{N})$ embeds into $\ell^\rho(\mathbb{N})$ for all $\rho>2$. 

Let us turn to the proof of \eqref{JTf} for $\mathbb{T}=\sup\limits_{\lambda>0} \lambda\sqrt{N_{\lambda} \circ \mathcal{T}}$. It was proved in \cite[Lemma 1.3]{JSW08} that 
\begin{equation*}
\lambda \sqrt{N_\lambda(\mathcal Tf)(x)} 
\lesssim S_2(\mathcal Tf)(x)+\lambda\sqrt{N_{\lambda/3}(\{T_{\Omega, 2^k}f\})(x)}, \quad x \in \Rn, 
\end{equation*}
where
\begin{align*}
S_2(\mathcal Tf)(x) & := \bigg(\sum_{j\in\mathbb Z} V_{2, j}(\mathcal Tf)(x)^2\bigg)^{\frac12},
\\
V_{2,j}(\mathcal Tf)(x) & := \bigg(\sup_{\substack{t_1<\cdots<t_N \\ [t_l,t_{l+1}]\subset[2^j,2^{j+1}]}} 
\sum_{l=1}^{N-1}|T_{\Omega, t_{l+1}}f(x)-T_{\Omega, t_l}f(x)|^2\bigg)^{\frac12}.
\end{align*}
Thus, we are reduced to proving 
\begin{align}
\label{d-jump}
\Big\|\sup_{\lambda>0}\lambda \sqrt{N_\lambda(\{T_{\Omega, 2^k}f\})} \Big\|_{L^p(w)}
& \lesssim [w]_{A_{p/q'}}^{7\max\{1, \frac{2}{p/q'-1}\}} \|f\|_{L^p(w)}, 
\\ 
\label{short-var}
\|S_2(\mathcal Tf)\|_{L^p(w)} 
& \lesssim [w]_{A_{p/q'}}^{4\max\{1, \frac{2}{p/q'-1}\}} \|f\|_{L^p(w)}. 
\end{align}

\subsubsection{\bf Dyadic jump estimates}
We are going to show \eqref{d-jump} in this subsection. Let $\phi \in \S(\Rn)$ be a radial function such that $\widehat{\phi}(\xi)=1$ for $|\xi| \le2$ and $\widehat{\phi}(\xi)=0$ for $|\xi|>4$. Define $\widehat{\phi}_k(\xi) = \widehat{\phi}(2^k\xi)$ for each $k \in \Z$. For each $j \in \Z$,  set $\nu_j(x) := \frac{\Omega(x)}{|x|^n} \mathbf{1}_{\{2^j\le |x|<2^{j+1}\}}(x)$. Then for any $k \in \Z$, 
\begin{align*}
T_{\Omega, 2^k} f(x)  
& = \int_{|x-y| \ge 2^k} \frac{\Omega(x-y)}{|x-y|^n}f(y) \, dy 
=\sum_{j \ge k} \nu_j \ast f(x) 
\\
&=\phi_k \ast T_{\Omega} f + \sum_{s \ge 0}(\delta_0 - \phi_k) \ast \nu_{k+s} \ast f - \phi_k \ast \sum_{s<0} \nu_{k+s}\ast f 
\\
&=: T^1_k f + T^2_k f - T_k^3 f,
\end{align*}
where $\delta_0$ is the Dirac measure at 0. Let $\mathscr T^if$ denote the family $\{T^i_{k}f\}_{k \in \Z}$, 
$i=1, 2, 3$. Hence, to get \eqref{d-jump} it suffices to prove the following:
\begin{equation}\label{TTT}
\Big\|\sup_{\lambda>0}\lambda \sqrt{N_{\lambda}(\mathscr T^if)} \Big\|_{L^p(w)} 
\lesssim [w]_{A_{p/q'}}^{7\max\{1, \frac{2}{p/q'-1}\}} \|f\|_{L^p(w)}, \quad i=1,2,3. 
\end{equation}

We begin with showing \eqref{TTT} for $i=1$. Define 
\[
\mathbb{D}_j f := \mathbb{E}_j f - \mathbb{E}_{j-1} f, \quad 
\mathscr{E}f :=\{\mathbb{E}_j f\}_{j \in \Z}, \quad\text{ where }\quad  
\mathbb{E}_j f := \sum_{Q \in \D_j} \bigg(\fint_Q f \, dx \bigg) \mathbf{1}_Q, 
\]
where $\D_j$ is the family of dyadic cubes with sidelength $2^j$.

\begin{lemma}\label{lem:}
For any $p \in (1, \infty)$ and $w \in A_p$, 
\begin{align}\label{eq:DN-1}
\|\mathbb{T}f\|_{L^p(w)}
\lesssim [w]_{A_p}^{\max\{1, \frac{1}{p-1}\}} \|f\|_{L^p(w)}, 
\end{align}
where $\mathbb{T}f \in\big\{\big(\sum_{j \in \Z} |\mathbb{D}_j f|^2 \big)^{\frac12},\, \sup_{\lambda>0} \lambda\sqrt{N_{\lambda}(\mathscr{E}f)}\big\}$.
\end{lemma}

\begin{proof}
For $p=2$, the estimate \eqref{eq:DN-1} for dyadic operators is contained in \cite{LPR}, which established a sharp weighted inequality for the Haar shift operators. The general case is a consequence of the case $p=2$ and Theorem \ref{thm:Ap}. Then \eqref{eq:DN-1} for jump operators follows at once from \eqref{eq:DN-1} for $\mathbb{T}f =\big(\sum_{j \in \Z} |\mathbb{D}_j f|^2 \big)^{\frac12}$ and the proof of \cite[Proposition 4.1]{KZ}.
\end{proof}

Define the square function as follows: 
\begin{equation}\label{def:Sf}
\mathfrak{S} f := \bigg(\sum_{k \in \Z} |\phi_k*f - \mathbb{E}_k f|^2 \bigg)^{\frac12}. 
\end{equation}
\begin{lemma}\label{lem:SS}
For any $w \in A_1$, 
\begin{align}\label{SS-weak} 
\|\mathfrak{S}\|_{L^2(w) \to L^2(w)} \lesssim [w]_{A_1}^2 
\quad\text{ and }\quad
\|\mathfrak{S}\|_{L^1(w) \to L^{1, \infty}(w)} \lesssim [w]_{A_1}^5, 
\end{align}
where the implicit constant is independent of $[w]_{A_1}$.
\end{lemma}

\begin{proof}
We claim that for all $k, j \in \Z$,  
\begin{align}\label{eq:kjDj}
\|\mathcal{I}_{k, j}\|_{L^2(w)}
:= \|\phi_{k+j}*\mathbb{D}_j f - \mathbb{E}_{k+j} \mathbb{D}_j f\|_{L^2(w)}
\lesssim 2^{-\theta |k|} [w]_{A_1} \|\mathbb{D}_j f\|_{L^2(w)}, 
\end{align}
for some $\theta>0$, where the implicit constant and $\theta$ are independent of $k$ and $j$. To show \eqref{eq:kjDj}, we first note that by \cite[p. 6722]{JSW08}, for any $k \ge 0$, 
\[
\mathbb{E}_{k+j} \mathbb{D}_j f=0 
\quad\text{ and }\quad 
|\phi_{k+j} *\mathbb{D}_j f| 
\lesssim 2^{-k} M(\mathbb{D}_j f), 
\]
which along with \eqref{eq:sharp} gives 
\begin{align*}
\|\mathcal{I}_{k, j}\|_{L^2(w)}
\lesssim 2^{-k} \|M(\mathbb{D}_j f)\|_{L^2(w)}
\lesssim 2^{-k} [w]_{A_2} \|\mathbb{D}_j f\|_{L^2(w)}
\le 2^{-k} [w]_{A_1} \|\mathbb{D}_j f\|_{L^2(w)}. 
\end{align*}
To control the case $k<0$, we use the argument in \cite[p. 2461--2463]{CDHL} and that 
\[
w(\lambda Q) \le \lambda^n [w]_{A_1} w(Q), \quad\text{ for any cube } Q, 
\]
to see $\mathcal{I}_{k, j}(x) = \sum_{d \ge 0} I_d(x)$, where for some $\delta>0$, 
\begin{align*}
\|I_d\|_{L^2(w)} &\lesssim 2^{-\delta |k|n/4} [w]_{A_1}^{\frac12} \|\mathbb{D}_j f\|_{L^2(w)}, \quad d \le |k|/2, 
\\
\|I_d\|_{L^2(w)} &\lesssim 2^{-d} [w]_{A_1}^{\frac12} \|\mathbb{D}_j f\|_{L^2(w)}, \quad d \ge |k|/2, 
\end{align*}
Then summing these estimates up, we obtain \eqref{eq:kjDj} as desired. 

Having shown \eqref{eq:kjDj}, we use $f(x) = \sum_{j \in \Z} \mathbb{D}_j f(x)$, a.e. $x \in \Rn$, to deduce that 
\begin{align*}
\|\mathfrak{S}f\|_{L^2(w)} 
&= \bigg\|\bigg(\sum_{k \in \Z} \Big| \sum_{j \in \Z} \big(\phi_k* \mathbb{D}_j f - \mathbb{E}_k \mathbb{D}_j f \big) \Big|^2 \bigg)^{\frac12}\bigg\|_{L^2(w)}
\\
&\le \bigg(\sum_{k \in \Z} \Big(\sum_{j \in \Z} \|\phi_k* \mathbb{D}_j f - \mathbb{E}_k \mathbb{D}_j f\|_{L^2(w)} \Big)^2 \bigg)^{\frac12}
\\
&\lesssim [w]_{A_1} \bigg(\sum_{k \in \Z} \Big(\sum_{j \in \Z} 2^{-\theta |k-j|} \|\mathbb{D}_j f\|_{L^2(w)} \Big)^2 \bigg)^{\frac12}
\\
&\lesssim [w]_{A_1} \bigg[\sum_{k \in \Z} \Big(\sum_{j \in \Z} 2^{- \theta |k-j|} \Big) 
\Big(\sum_{j \in \Z} 2^{- \theta |k-j|} \|\mathbb{D}_j f\|_{L^2(w)}^2 \Big)  \bigg]^{\frac12}
\\
&\lesssim [w]_{A_1} \bigg[\sum_{j \in \Z} \Big(\sum_{k \in \Z} 2^{- \theta |k-j|} \Big) 
\|\mathbb{D}_j f\|_{L^2(w)}^2 \bigg]^{\frac12}
\\
&\simeq [w]_{A_1} \bigg\| \Big(\sum_{j \in \Z} |\mathbb{D}_j f|^2 \Big)^{\frac12} \bigg\|_{L^2(w)} 
\lesssim [w]_{A_1}^2 \|f\|_{L^2(w)}, 
\end{align*}
where we have used Minkowski's inequality, \eqref{eq:kjDj}, Cauchy-Schwarz inequality, and \eqref{eq:DN-1}. This shows the first estimate in \eqref{SS-weak}. Then, using the first inequality in \eqref{SS-weak} and Calder\'{o}n-Zygmund decomposition as in \cite[p. 2458--2460]{CDHL}, we obtain the second estimate in \eqref{SS-weak}. The proof is complete. 
\end{proof}

\begin{lemma}\label{lem:Uf}
Let $\mathscr{U}$ be a family of operators given by $\mathscr{U}f :=\{\phi_k * f\}_{k \in \Z}$. Then for all $p \in (1, \infty)$ and for all $w \in A_p$, 
\begin{align*}
\Big\|\sup_{\lambda>0} \lambda\sqrt{N_{\lambda}(\mathscr{U}f)} \Big\|_{L^p(w)}
&\lesssim [w]_{A_p}^{\max\{5, \frac{1}{p-1}\}} \black \|f\|_{L^p(w)}. 
\end{align*}
\end{lemma}

\begin{proof}
Since $N_{\lambda}$ is subadditive, 
\begin{align}\label{NNN-1}
N_{\lambda}(\mathscr{U}f) 
\le N_{\lambda}(\mathscr{D}f) + N_{\lambda}(\mathscr{E}f), 
\end{align}
where $\mathscr{D}f :=\{\phi_k * f - \mathbb{E}_k f\}_{k \in \Z}$ and $\mathscr{E}f :=\{\mathbb{E}_k f\}_{k \in \Z}$. Recall the square function in \eqref{def:Sf} and observe that $\sup_{\lambda>0} \lambda\sqrt{N_{\lambda}(\mathscr{D}f)}
\le \mathfrak{S} f$, 
which together with \eqref{SS-weak} and Theorem \ref{thm:weakAp} applied to $p_0=1$ implies 
\begin{align*}
\Big\|\sup_{\lambda>0} \lambda\sqrt{N_{\lambda}(\mathscr{D}f)} \Big\|_{L^p(w)}
\lesssim [w]_{A_p}^5 \|f\|_{L^p(w)}. 
\end{align*}
In view of \eqref{eq:DN-1} and \eqref{NNN-1}, this gives at once the desired estimate.
\end{proof}

Now using Lemma \ref{lem:Uf} and \eqref{MOTO-2}, we obtain 
\begin{align*}
\Big\|\sup_{\lambda>0} \lambda\sqrt{N_{\lambda}(\mathscr{T}^1 f)} \Big\|_{L^p(w)}
&= \Big\|\sup_{\lambda>0} \lambda\sqrt{N_{\lambda}(\{\phi_k*(T_{\Omega}f)\}}) \Big\|_{L^p(w)}
\\ 
&\lesssim [w]_{A_p}^{\max\{5, \frac{1}{p-1}\}} \|T_{\Omega}f\|_{L^p(w)}
\lesssim [w]_{A_{p/q'}}^{7 \max\{1, \frac{1}{p/q'-1}\}} \|f\|_{L^p(w)}. 
\end{align*}
which shows \eqref{TTT} for $i=1$. 

For the term with $\mathscr{T}^2$, it was shown in \cite[p. 2453]{CDHL} that 
\begin{align}\label{Gss-1}
\sup_{\lambda>0} \lambda \sqrt{N_\lambda(\mathscr T^{2}f)} 
\le \sum_{s \ge 0} \Big(\sum_{k\in \Z} \big|(\delta_0 - \phi_k) \ast \nu_{k+s} \ast f\big|^2\Big)^{\frac12} 
=: \sum_{s \ge 0} G_s f, 
\end{align}
where 
\begin{align*}
G_s f
\le \sum_{l \in \Bbb Z}\bigg(\sum_{k\in \Z} |(\delta_0-\phi_k) \ast \nu_{s+k} \ast \Delta_{l-k}^2 f|^2\bigg)^{\frac12} 
=: \sum_{l\in \Z} G_s^l f, 
\end{align*}
with 
\begin{align}\label{Gs-L2}
\|G_s^l f\|_{L^2(\Rn)} 
\lesssim 2^{ -\gamma_0 s}  \min\{2^l, 2^{-\gamma_0 l}\}\|f\|_{L^2(\Rn)}.
\end{align}
It follows from Lemmas \ref{lem:phif} and \ref{lem:GH} that for any $v \in A_{p/q'}$, 
\begin{multline}\label{Gs-Lp}
\|G_s^l f\|_{L^p(v)} 
\lesssim [v]_{A_p}^{\frac12 \max\{1, \frac{2}{p-1}\}} 
\bigg\|\Big(\sum_{k \in \Z} |\nu_{s+k} \ast \Delta_{l-k}^2 f|^2\Big)^{\frac12}\bigg\|_{L^p(v)} 
\\ 
\lesssim [v]_{A_p}^{\frac12 \max\{1, \frac{2}{p-1}\} + \frac52\max\{1, \frac{2}{p/q'-1}\}}  \|f\|_{L^p(v)}
\lesssim [v]_{A_p}^{3\max\{1, \frac{2}{p/q'-1}\}}  \|f\|_{L^p(v)}.
\end{multline}
Then interpolating between \eqref{Gs-L2} and \eqref{Gs-Lp} with $v \equiv 1$ gives 
\begin{align}\label{Gs-Lpp}
\|G_s^l f\|_{L^p(\Rn)} 
\lesssim 2^{ -\alpha s}  2^{-\beta |l|} \|f\|_{L^p(\Rn)}, \quad\text{for some } \alpha, \beta>0. 
\end{align}
On the other hand, for $w \in A_{p/q'}$, by Lemma \ref{lem:open}, there exists $\gamma=\gamma_w \in (0, 1)$ such that 
\[
(1+\gamma)' = c_n [w]_{A_{p/q'}}^{\max\{1, \frac{1}{p/q'-1}\}} =: c_n B_0, 
\quad\text{and}\quad
[w^{1+\gamma}]_{A_p} \le [w^{1+\gamma}]_{A_{p/q'}} \lesssim [w]_{A_{p/q'}}^{1+\gamma},
\] 
which along with \eqref{Gs-Lp} implies 
\begin{align}\label{GsLpwg}
\|G_s^l f\|_{L^p(w^{1+\gamma})} 
\lesssim [w]_{A_{p/q'}}^{3(1+\gamma)\max\{1, \frac{2}{p/q'-1}\}} \|f\|_{L^p(w^{1+\gamma})}. 
\end{align}
Considering Theorem \ref{thm:SW} with $w_0 \equiv 1$, $w_1=w^{1+\gamma}$, and $\theta=\frac{1}{1+\gamma}$, we interpolate between \eqref{Gs-Lpp} and \eqref{GsLpwg} to arrive at 
\begin{align}\label{Gsaa}
\|G_s^l f\|_{L^p(w)} 
\lesssim 2^{ -\alpha s(1-\theta)}  2^{-\beta |l| (1-\theta)} 
[w]_{A_{p/q'}}^{3\max\{1, \frac{2}{p/q'-1}\}} \|f\|_{L^p(w)}. 
\end{align}
Note that $e^{-t}<2t^{-2}$ for any $t>0$, and 
\begin{align}\label{2as}
\sum_{s \ge 0} 2^{-\alpha s(1-\theta)} 
=\sum_{0 \le s \le B_0} 2^{-\frac{\alpha s}{c_n B_0}} + \sum_{s > B_0} 2^{-\frac{\alpha s}{c_n B_0}} 
\lesssim B_0 + \sum_{s>B_0} s^{-2} B_0^2
\lesssim B_0. 
\end{align}
Similarly, 
\begin{align}\label{2bs}
\sum_{l \in \Z} 2^{-\beta |l| (1-\theta)} \lesssim B_0. 
\end{align}
Hence, \eqref{Gss-1} and \eqref{Gsaa}--\eqref{2bs} imply 
\begin{align*}
\Big\|\sup_{\lambda>0} \lambda\sqrt{N_{\lambda}(\mathscr{T}^2 f)} \Big\|_{L^p(w)}
\lesssim [w]_{A_{p/q'}}^{4 \max\{1, \frac{2}{p/q'-1}\}} \|f\|_{L^p(w)}. 
\end{align*}
This shows \eqref{TTT} for $i=2$. 

To control the term with $\mathscr{T}^3$, we note that by \cite[p. 2456]{CDHL}, 
\begin{align*}
\sup_{\lambda>0} \lambda \sqrt{N_\lambda(\mathscr T^3f)}
\le \sum_{s<0} \Big(\sum_{k\in \Bbb Z}\big| \phi_k\ast\nu_{k+s} \ast f\big|^2 \Big)^{\frac12}
=: \sum_{s<0}H_{s}f, 
\end{align*}
where 
\begin{align*}
\|H_s f\|_{L^p(w)} 
\le \sum_{l \in \Z} \bigg\|\Big(\sum_{k\in \Z} 
|\phi_k \ast \nu_{k+s} \ast \Delta_{l-k}^2 f|^2 \Big)^{\frac12}\bigg\|_{L^p(w)} 
=: \sum_{l \in \Z} \|H_s^l f\|_{L^p(w)},
\end{align*}
with 
\begin{align*}
\|H_s^l f\|_{L^2(\Rn)} 
\lesssim 2^s  \min\{2^l, 2^{-\gamma l}\} \|f\|_{L^2(\Rn)}.
\end{align*}
Analogously to \eqref{Gsaa}, one has 
\begin{align*}
\|H_s^l f\|_{L^p(w)} 
\lesssim 2^{ -\alpha s(1-\theta)}  2^{-\beta |l| (1-\theta)} 
[w]_{A_{p/q'}}^{3\max\{1, \frac{2}{p/q'-1}\}} \|f\|_{L^p(w)}, 
\end{align*}
and eventually, 
\begin{align*}
\Big\|\sup_{\lambda>0} \lambda\sqrt{N_{\lambda}(\mathscr{T}^3 f)} \Big\|_{L^p(w)}
\lesssim [w]_{A_{p/q'}}^{4 \max\{1, \frac{2}{p/q'-1}\}} \|f\|_{L^p(w)}. 
\end{align*}
This shows \eqref{TTT} for $i=3$.

\subsubsection{\bf Short variation estimates}
We will prove \eqref{short-var} in this subsection. As did in \cite{CDHL}, 
\begin{align}
\label{S2k} & S_2(\mathcal{T} f) (x) 
\le \sum_{k \in \Z} S_{2, k}(\mathcal{T} f) (x), 
\\
\label{S2k-1} &\|S_{2, k}(\mathcal{T} f)\|_{L^p(\Rn)} 
\lesssim 2^{-\delta |k|} \|f\|_{L^p(\Rn)}, \quad \forall k \in \Z, 
\\
\label{S2k-2} &S_{2, k}(\mathcal{T} f) (x)
\lesssim \bigg(\sum_{j \in \Z} |M_{\Omega}(\Delta_{k-j}^2 f)(x)| \bigg)^{\frac12}, \quad \forall k \in \Z, 
\end{align}
and for $q<2$, 
\begin{align}\label{S2k-3}
\big\|S_{2,k}(\mathcal{T}f)\big\|_{L^p(w)}
\le \|I_{1, k}f\|_{L^p(w)}^{\frac12} \|I_{2,k}f\|_{L^p(w)}^{\frac12}, 
\end{align}
where 
\begin{align*}
\|I_{1,k} f\|_{L^p(w)}
&\le \bigg(\int_1^2 \bigg\|\Big(\sum_{j \in \Z} |\nu_{j, t} * 
\Delta_{k-j}^2 f|^2 \Big)^{\frac12}\bigg\|_{L^p(w)}^2 \frac{dt}{t} \bigg)^{\frac12}, 
\\
\|I_{2, k} f\|_{L^p(w)}
&\lesssim \bigg(\int_{\Rn} M_{\Omega^{2-q}}(g w)(x) \sum_{j \in \Z} |\Delta_{k-j}^2 f(x)|^2 \, dx \bigg)^{\frac12}, 
\end{align*}
where $\nu_{j, t}(x) := \frac{\Omega(x)}{|x|^n} \mathbf{1}_{\{2^j t \le |x| \le 2^{j+1}\}}(x)$. 

We claim that 
\begin{align}\label{S2k-6}
\big\|S_{2,k}(\mathcal{T}f)\big\|_{L^p(w)}
\lesssim [w]_{A_{p/q'}}^{3 \max\{1, \frac{2}{p/q'-1}\}} \|f\|_{L^p(w)}. 
\end{align}
Once \eqref{S2k-6} is obtained, we use \eqref{S2k}, \eqref{S2k-1}, and Stein-Weiss's interpolation Theorem \ref{thm:SW} as before to get 
\begin{align*}
\big\|S_2(\mathcal{T}f)\big\|_{L^p(w)}
\lesssim [w]_{A_{p/q'}}^{4 \max\{1, \frac{2}{p/q'-1}\}} \|f\|_{L^p(w)}, 
\end{align*}
which shows \eqref{short-var} as desired. 

It remains to demonstrate \eqref{S2k-6}. If $q>2$, we invoke \eqref{S2k-2}, \eqref{MOTO-3}, \eqref{eq:phik-3}, and \eqref{eq:phik-1} to deduce 
\begin{align*}
\big\|S_{2,k}(\mathcal{T}f)\big\|_{L^p(w)}
&\lesssim [w]_{A_{p/q'}}^{\frac{1}{p-q'}} \bigg\| \Big(\sum_{j \in \Z} |\Delta_{k-j}^2 f|^2 \Big)^{\frac12}\bigg\|_{L^p(w)}
\\ \nonumber
&\lesssim [w]_{A_{p/q'}}^{\frac{1}{p-q'}+\frac12\max\{1, \frac{2}{p-1}\}} \bigg\| \Big(\sum_{j \in \Z} |\Delta_{k-j} f|^2 \Big)^{\frac12}\bigg\|_{L^p(w)}
\\ \nonumber
&\lesssim [w]_{A_{p/q'}}^{\frac{1}{p-q'}+\max\{1, \frac{2}{p-1}\}} \|f\|_{L^p(w)} 
\lesssim [w]_{A_{p/q'}}^{\frac32\max\{1, \frac{2}{p-q'}\}} \|f\|_{L^p(w)}. 
\end{align*}
To treat the case $q<2$ (trivially, $p>2$), we observe that much as \eqref{vukk}, 
\begin{align}\label{Ikk-1}
\|I_{1, k} f\|_{L^p(w)}
\lesssim [w]_{A_{p/q'}}^{\frac72\max\{1, \frac{2}{p/q'-1}\}} \|f\|_{L^p(w)}, 
\end{align}
and 
\[
\Omega^{2-q} \in L^{\frac{q}{2-q}}(\Sn), \quad 
\big(w^{1-(p/2)'}\big)^{1-p/2} = w \in A_{p/q'} 
= A_{(p/2)/(\frac{q}{2-q})'}. 
\]
The latter, along with by H\"{o}lder's inequality, Theorem \ref{thm:MTT} applied to $(p/2)'$ and $\frac{q}{2-q}$ instead of $p$ and $q$, \eqref{eq:phik-3}, and \eqref{eq:phik-1}, gives 
\begin{align}\label{Ikk-2} 
\|I_{2, k} f\|_{L^p(w)}
&\lesssim \big\|M_{\Omega^{2-q}}(g w)\|_{L^{(p/2)'}(w^{1-(p/2)'})}^{\frac12} 
\bigg\| \Big(\sum_{j \in \Z} |\Delta_{k-j}^2 f|^2\Big)^{\frac12}\bigg\|_{L^p(w)}
\\ \nonumber
& \lesssim [w]_{A_{p/q'}}^{\max\{1, \frac{1}{p/q'-1}\} + \frac12\max\{1, \frac{2}{p-1}\}} 
\bigg\| \Big(\sum_{j \in \Z} |\Delta_{k-j} f|^2 \Big)^{\frac12}\bigg\|_{L^p(w)}
\\ \nonumber
& \lesssim [w]_{A_{p/q'}}^{\max\{1, \frac{1}{p/q'-1}\} + \max\{1, \frac{2}{p-1}\}} \|f\|_{L^p(w)}
\lesssim [w]_{A_{p/q'}}^{2 \max\{1, \frac{2}{p/q'-1}\}}  \|f\|_{L^p(w)}. 
\end{align}
Therefore, in the case $q<2$, \eqref{S2k-6} follows from \eqref{S2k-3}, \eqref{Ikk-1}, and \eqref{Ikk-2}. 
\qed

\subsection{Riesz transforms associated to Schr\"{o}dinger operators}
Consider a real vector potential $\vec{a}=(a_1,\ldots, a_n)$ and an electric potential $V$. Assume that
\begin{equation}\label{eq:V-ak}
0 \leq V \in L^1_{\loc}(\Rn) \quad\text{and}\quad a_k \in L^2_{\loc}(\Rn), \quad\, k=1, \ldots, n.
\end{equation}
Denote
\begin{align*}
L_0=V^{1/2} \quad\text{and} \quad L_k=\partial_k - i a_k, \quad  k=1, \ldots, n.
\end{align*}
We define the form $Q$ by
\begin{align*}
Q(f, g)=\sum_{k=1}^n \int L_k f(x) \overline{L_k g(x)} \, dx + \int_{\Rn} Vf(x) \overline{g(x)} \, dx
\end{align*}
with domain 
\begin{align*}
\mathcal{D}(Q) := \{f \in L^2(\Rn): L_k f \in L^2(\Rn), \, k=0, 1, \ldots, n \}.
\end{align*}
Let us denote by $A$ the self-adjoint operator associated with $Q$. Then $A$ is given by the expression
\begin{equation*}
A f = \sum_{k=1}^n L_k^{*} L_k f + V f,
\end{equation*}
and the domain of $A$ is given by
\begin{align*}
\mathcal{D}(A) =\Big\{f \in \mathcal{D}(Q), \exists g \in L^2(\Rn) \text { such that }
Q(f, \varphi)=\int_{\Rn} g \overline{\varphi} \, dx, \forall \varphi \in \mathcal{D}(Q) \Big\}.
\end{align*}
Formally, we write
\begin{equation*}
A = - (\nabla-i \vec{a}) \cdot(\nabla - i \vec{a})+V.
\end{equation*}
For convenience, denote
\begin{equation*}
\mathcal{R}_k := L_k A^{-1/2},\quad k = 0, 1, \ldots, n. 
\end{equation*}

Duong et al.  \cite{DOY, DY} consecutively established the $L^p$ boundedness of Riesz transform $\mathcal{R}_k$ and its commutator $[\mathcal{R}_k, b]$, $k=0, 1, \ldots, n$.  More specifically, under the assumption \eqref{eq:V-ak}, we have for any $1<p<2$
\begin{equation}\label{eq:Rk-bRk}
\mathcal{R}_k, [\mathcal{R}_k, b]: L^p(\Rn) \to L^p(\Rn),\quad k=0, 1, \ldots, n,
\end{equation}
provided by $b \in \BMO$.

We would like to establish weighted version of \eqref{eq:Rk-bRk} as follows. 
\begin{theorem}\label{thm:bRk-Lp}
Assume that $\vec{a}$ and $V$ satisfy \eqref{eq:V-ak}. Let $b \in \BMO$. Then for every $p \in (1, 2)$, for every weight $w^p \in A_p \cap RH_{(2/p)'}$, and for every $k=0, 1, \ldots, n$, both $\mathcal{R}_k$ and $[\mathcal{R}_k, b]$ are bounded on $L^p(w^p)$. 
\end{theorem}

A particular case is the operator $\mathscr{L}_V = -\Delta+V$, where $V \in L^1_{\loc}(\Rn)$ is a non-negative function. The $L^2(\Rn)$ boundedness of $\mathcal{R}_V := \nabla \mathscr{L}_V^{-1/2}$ was given in \cite[Theorem 8.1]{Ouh}, while it was proved in \cite{DOY} that $\mathcal{R}_V$ is bounded from $H_L^1(\Rn)$ to $L^1(\Rn)$. Then the interpolation implies
\begin{align}\label{eq:Lp12}
\mathcal{R}_V \text{ is bounded on } L^p(\Rn), \quad\forall p \in (1, 2].
\end{align}
However, \eqref{eq:Lp12} fails for general potentials $V \in L^1_{\loc}(\Rn)$ when $p>2$, see \cite{Shen}. Now Theorem \ref{thm:bRk-Lp} immediately implies the following weighted inequalities. 

\begin{theorem}\label{thm:LV}
Let $\mathscr{L}_V=-\Delta+V$ with $0 \le V \in L^1_{\loc}(\Rn)$, and set $\mathcal{R}_V := \nabla \mathscr{L}_V^{-1/2}$. Then for any $p \in (1, 2)$, for any $w^p \in A_p \cap RH_{(2/p)'}$, and for any $b \in \BMO$, both $\mathcal{R}_V$ and $[\mathcal{R}_V, b]$ are bounded on $L^p(w^p)$.  
\end{theorem}

The rest of this subsection is devoted to showing Theorem \ref{thm:bRk-Lp}. For this purpose, we present two useful lemmas below. 
\begin{lemma}[\cite{AM-1}]\label{lem:GGHH}
Fix $1<q \le \infty$, $a \geq 1$ and $w \in RH_{s'}$, $1 \le s<\infty$. Assume that $F$, $G$, $H_1$ and $H_2$ are non-negative measurable functions on $\Rn$ such that for each ball $B$ there exist non-negative functions $G_B$ and $H_B$ with $F(x) \leq G_B(x) + H_B(x)$ for a.e. $x \in B$ and for all $x, \bar{x} \in B$,
\begin{equation}\label{eq:GH-hypo}
\fint_B G_B \, dy \leq G(x) \quad\text{ and }\quad
\bigg(\fint_B H_B^q \, dy\bigg)^{\frac1q} \leq a \, \big(M F(x) + H_1(x) + H_2(\bar{x}) \big).
\end{equation}
Then for all $p \in (0, q/s)$,
\begin{equation}\label{eq:MFGH}
\|M F\|_{L^p(w)}
\leq C \, \big(\|G\|_{L^{p}(w)}+\|H_1\|_{L^p(w)}+\|H_2\|_{L^p(w)}\big), 
\end{equation}
where the constant $C$ depends only on $n$, $a$, $p$, $q$, and $[w]_{RH_{s'}}$. 
\end{lemma}

To proceed, we introduce some notation. Given a ball $B$ we set $C_j(B) := 4B$ for $j=1$ and $C_j(B) := 2^{j+1}B \setminus 2^jB$ for $j \geq 2$, and
\begin{equation*}
\fint_{C_j(B)} f(x) \, dx := \frac{1}{|2^{j+1} B|} \int_{C_j(B)} f(x) \, dx.
\end{equation*}

\begin{lemma}\label{lem:A-RIA}
Let $1\leq q \leq 2$ and $B$ be a given ball and $f \in L^q(\Rn)$ with $\supp(f) \subseteq B$. Let $\A_{r_B}=I-(I-e^{-r_B^2 A})^m$ with a given integer $m \geq 1$. Then for all $j \geq 1$ and $k=0, 1, \ldots, n$,
\begin{align}
\label{eq:CjB-A} \bigg(\fint_{C_j(B)} |\A_{r_B} f(x)|^{q} dx \bigg)^{\frac1q}
&\lesssim e^{-4^j C_1} \bigg(\fint_{B} |f(x)|^q dx \bigg)^{\frac1q},
\\
\label{eq:CjB-RkIA} \bigg(\fint_{C_j(B)} |\mathcal{R}_k(I-\A_{r_B}) f(x)|^q dx\bigg)^{\frac1q}
& \lesssim 2^{-(n+1)j} \bigg(\fint_{B} |f(x)|^q dx\bigg)^{\frac1q},
\end{align}
where the implicit constants are independent of $B$, $f$, $j$ and $k$.
\end{lemma}

\begin{proof}
We begin with showing \eqref{eq:CjB-A}. It follows from (3.1) and (3.2) in \cite{DY} that
the kernel $p_t(x,y)$ of $e^{-tA}$ satisfies
\begin{align*}
|p_t(x, y)|
& \leq (4 \pi t)^{-\frac{n}{2}} \exp \bigg(-\frac{|x-y|^{2}}{4 t}\bigg), \quad \forall t>0 \text{ and a.e.}\  x,y \in \Rn,
\\
|\partial_t^k p_t(x, y)|
& \leq C_k t^{-(n/2+k)} \exp \bigg(-\frac{|x-y|^{2}}{c_k t}\bigg), \quad \forall t>0 \text{ and a.e. } x, y \in \Rn.
\end{align*}
Thus for all $x \in C_j(B)$ and $j \geq 2$, we have $|x-y| \simeq 2^j r_B$ for any $y \in B$ and
\begin{align}\label{eq:ekrB}
\big|e^{-k r_B^2 A} f(x)\big|
\lesssim \int_{B} r_B^{-n} \exp \bigg(-\frac{|x-y|^{2}}{4r^2_B}\bigg) |f(y)| \, dy
\lesssim e^{-4^j C_1} \fint_B |f| dy.
\end{align}
The above inequality also holds for $j=1$. The desired estimate \eqref{eq:CjB-A} immediately follows
from \eqref{eq:ekrB} and the expansion
\begin{align*}
\A_{r_B}=I - (I-e^{-r^2_B A})^m =\sum_{k=1}^m (-1)^{k+1} C_m^k e^{-k r_B^2 A}.
\end{align*}

Now we turn to the proof of \eqref{eq:CjB-RkIA}. Recalling that
\begin{align*}
A^{-1/2} = \frac{1}{\sqrt{\pi}} \int_{0}^{\infty}e^{-tA} \frac{dt}{\sqrt{t}}, 
\end{align*}
one has
\begin{align*}
\mathcal{R}_k(I-\A_{r_B}) f  =\int_{0}^{\infty} g_{r_B}(t) L_k e^{-tA}f \, dt,
\end{align*}
where $g_{r}(t) = \sum_{\ell=0}^m (-1)^{\ell} C_m^{\ell} \frac{\mathbf{1}_{\{t>\ell r^{2}\}}}{\sqrt{t-\ell r^{2}}}$. Now we claim that
\begin{equation}\label{eq:grt-bound}
\int_{0}^{\infty} |g_r(t)| e^{-\frac{4^j r^{2}}{c t}} \bigg(\frac{r}{t^{1/2}}\bigg)^{n-1}
\frac{dt}{\sqrt{t}} \leq C_m \, 2^{-nj}.
\end{equation}
Moreover, it was proved in \cite[Proposition 3.1]{DY} that for any $j \geq 2$, there exist positive constants $c_1$ and $c_2$ such that
 \begin{equation*}
\bigg(\fint_{C_j(B)} |L_k p_t(x, y)|^2 dx \bigg)^{\frac12}
\leq c_1 \frac{t^{-n/2}}{2^j r_B} \exp \bigg(-\frac{4^j r_B^2}{c_2 t} \bigg), \quad\forall t>0, \, y \in \Rn.
\end{equation*}
which along with \eqref{eq:grt-bound} gives 
\begin{align*}
&\bigg(\fint_{C_j(B)} |\mathcal{R}_k(I-\A_{r_B}) f(x)|^q dx\bigg)^{\frac1q}
\\
&\quad\leq \int_{0}^{\infty} |g_{r_B}(t)| \int_{B} |f(y)| \bigg(\fint_{C_j(B)} |L_k p_t(x, y)|^q dx \bigg)^{\frac1q} dy \, dt
\\
&\quad\lesssim 2^{-j} \int_{0}^{\infty} |g_{r_B}(t)| e^{-\frac{4^j r^{2}}{c t}}
\bigg(\frac{r_B}{t^{1/2}}\bigg)^{n-1} \frac{dt}{\sqrt{t}} \cdot \bigg(\fint_{B} |f(x)|^q dx\bigg)^{\frac1q}
\\
&\quad\leq C_1 2^{-(n+1)j} \bigg(\fint_{B} |f(x)|^q dx\bigg)^{\frac1q}.
\end{align*}

It remains to demonstrate \eqref{eq:grt-bound}. We will use the elementary estimates for $g_r(t)$:
\begin{align}
\label{eq:grt-1}|g_r(t)| &\leq \frac{C_m}{\sqrt{t-\ell r^2}}, \quad \ell r^2 < t \leq (\ell+1)r^2, \ell=0,1,\ldots,m,
\\
\label{eq:grt-2}|g_r(t)| &\leq C_m r^{2m} t^{-m-\frac12}, \quad t>(m+1)r^2.
\end{align}
The first one is easy. The second one is an application of Taylor's formula, see \cite[Sec. 3]{ACDH}.
Denote $\alpha=4^j/c$. Then the inequality \eqref{eq:grt-2} gives that
\begin{multline}\label{eq:large}
\int_{(m+1)r^2}^{\infty} |g_r(t)| e^{-\frac{4^j r^{2}}{c t}}
\bigg(\frac{r}{t^{1/2}}\bigg)^{n-1} \frac{dt}{\sqrt{t}}
\leq C_m \int_{(m+1)r^2}^{\infty} \bigg(\frac{r}{t^{1/2}}\bigg)^{2m+n-1}
e^{-\frac{4^j r^{2}}{c t}} \frac{dt}{t}
\\
=C_m \alpha^{-(m+\frac{n}{2}+\frac12)} \int_{0}^{\frac{\alpha}{m+1}} s^{m+\frac{n}{2}-\frac32} e^{-s} ds
\leq C_m 2^{-j(2m+n-1)} \Gamma \Big(m+\frac{n}{2}+\frac12 \Big).
\end{multline}
Write $\phi(s)=s^{-\frac{n}{2}} e^{-\frac{\alpha}{s}}$, $s>0$. It is easy to get
$\phi'(s)=s^{-\frac{n}{2}-2}e^{-\frac{\alpha}{s}}(\alpha-\frac{n}{2}s)$ and
\begin{equation}\label{eq:phi-bounds}
\phi(s) \leq \phi (2\alpha/n)
=(2\alpha/n)^{-\frac{n}{2}} e^{-\frac{n}{2}} \leq C_n 2^{-nj}, \quad \forall s>0.
\end{equation}
Thus, by \eqref{eq:grt-1}, changing variables and \eqref{eq:phi-bounds}, we have for any $0 \leq \ell \leq m$,
\begin{align}\label{eq:small}
\mathcal{I}_{\ell}
& := \int_{\ell r^2}^{(\ell+1)r^2} |g_r(t)| e^{-\frac{4^j r^{2}}{c t}}
\bigg(\frac{r}{t^{1/2}}\bigg)^{n-1} \frac{dt}{\sqrt{t}}
\\ \nonumber
&\leq C_m \int_{\ell r^2}^{(\ell+1)r^2} \frac{e^{-\frac{4^j r^{2}}{c t}}}{\sqrt{t-\ell r^2}}
\bigg(\frac{r}{t^{1/2}}\bigg)^{n-1} \frac{dt}{\sqrt{t}}
\\ \nonumber
&=C_m \int_{\ell}^{\ell+1} \frac{s^{-\frac{n}{2}} e^{-\frac{\alpha}{s}}}{\sqrt{s-\ell}}  ds
=C_m  \int_{\ell}^{\ell+1} \frac{\phi(s)}{\sqrt{s-\ell}} ds
\\ \nonumber
&=2C_m \phi(\ell+1) - 2 C_m \int_{\ell}^{\ell+1} (s-\ell)^{\frac12} \phi'(s) ds
\\ \nonumber
&\leq C_m 2^{-nj} + C_m 4^j \int_{0}^{\infty} s^{-\frac{n}{2}-2} e^{-\frac{\alpha}{s}} ds
\\ \nonumber
&= C_m 2^{-nj} + C_m 4^j \alpha^{-\frac{n}{2}-1} \int_{0}^{\infty} t^{\frac{n}{2}} e^{-t} dt
\\ \nonumber
&= C_m 2^{-nj} + C_m c^{\frac{n}{2}+1} 2^{-nj}  \Gamma \Big(\frac{n}{2}+1 \Big)
\leq C_m 2^{-nj},
\end{align}
where the constant $C_m$  depending only on $m$ and $n$ varies from line to line. Accordingly, the inequality \eqref{eq:grt-bound} follows from \eqref{eq:large} and \eqref{eq:small}. This completes the proof.
\end{proof}

\begin{proof}[{\bf Proof of Theorem \ref{thm:bRk-Lp}}]
Let $p \in (1, 2)$ and $w^p \in A_p \cap RH_{(2/p)'}$.We follow the ideas in \cite{AM-4}. Choose $p_0$ and $q_0$ such that
$1<p_0<p<q_0<2$ and $w^p \in A_{p/p_0} \cap RH_{(q_0/p)'}$. This together with Lemma \ref{lem:weights} part \eqref{list:ApRH-3} gives that $w^{-p'} \in A_r \cap RH_{s'}$, where $r=p'/q'_0$, $s=p'_0/p'$, and $\tau_p=\big(\frac{q_0}{p}\big)' \big(\frac{p}{p_0}-1\big)+1$. Note that $w^{-p'} \in \cap RH_{s'}$ implies $w^{-p'} \in RH_{s'_0}$ for some $s_0 \in (1, s)$. 

Fix $f \in L_c^{\infty}$ and a ball $B$ with the radius $r_B$. Write
\begin{align*}
F := |\mathcal{R}_k^* f|^{q'_0}\quad\text{and}\quad
\A_{r_B} := I-(I-e^{-r_B^2 A})^m,
\end{align*}
where $m \in \N$ is large enough. Observe that
\begin{align*}
F & \leq 2^{q'_0-1} \big|(I-\A_{r_B}^*) \mathcal{R}^*_k f\big|^{q'_0}
+ 2^{q'_0-1} \big|\A_{r_B}^* \mathcal{R}^*_k f\big|^{q'_0}
=: G_B + H_B.
\end{align*}
We first control $G_B$. By duality, there exists $g \in L^{q_0}(B, dx/|B|)$ with norm $1$ such that for all $x \in B$,
\begin{align}\label{eq:GB}
\bigg(\fint_B  G_B \, dy \bigg)^{\frac{1}{q'_0}}
&\simeq \bigg(\fint_B \big|(I-\A_{r_B}^*) \mathcal{R}^*_k f\big|^{q'_0} dy\bigg)^{\frac{1}{q'_0}}
\lesssim \frac{1}{|B|} \int_{\Rn} |f| |\mathcal{R}_k (I-\A_{r_B}) g| \, dy
\\ \nonumber
&\lesssim \sum_{j=1}^{\infty} 2^{jn}
\bigg(\fint_{C_j(B)} |f|^{q'_0}\bigg)^{\frac{1}{q'_0}}
\bigg(\fint_{C_j(B)} \big|\mathcal{R}_k (I-\A_{r_B}) g\big|^{q_0} \bigg)^{\frac{1}{q_0}}
\\ \nonumber
&\lesssim M (|f|^{q'_0})(x)^{\frac{1}{q'_0}}
\sum_{j=1}^{\infty}  2^{-j} \|g\|_{L^{q_0}(dx/|B|)}
\lesssim M (|f|^{q'_0})(x)^{\frac{1}{q'_0}},
\end{align}
where we have used \eqref{eq:CjB-RkIA}. To estimate $H_B$, we set $q:=p'_0/q'_0$ and observe that by duality there exists
$h \in L^{p_0}(B, dx/|B|)$ with norm $1$ such that for all $x \in B$,
\begin{align}\label{eq:HB}
\bigg(\fint_{B} H_B^q \, dy \bigg)^{\frac{1}{q q'_0}}
&\lesssim \frac{1}{|B|} \bigg|\int_{\Rn} \A_{r_B}^* \mathcal{R}^*_k f \cdot h \, dy \bigg|
\le \frac{1}{|B|} \int_{\Rn} |\mathcal{R}^*_k f| |\A_{r_B}h| \, dy
\\ \nonumber
&\lesssim \sum_{j=1}^{\infty} 2^{jn} \bigg(\fint_{C_j(B)} |\mathcal{R}^*_k f|^{q'_0} \bigg)^{\frac{1}{q'_0}}
\bigg(\fint_{C_j(B)} |\A_{r_B} h|^{q_0} \bigg)^{\frac{1}{q_0}}
\\ \nonumber
&\lesssim M F(x)^{\frac{1}{q'_0}}  \sum_{j=1}^{\infty} 2^{jn} e^{-4^j C_1}
\bigg(\fint_{B} |h|^{q_0} \bigg)^{\frac{1}{q_0}}
\lesssim M F(x)^{\frac{1}{q'_0}},
\end{align}
where \eqref{eq:CjB-A} was used in the last step.

Consequently, \eqref{eq:GB} and \eqref{eq:HB} verify the hypotheses \eqref{eq:GH-hypo} with $G(x) = M (|f|^{q'_0})(x)$ and $H_1=H_2 \equiv 0$.  Observe that $r=p'/q'_0=q/s<q/s_0$. Then,  invoking \eqref{eq:MFGH} applied to $r$, $s_0$, and $w^{-p'}$ in place of $p$, $s$, and $w$, respectively, we obtain
\begin{align*}
\|\mathcal{R}_k^{*} f \|_{L^{p'}(w^{-p'})}^{q'_0}
&=\|F\|_{L^r(w^{-p'})}
\leq \|M F\|_{L^r(w^{-p'})}
\lesssim \|M(|f|^{q'_0})\|_{L^r(w^{-p'})}
\lesssim \|f\|_{L^{p'}(w^{-p'})}^{q'_0}.
\end{align*}
which together with duality yields the $L^p(w^p)$-boundedness of $\mathcal{R}_k$. This along with Theorem \ref{thm:lim-Tb} implies the $L^p(w^p)$-boundedness of $[\mathcal{R}_k, b]$.
\end{proof}


\end{document}